\documentclass[reqno]{amsart}
\usepackage{amssymb, hyperref}

\AtBeginDocument{
\vspace{9mm}}

\begin{document}
\title[
]
{SCATTERING FOR THE RADIAL 3D CUBIC FOCUSING INHOMOGENEOUS NONLINEAR SCHR\"ODINGER EQUATION}
\author[c.M.GUZMAN , L.G FARAH]
{LUIZ G. FARAH AND CARLOS M. GUZM\'AN}  

\address{LUIZ G. FARAH \hfill\break
Department of Mathematics, Federal University of Minas Gerais, BRAZIL}
\email{lgfarah@gmail.com}

\address{CARLOS M. GUZM\'AN \hfill\break
Department of Mathematics, Federal University of Minas Gerais, BRAZIL}
\email{carlos.guz.j@gmail.com}

%date{}
%thanks{Submitted July 2, 2005. Published October 26, 2005.}
%thanks{The first author is partly supported by
% grant from the Packard Foundation. The authors also thank Kenji Nakanishi for
%ointing out the connections between this paper and \cite{nakanishi}.}
%subjclass[2000]{35J10}
%keywords{Local well-posedness; uniform well-posedness; scattering theory;
%hfill\break\indent
 %trichartz estimates}

\begin{abstract}
The purpose of this work is to study the 3D focusing inhomogeneous nonlinear Schr\"odinger equation 
$$
i u_t +\Delta u+|x|^{-b}|u|^2 u = 0, 
$$
where $0<b<1/2$. Let $Q$ be the ground state solution of $-Q+\Delta Q+ |x|^{-b}|Q|^{2}Q=0$ and $s_c=(1+b)/2$. We show that if the radial initial data $u_0$ belongs to $H^1(\mathbb{R}^3)$ and satisfies $E(u_0)^{s_c}M(u_0)^{1-s_c}<E(Q)^{s_c}M(Q)^{1-s_c}$ and $\|  \nabla u_0 \|_{L^2}^{s_c} \|u_0\|_{L^2}^{1-s_c}<\|\nabla Q \|_{L^2}^{s_c} \|Q\|_{L^2}^{1-s_c}$, then the corresponding solution is global and scatters in $H^1(\mathbb{R}^3)$. Our proof is based in the ideas introduced by Kenig-Merle \cite{KENIG} in their study of the energy-critical NLS and Holmer-Roudenko \cite{HOLROU} for the radial 3D cubic NLS.
\end{abstract}

%%%%%%%%%%%%%%%%%%%%%%%%%%%%%%%%%%%%%%%%%%%%%%%%%%%%%%%%%%%%%%%%%%%%%%
%%%%%%%%%%%%%%%%%%%%%%%%%%%%%%%%%%%%%%%%%%%%%%%%%%%%%%%%%%%%%%%%%%%%%%

\maketitle
\numberwithin{equation}{section}
\newtheorem{theorem}{Theorem}[section]
\newtheorem{proposition}[theorem]{Proposition}
\newtheorem{lemma}[theorem]{Lemma}
\newtheorem{corollary}[theorem]{Corollary}
\newtheorem{remark}[theorem]{Remark}
\newtheorem{definition}[theorem]{Definition}
%%%%%%%%%%%%%%%%%%%%%%%%%%%%%%%%%%%%%%%%%%%%%%%%%%%%%%%%%%%%%%%%%%%%%
%%%%%%%%%%%%%%%%%%%%%%%%%%%%%%%%%%%%%%%%%%%%%%%%%%%%%%%%%%%%%%%%%%%%%

\section{Introduction}
In this paper, we consider the Cauchy problem, also called the initial value problem (IVP), for the focusing inhomogenous nonlinear Schr\"odinger (INLS) equation on $\mathbb{R}^3$, that is  
%on $\mathbb{R}^3$, with radial data in $H^1(\mathbb{R}^3)$,
\begin{equation}\label{INLS}
\left\{\begin{array}{cl}
i\partial_tu +\Delta u + |x|^{-b} |u|^2 u =0, & \;\;\;t\in \mathbb{R} ,\;x\in \mathbb{R}^3,\\
u(0,x)=u_0(x), &
\end{array}\right.
\end{equation}
where $u = u(t,x)$ is a complex-valued function in space-time  $\mathbb{R}\times\mathbb{R}^3$ and $0<b<1/2$.

\ Before review some results about the Cauchy problem \eqref{INLS}, let us recall the critical Sobolev index. For a fixed $\delta >0$, the rescaled function  $u_\delta(t,x)=\delta^{\frac{2-b}{2}}u(\delta^2 t,\delta x)$ is solution of \eqref{INLS} if only if $u(t,x)$ is a solution. This scaling property gives rise to a scale-invariant norm. Indeed, computing the homogeneus Sobolev norm of $u_\delta(0,x)$ we get
$$
\|u_\delta(0,.)\|_{\dot{H^s}}=\delta^{s-\frac{3}{2}+\frac{2-b}{2}}\|u_0\|_{\dot{H^s}}.
$$
Thus, the scale invariant Sobolev space is $H^{s_c}(\mathbb{R}^3)$, where $s_c=\frac{1+b}{2}$ (the critical Sobolev index). Note that, the restriction $0<b< 1/2$ implies $0<s_c<1$ and therefore we are in the mass-supercritical and energy-subcritical case. In addition, we recall that the INLS equation has the following conserved quantities
\begin{equation}\label{mass}
M[u_0]=M[u(t)]=\int_{\mathbb{R}^3}|u(t,x)|^2dx
\end{equation}
and
\begin{equation}\label{energy}
E[u_0]=E[u(t)]=\frac{1}{2}\int_{\mathbb{R}^3}| \nabla u(t,x)|^2dx-\frac{1}{4}\left\| |x|^{-b}|u|^{4}\right\|_{L^1_x},
\end{equation}
which are calling Mass and Energy, respectively.

\ Next, we briefly review recent developments on the well-posedness theory for the general INLS equation
\begin{equation}\label{GINLS}
\left\{\begin{array}{cl}
i\partial_tu +\Delta u + |x|^{-b} |u|^\alpha u =0, & \;\;\;x\in \mathbb{R}^N,\\
u(0,x)=u_0(x). &
\end{array}\right.
\end{equation}
Genoud and Stuart \cite{GENOUD}-\cite{GENSTU}, using the abstract theory developed by Cazenave \cite{CAZENAVEBOOK} and some sharp Gagliardo-Nirenberg inequalities, showed that \eqref{GINLS} is well-posed in $H^1(\mathbb{R}^N)$
      \begin{itemize}
  	\item locally if $0<\alpha <2^*,
  	$ 
  	\item globally for small initial condition if 
  	$\frac{4-2b}{N}<\alpha <\frac{4-2b}{N-2}$,
  	\item globally for any initial condition if 
  	$ 0 < \alpha < \frac{4-2b}{N}$,
  	\item globally if $\alpha = \frac{4-2b}{N}$, assuming $\|u_0\|_{L^2}<\|Q\|_{L^2}$,
 \end{itemize}
where $Q$ is the ground state of the equation $-Q+\Delta Q+ |x|^{-b}|Q|^{\frac{4-2b}{N}}Q=0$ and $2^*=\frac{4-2b}{N-2}$, if $N\geq 3$ or $2^* =\infty$, if $N=1,2$. Also, Combet and Genoud \cite{COMGEN} established the classification of minimal mass blow-up solutions of \eqref{GINLS} with $L^2$ critical nonlinearity, that is, $\alpha = \frac{4-2b}{N}$.   \\  
%\begin{equation}\label{GSE} % GSE=GROUND STATE EQUATION
%-Q+\Delta Q+ |x|^{-b}|Q|^{\alpha}Q=0.
%\end{equation}
Recently, the second author in \cite{CARLOS}, using the contraction mapping principle based on the Strichartz estimates, proved that the IVP \eqref{GINLS} is locally well-posed in $H^1(\mathbb{R}^N)$, for $0<\alpha<2^*$. Moreover, for $N\geq 2$, $\frac{4-2b}{N}<\alpha <2^*$ these solutions are global in $H^1(\mathbb{R}^N)$ for small initial data. It worth mentioning that Genoud and Stuart \cite{GENOUD}-\cite{GENSTU} consider $0<b<\min\{2,N\}$, and second author in \cite{CARLOS} assume $0<b<\widetilde{2}$, where $\widetilde{2}=N/3$ if $N=1,2,3$ and $\widetilde{2}=2$ if $N\geq 4$. This new restriction on $b$ is needed to estimate the nonlinear part of the equation in order to use the well known Strichartz estimates associated to the linear flow.

\ On the other hand, since 
\begin{equation}\label{QI1}
\|u_\delta\|_{L^2_x}=\delta^{-s_c}\|u\|_{L^2_x},\;\;\;\;\|\nabla u_\delta\|_{L^2_x}=\delta^{1-s_c}\|\nabla u\|_{L^2_x}
\end{equation}
and
$$
\left \| |x|^{-b}|u_\delta|^{4} \right\|_{L^1_x}=\delta^{2(1-s_c)}\left \| |x|^{-b}|u|^{4} \right\|_{L^1_x},
$$
the following quantities enjoy a scaling invariant property
\begin{equation}\label{QI2}
E[u_\delta]^{s_c}M[u_\delta]^{1-s_c}=E[u]^{s_c}M[u]^{1-s_c},\;\;\|\nabla u_\delta\|^{s_c}_{L^2_x}\|u_\delta\|^{1-s_c}_{L^2_x}=\|\nabla u\|^{s_c}_{L^2_x}\|u\|^{1-s_c}_{L^2_x}.
\end{equation}
These quantities were introduced in Holmer-Roudenko \cite{HOLROU} in the context of mass-supercritical and energy-subcritical nonlinear Schr\"odinger equation (NLS), which is equation \eqref{INLS} with $b=0$, and they were used to understand the dichotomy between blowup/global regularity. Indeed, in \cite{HOLROU}, the authors consider the $3D$ cubic NLS and proved that if the initial data  $u_0\in H^1(\mathbb{R}^3)$ is radial and satisfy 
\begin{equation}\label{ENLS}
E(u_0)M(u_0)<E(Q)M(Q)
\end{equation}
and  
\begin{equation}\label{MNLS}
\|  \nabla u_0 \|_{L^2} \|u_0\|_{L^2}<\|\nabla Q \|_{L^2} \|Q\|_{L^2},
\end{equation}
then the corresponding solution $u(t)$ of the Cauchy problem \eqref{INLS} (with $b=0$) is globally defined and scatters\footnote{Notice that, in this case the critical Sobolev index is $s_c=1/2$.} in $H^1(\mathbb{R}^3)$ where $Q$ is the ground state solution of the nonlinear elliptic equation $-Q+\Delta Q+|Q|^{2}Q=0$. The subsequent work Duyckaerts-Holmer-Roudenko \cite{DUCHOLROU} has removed the radial assumption on the initial data. In both papers, they used the method of the concentration-compactness and rigidity technique employed by Kenig-Merle \cite{KENIG} in their study of the energy critical NLS. Inspired by these works, we investigate same problem for the IVP \eqref{INLS}.\\

\begin{remark}
The results in Holmer-Roudenko \cite{HOLROU} and Duyckaerts-Holmer-Roudenko \cite{DUCHOLROU} have been generalized for the general NLS equation \eqref{GINLS} (with $b=0$) in the mass-supercritical and energy-subcritical case, by Fang-Xie-Cazenave \cite{JIANCAZENAVE} and Guevara \cite{GUEVARA}. Moreover, the recent works of Hong \cite{HONG} and Killip-Murphy-Visan-Zheng \cite{KMVZ2016} also obtained analogous result for the cubic focusing NLS equation perturbed by a potential. It's worth mentioning that global well-posedness and scattering for the mass critical and energy critical NLS has also received a lot of attention in the literature and we refer to Dodson \cite{D12}-\cite{D15}-\cite{D16}, Tao-Visan-Zhang \cite{TVZ07}, Killip-Tao-Visan \cite{KTV09}, Killip-Visan-Zhang \cite{KVZ08}, Colliander-Keel-Staffilani-Takaoka-Tao \cite{CKSTT08}, Ryckman-Visan \cite{RV07}, Visan \cite{V07} and Killip-Visan \cite{KV10} for the results in these directions. 

\end{remark}
\ In a recent work of the first author in \cite{LG} showed global well-posedness for the $L^2$-supercritical and $H^1$-subcritical inhomogeneous nonlinear Schr\"odinger equation \eqref{GINLS} under assumptions similar to \eqref{ENLS}-\eqref{MNLS}. Below we state his result for the $3D$ cubic INLS, since this is the case we are interested in the present work.

\begin{theorem}\label{TG}%    TG:theorem global
Let $0<b<1$. Suppose that $u(t)$ is the solution of \eqref{INLS} with initial data $u_0\in H^1(\mathbb{R}^3)$ satisfying
\begin{equation}\label{EMC} %EMC: ENERGY MASS CONDITION
E[u_0]^{s_c}M[u_0]^{1-s_c}<E[Q]^{s_c}M[Q]^{1-s_c} 
\end{equation}
and 
\begin{equation}\label{GFC}%GFC: GRADIENT FUNCTION CONDITION
\|  \nabla u_0 \|_{L^2}^{s_c} \|u_0\|_{L^2}^{1-s_c}<\|\nabla Q \|_{L^2}^{s_c} \|Q\|_{L^2}^{1-s_c},
\end{equation}
then $u(t)$ is a global solution in $H^1(\mathbb{R}^3)$. Furthermore, for any $t\in \mathbb{R}$ we have
\begin{equation}\label{GR}%GFC: Global result
\|  \nabla u(t) \|_{L^2}^{s_c} \|u(t)\|_{L^2}^{1-s_c}<\|\nabla Q \|_{L^2}^{s_c} \|Q\|_{L^2}^{1-s_c},
\end{equation}
where $Q$ is unique positive solution of the elliptic equation 
\begin{equation}\label{GSE} % GSE=GROUND STATE EQUATION
-Q+\Delta Q+ |x|^{-b}|Q|^{2}Q=0.
\end{equation}
\end{theorem}
\begin{remark} In \cite[Teorema $1.6$]{LG} the author also considers the case 
$$
\|  \nabla u_0 \|_{L^2}^{s_c} \|u_0\|_{L^2}^{1-s_c}>\|\nabla Q \|_{L^2}^{s_c} \|Q\|_{L^2}^{1-s_c}.
$$
Indeed assuming the last relation and \eqref{EMC} then the solution blows-up in finite time if the initial data $u_0$ has finite variance, i.e., $|x|u_0\in L^2(\mathbb{R}^3)$. This is the extension to the INLS model of the result proved by Holmer-Roudenko \cite{HOLROU2} for the NLS equation.
\end{remark}   

\ Our aim in this paper is to show that the global solutions obtained in Theorem \ref{TG} also scatters (in the radial case) according to the following definition

\begin{definition}\label{SCATTER}
A global solution $u(t)$ to the Cauchy problem \eqref{INLS} scatters forward in time
in $H^1(\mathbb{R}^3)$, if there exists  $\phi^+\in H^1(\mathbb{R}^3)$ such that
$$
\lim_{t\rightarrow +\infty}\|u(t)-U(t)\phi^+\|_{H^1}=0.
$$
Also, we say that  $u(t)$ scatters backward in time if there exist $\phi^-\in H^1(\mathbb{R}^3)$ such that 
$$
\lim_{t\rightarrow -\infty}\|u(t)-U(t)\phi^-\|_{H^1}=0.
$$
Here, $U(t)$ denotes unitary group associated to the linear equation $i\partial_tu +\Delta u=0$, with initial data $u_0$.
\end{definition}

The precise statement of our main theorem is the following.
 
 \begin{theorem}\label{SCATTERING}
 Let $u_0\in H^1(\mathbb{R}^3)$ be radial and $0<b<1/2$. Suppose that \eqref{EMC} and \eqref{GFC} are satisfied then the solution $u$ of \eqref{INLS} is global in $H^1(\mathbb{R}^3)$ and scatters both forward and backward in time. 
 \end{theorem} 
\begin{remark} The above theorem extends the result obtained by Holmer-Roudenko \cite{HOLROU} to the INLS model. On the other hand, since the solutions of the INLS equation do not enjoy conservation of Momentum, we were not able to use the same ideas introduced by Duyckaerts-Holmer-Roudenko \cite{DUCHOLROU} to remove the radial assumption.

\end{remark}
\ The plan of this work is as follows: in the next section we introduce some notations and estimates. In Section $3$, we sketch the proof of our main result (Theorem \ref{SCATTERING}), assuming all the technical points. In Section $4$, we collect some preliminary results about the Cauchy problem \eqref{INLS}. Next in Section $5$, we recall some properties of ground state and show the existence of wave operator. In Section $6$, we construct a critical solution denoted by $u_c$ and show some of its properties (the key ingredient in this step is a profile decomposition result related to the linear flow). Finally, Section $7$ is devoted to the rigidity theorem.

\section{Notation and preliminaries}

\ Let us start this section by introducing the notation used throughout the paper. We use $c$ to denote various constants that may vary line by line. Given any positive numbers $a$ and $b$, the notation $a\lesssim b$ means that there exists a positive constant $c$ that $a\leq cb$, with $c$ uniform with respect to the set where a and b vary. Let a set $A\subset \mathbb{R}^3$, $A^C=\mathbb{R}^N \backslash A$ denotes the complement of $A$. Given $x,y \in \mathbb{R}^3$, $x\cdot y$ denotes the inner product of $x$ and $y$ in $\mathbb{R}^3$. 
	
\ We use $\|.\|_{L^p}$ to denote the $L^p(\mathbb{R}^3)$ norm with $p\geq 1$. If necessary, we use subscript to inform with variable we are concerned with. The mixed norms in the spaces $L^q_tL^r_x$ and $L^q_TL^r_x$ of $f(x,t)$ are defined, respectively, as 
$$
\|f\|_{L^q_tL^r_x}=\left(\int_{\mathbb{R}}\|f(t,.)\|^q_{L^r_x}dt\right)^{\frac{1}{q}}
$$
and
$$
\|f\|_{L^q_TL^r_x}=\left(\int_T^\infty\|f(t,.)\|^q_{L^r_x}dt\right)^{\frac{1}{q}}
$$
with the usual modifications when $q=\infty$ or $r=\infty$.
	
\ For $s\in \mathbb{R}$, $J^s$ and $D^s$ denote the Bessel and the Riesz potentials of order $s$, given via Fourier transform by the formulas
	$$
	\widehat{J^s f}=(1+|y|^2)^{\frac{s}{2}}\widehat{f}\;\;\;\textnormal{and} \;\;\;\;\widehat{D^sf}=|y|^s\widehat{f},
	$$
	where the Fourier transform of $f(x)$ is given by
	$$
	\widehat{f}(y)=\int_{\mathbb{R}^3}e^{ix.y}f(x)dx.
	$$
On the other hand, we define the norm of the Sobolev spaces $H^{s,r}(\mathbb{R}^3)$ and $\dot{H}^{s,r}(\mathbb{R}^3)$, respectively, by
$$
\|f\|_{H^{s,r}}:=\|J^sf\|_{L^r}\;\;\;\;\textnormal{and}\;\;\;\;\|f\|_{\dot{H}^{s,r}}:=\|D^sf\|_{L^r}.
$$
If $r=2$ we denote $H^{s,2}=H^s$ and $\dot{H}^{s,2}=\dot{H}^s$.

\ Next, we recall some Strichartz type estimates associated to the linear Schr\"odinger propagator.\\
\textbf{Strichartz type estimates.} We say the pair $(q,r)$ is $L^2$-admissible or simply admissible par if they satisfy the condition
\begin{equation}\label{L2admissivel}
\frac{2}{q}=\frac{3}{2}-\frac{3}{r},
\end{equation}
where $2\leq  r  \leq 6$. 
We also called the pair $\dot{H}^s$-admissible if\footnote{It worth mentioning that, the pair $\left(\infty,\frac{6}{3-2s}\right)$ also satisfies the relation \eqref{Hsadmissivel}, however, in our work we will not make use of this pair when we estimate the nonlinearity $|x|^{-b} |u|^2 u$.}
\begin{equation}\label{Hsadmissivel}%cond par adm
\frac{2}{q}=\frac{3}{2}-\frac{3}{r}-s,
\end{equation}
where $\frac{6}{3-2s}\leq  r  \leq 6^{-}$. Here, $a^-$ is a fixed number slightly smaller than a ($a^-=a-\varepsilon$ with $\varepsilon>0$ small enough) and, in a similar way, we define $a^+$. Finally we say that $(q,r)$ is $\dot{H}^{-s}$-admissible if 
$$
\frac{2}{q}=\frac{3}{2}-\frac{3}{r}+s,
$$
where $\left(\frac{6}{3-2s}\right)^{+}\leq  r  \leq 6^{-}$.

\  Given $s\in \mathbb{R}$, we use the set $\mathcal{A}_s=\{(q,r);\; (q,r)\; \textnormal{is} \;\dot{H}^s\textnormal{-admissible}\}$ to define the Strichartz norm
$$
\|u\|_{S(\dot{H}^{s})}=\sup_{(q,r)\in \mathcal{A}_{s}}\|u\|_{L^q_tL^r_x}. 
$$
In the same way, the dual Strichartz norm is given by
$$
\|u\|_{S'(\dot{H}^{-s})}=\inf_{(q,r)\in \mathcal{A}_{-s}}\|u\|_{L^{q'}_tL^{r'}_x},
$$
where $(q',r')$ is such that $\frac{1}{q}+\frac{1}{q'}=1$ and $\frac{1}{r}+\frac{1}{r'}=1$ for $(q,r)\in \mathcal{A}_s$.

\ Note that, if $s=0$ then $\mathcal{A}_0$ is the set of all $L^2$-admissible pairs. Moreover, if $s=0$, $S(\dot{H}^0)=S(L^2)$ and $S'(\dot{H}^{0})=S'(L^2)$. We write $S(\dot{H}^s)$ or $S'(\dot{H}^{-s})$ if the mixed norm is evaluated over $\mathbb{R}\times\mathbb{R}^3$. To indicate a restriction to a time interval $I\subset (-\infty,\infty)$ and a subset $A$ of $\mathbb{R}^3$, we use the notations $S(\dot{H}^s(A);I)$ and $S'(\dot{H}^{-s}(A);I)$. 

\ The next lemmas provide some inequalities that will be useful in our work.  
	%\ We now recall some useful inequalities.
\begin{lemma}\label{ILE}  %Important linear estimate
If $t\neq 0$, $\frac{1}{p}+\frac{1}{p'}=1$ and $p'\in[1,2]$, then $U(t):L^{p'}(\mathbb{R}^3)\rightarrow L^p(\mathbb{R}^3)$ is continuous and 
$$
\|U(t)f\|_{L^p_x}\lesssim|t|^{-\frac{3}{2}\left(\frac{1}{p'}-\frac{1}{p}\right)}\|f\|_{L^{p'}}.
$$
\begin{proof} See Linares-Ponce \cite[Lemma $4.1$]{FELGUS}.
\end{proof}
 \end{lemma} 
 \begin{lemma}\textbf{(Sobolev embedding)}\label{SI} Let $1\leq p<+\infty$. If $s\in \left(0,\frac{3}{2}\right)$ then $H^{s}(\mathbb{R}^3)$ is continuously embedded in $L^r(\mathbb{R}^3)$ where $s=\frac{3}{p}-\frac{3}{r}$. Moreover, 
\begin{equation}\label{SEI} % SEI=sobol embedd ineq.
\|f\|_{L^r}\leq c\|D^sf\|_{L^{2}}.
\end{equation}
\begin{proof} See Linares-Ponce \cite[Theorem $3.3$]{FELGUS}. 
\end{proof}
\end{lemma}
\begin{remark}\label{SEI2} Using Lemma \ref{SI} we have that $H^s(\mathbb{R}^3)$ is continuously embedded in $L^r(\mathbb{R}^3)$ and 
\begin{equation}\label{SEI22} % SEI=sobol embedd ineq.
\|f\|_{L^r}\leq c\|f\|_{H^{s}},
\end{equation}    
where $r\in[2,\frac{6}{3-2s}]$.
\end{remark}
%\begin{lemma}\textbf{(Fractional product rule)}\label{PRODUCTRULE}
%Let $s\in (0, 1]$ and $1 < r, r_1, r_2, p_1, p_2 <+\infty$ are such that $\frac{1}{r}=\frac{1}{r_i}+\frac{1}{p_i}$ for $i=1,2$. Then,
%$$
%\| D^s (fg) \|_{L^r} \leq c \|f\|_{L^{r_1}}\|D^s g\|_{L^{p_1}}+c \|D^sf\|_{L^{r_2}}\| g\|_{L^{p_2}}.
%$$
%\begin{proof} See Kenig-Ponce-Vega \cite{KPV}.
%\end{proof}
%\end{lemma}
%\begin{lemma}\textbf{(Fractional chain rule)}\label{CHAINRULE}
%Suppose $G\in C^1(\mathbb{C})$, $s\in (0, 1]$, and $1 < r, r_1, r_2 <+\infty$ are such that $\frac{1}{r}=\frac{1}{r_1}+\frac{1}{r_2}$. Then,
%$$
%\| D^s G(u) \|_{L^r} \leq c \|G^{'}(u)\|_{L^{r_1}}\|D^s u\|_{L^{r_2}}.
%$$
%\begin{proof} See Kenig-Ponce-Vega \cite{KPV}.
%\end{proof}
%\end{lemma}
	
 \ Next we list the well-known Strichartz estimates we are going to use in this work. We refer the reader to Linares-Ponce \cite{FELGUS} and Kato \cite{KATO} for detailed proofs of what follows (see also Holmer-Roudenko \cite{HOLROU} and Guevara \cite{GUEVARA}).
\begin{lemma}\label{Lemma-Str} The following statements hold.
 \begin{itemize}
\item [(i)] (Linear estimates).
\begin{equation}\label{SE1}
\| U(t)f \|_{S(L^2)} \leq c\|f\|_{L^2},
\end{equation}
\begin{equation}\label{SE2}
\|  U(t)f \|_{S(\dot{H}^s)} \leq c\|f\|_{\dot{H}^s}.
\end{equation}
\item[(ii)] (Inhomogeneous estimates).
\begin{equation}\label{SE3}					 
\left \| \int_{\mathbb{R}} U(t-t')g(.,t') dt' \right\|_{S(L^2)}\;+\; \left \| \int_{0}^t U(t-t')g(.,t') dt' \right \|_{S(L^2) } \leq c\|g\|_{S'(L^2)},
\end{equation}
%\begin{equation*}\label{SE4}
%\left \| \int_{0}^t U(t-t')g(.,t') dt' \right\|_{S(\dot{H}^s) } \leq c\|D^sg\|_{S'(L^2)},
%\end{equation*}
\begin{equation}\label{SE5}
\left \| \int_{0}^t U(t-t')g(.,t') dt' \right \|_{S(\dot{H}^s) } \leq c\|g\|_{S'(\dot{H}^{-s})}.
\end{equation}
\end{itemize}
\end{lemma} 

\ We end this section with three important remarks. 	
\begin{remark}\label{nonlinerity}
Let $F(x,z)=|x|^{-b}|z|^2 z$, and $f(z)=|z|^2 z$. The complex derivative of $f$ is $f_z(z)=2|z|^2$ and $f_{\bar{z}}(z)=z^2$.
%\begin{equation*}\label{nonli1}
%f_z(z)=2|z|^2\;\;\;\;\;\textnormal{and}\;\;\;\; f_{\bar{z}}    %(z)=z^2. 
%\end{equation*}
For $z,w\in \mathbb{C}$, we have 
\begin{equation*}
 f(z)-f(w)=\int_{0}^{1}\left[f_z(w+\theta(z-w))(z-w)+f_{\bar{z}}(w+\theta(z-w))\overline{(z-w)}\right]d\theta.
\end{equation*}
Thus,
\begin{equation}\label{FEI}% first elementary inequality
 |F(x,z)-F(x,w)|\lesssim |x|^{-b}\left( |z|^2+ |w|^2 \right)|z-w|.
\end{equation}

\ Now we are interested in estimating $\nabla \left( F(x,z)-F(x,w) \right)$. A simple computation gives  
\begin{equation}\label{NONLI11}
\nabla F(x,z)=\nabla(|x|^{-b})f(z)+|x|^{-b} \nabla f(z)
\end{equation}
where $\nabla f(z)=f'(z)\nabla z=f_z(z)\nabla z+f_{\bar{z}}(z) \overline{\nabla z}$.\\
%\begin{equation*}\label{NONLI22}
%\nabla f(z)=f'(z)\nabla z=f_z(z)\nabla z+f_{\bar{z}}(z) \overline{\nabla z}.
%\end{equation*}
First we estimate $|\nabla (f(z)-f(w))|$. Note that
 \begin{equation}\label{NONLI55}
 \nabla (f(z)-f(w))=f'(z)(\nabla z-\nabla w)+(f'(z)-f'(w))\nabla w.
 \end{equation}
So, since  
\begin{equation*}\label{NONLI66}
|f_z(z)-f_z(w)|\;,\;| f_{\bar{z}}(z)-f_{\bar{z}}(w)|\lesssim (|z|+|w|)|z-w| 
\end{equation*}
%and 
%\begin{equation*}\label{NONLI77}
%| f_{\bar{z}}(z)-f_{\bar{z}}(w)|\lesssim (|z|+|w|)|z-w|,
%\end{equation*}
we get, by \eqref{NONLI55} 
\begin{equation*}\label{NONLI88}
|\nabla (f(z)-f(w))|\lesssim |z|^2|\nabla (z- w)|+(|z|+|w|)|\nabla w||z-w|.
\end{equation*}
Therefore, by \eqref{NONLI11}, \eqref{FEI} and the last two inequalities we obtain \vspace{0.1cm}
\begin{equation}\label{SECONDEI}
\left|\nabla \left(F(x,z)-F(x,w)\right)\right|\lesssim  |x|^{-b-1}(|z|^{2}+|w|^{2})|z-w|+|x|^{-b}|z|^2|\nabla (z- w)|+M, 
\end{equation}
where $M \lesssim  |x|^{-b}(|z|+|w|)|\nabla w||z-w|$. 
%\begin{eqnarray*}\label{NONLI100}
% M &\lesssim & |x|^{-b}(|z|+|w|)|\nabla w||z-w|. 
% \end{eqnarray*}
\end{remark}

\begin{remark}\label{RB} % RB=remark bola
Let $B=B(0,1)=\{ x\in \mathbb{R}^3;|x|\leq 1\}$ and $b>0$. If $x\in B^C$ then $|x|^{-b}<1$ and so
$$ 
\left	\||x|^{-b}f \right\|_{L^r_x}\leq \|f\|_{L_x^r(B^C)}+\left\||x|^{-b}f\right\|_{L_x^r(B)}. 
$$
\end{remark}

\ The next remark provides a condition for the integrability of $|x|^{-b}$ on $B$ and $B^C$. 
\begin{remark}\label{RIxb} %RIxb=remark integ of xb.
Note that if $\frac{3}{\gamma}-b>0$ then $\||x|^{-b}\|_{L^\gamma(B)}<+\infty$. Indeed 
\begin{equation*}
\int_{B}|x|^{-\gamma b}dx=c\int_{0}^{1}r^{-\gamma b}r^{2}dr=c_1 \left. r^{3-\gamma b} \right |_0^1<+\infty\;\;\textnormal{if}\;\;\frac{3}{\gamma} - b>0.
\end{equation*}
Similarly, we have that $\||x|^{-b}\|_{L^\gamma(B^C)}$ is finite if $\frac{3}{\gamma}- b<0$.
\end{remark}

\section{Sketch of the proof of Theorem \ref{SCATTERING}}\label{SPMR}
\ Similarly as in the NLS model, we have the following scattering criteria for global solution in $H^1(\mathbb{R}^3)$ (the proof will be given after Proposition \ref{GWPH1} below).
\begin{proposition}\label{SCATTERSH1}{\bf ($H^1$ scattering)} Let $0<b<1/2$. If $u(t)$ be  a global solution of \eqref{INLS} with initial data $u_0 \in  H^1(\mathbb{R}^3)$. If $\|u\|_{S(\dot{H}^{s_c})}< +\infty$ and $\sup\limits_{t\in \mathbb{R}}\|u(t)\|_{H^1_x}\leq B$, then $u(t)$ scatters in $H^1(\mathbb{R}^3)$ as $t \rightarrow \pm\infty$.
\end{proposition}

\ Let $u(t)$ be the corresponding $H^1$ solution for the Cauchy problem \eqref{INLS} with radial data $u_0\in H^1(\mathbb{R}^3)$ satisfying \eqref{EMC} and \eqref{GFC}. We already know by Theorem \ref{TG} that the solution is globally defined and $\sup\limits_{t\in \mathbb{R}}\|u(t)\|_{H^1}< \infty$. So, in view of Proposition \ref{SCATTERSH1}, our goal is to show that (recalling $s_c=\frac{1+b}{2}$)
\begin{equation}\label{HsFINITE} %HsFINITE=NORMA FINITA
\|u\|_{S(\dot{H}^{s_c})}<+\infty.
\end{equation}
The technique employed here to achieve the scattering property \eqref{HsFINITE} combines the concentration-compactness and rigidity ideas introduced by Kenig-Merle \cite{KENIG}. It is also based on the works of Holmer-Roudenko \cite{HOLROU} and Duyckaerts-Holmer-Roudenko \cite{DUCHOLROU}. We describe it in the sequel, but first we need some preliminary definitions.
\begin{definition}
We shall say that SC($u_0$) holds if the solution $u(t)$ with initial data $u_0\in H^1(\mathbb{R}^3)$ is global and \eqref{HsFINITE} holds.
\end{definition}
\begin{definition}
For each $\delta > 0$ define the set $A_\delta$ to be the collection of all initial data in $H^1(\mathbb{R}^3)$ satisfying
\begin{align*}
A_\delta=\{u_0\in H^1:E[u_0]^{s_c}M[u_0]^{1-s_c}<\delta\;\textnormal{and}\;\|\nabla u_0\|^{s_c}_{L^2}\| u_0\|^{1-s_c}_{L^2}<\|\nabla Q\|^{s_c}_{L^2}\| Q\|^{1-s_c}_{L^2} \}
\end{align*}
and define
\begin{equation}\label{deltac}
\delta_c=\sup \{\; \delta>0:\; u_0\; \in A_\delta\;  \Longrightarrow SC(u_0)\; \textnormal{holds} \}=\sup_{\delta>0} B_\delta.
\end{equation}
\end{definition}
Note that $B_\delta \neq \emptyset$. In fact, applying the Strichartz estimate \eqref{SE2}, interpolation and Lemma \ref{LGS} (i) below, we obtain
  \begin{eqnarray*}
  \|U(t)u_0\|_{S(\dot{H}^{s_c})}&\leq& c\|u_0\|_{\dot{H}^{s_c}}\leq c\|\nabla u_0\|^{s_c}_{L^2}\|u_0\|^{1-s_c}_{L^2}\\
  &\leq& c\left(\frac{3+b}{ s_c}\right)^{\frac{s_c}{2}}E[u_0]^{\frac{s_c}{2}}M[u_0]^{\frac{1-s_c}{2}}.
  \end{eqnarray*}
 So if $u_0\in A_\delta$ then $E[u_0]^{s_c}M[u_0]^{1-s_c}<\left(\frac{ s_c}{3+2b}\right)^{s_c}\delta'^2$,
  %$$
 % E[u_0]^{s_c}M[u_0]^{1-s_c}<\left(\frac{\alpha s_c}{6+2b}\right)^{s_c}\delta'^2,
 % $$
 which implies $\|U(t)u_0\|_{S(\dot{H}^{s_c})}\leq c\delta'$.
  %$$
  %\|U(t)u_0\|_{S(\dot{H}^{s_c})}\leq c\delta'.
  %$$
 Then, by the small data theory (Proposition \ref{GWPH1} below) we have that $SC(u_0)$ holds for $\delta'>0$ small enough. 
 
 \ Next, we sketch the proof of Theorem \ref{SCATTERING}. If $\delta_c\geq E[Q]^{s_c}M[Q]^{1-s_c}$ then we are done. Assume now, by contradiction, that $\delta_c<E[Q]^{s_c}M[Q]^{1-s_c}$. Therefore, there exists a sequence of radial solutions $u_n$ to \eqref{INLS} with $H^1$ initial data $u_{n,0}$ (rescale all of them to have $\|u_{n,0}\|_{L^2} = 1$ for all $n$) such that\footnote{We can rescale $u_{n,0}$ such that $\|u_{n,0}\|_{L^2} = 1$. Indeed, if $u^\lambda_{n,0}(x)=\lambda^{\frac{2-b}{2}}u_{n,0}(\lambda x)$ then by \eqref{QI2} we have $E[u^\lambda_{n,0}]^{s_c}M[u^\lambda_{n,0}]^{1-s_c}<E[Q]^{s_c}M[Q]^{1-s_c}$ and $\|\nabla u^\lambda_{n,0}\|^{s_c}_{L^2}\| u^\lambda_{n,0}\|^{1-s_c}_{L^2}<\|\nabla Q\|^{s_c}_{L^2}\| Q\|^{1-s_c}_{L^2} $. Moreover, since $\|u^\lambda_{n,0}\|_{L^2} = \lambda^{-s_c}\|u_{n,0}\|_{L^2}$ by \eqref{QI1}, setting $\lambda^{s_c}=\|u_{n,0}\|_{L^2}$ we have $\|u^\lambda_{n,0}\|_{L^2} = 1$.}   %\footnote{By \eqref{QI1} and \eqref{QI2} we can rescale $u_n$ such that $\|u_n\|_{L^2} = 1$ } 
\begin{equation}\label{CC0}   %   CC=critical condition
\|\nabla u_{n,0}\|^{s_c}_{L^2} <  \|\nabla Q\|^{s_c}_{L^2}\|Q\|^{1-s_c}_{L^2}
\end{equation}
and
\begin{equation*}
E[u_n]^{s_c} \searrow \delta_c\; \textnormal{as}\; n \rightarrow +\infty,
\end{equation*} 	
for which SC($u_{n,0}$) does not hold for any $n\in \mathbb{R}^3$. However, we already know by Theorem \ref{TG} that $u_n$ is globally defined. Hence, we must have $\|u_n\|_{S(\dot{H}^{s_c})}=+\infty$. Then using a profile decomposition result (see Proposition \ref{LPD} below) on the sequence $\{u_{n,0}\}_{n\in \mathbb{N}}$ we can construct a critical solution of \eqref{INLS}, denoted by $u_c$, that lies exactly at the threshold $\delta_c$, satisfies \eqref{CC0} (therefore $u_c$ is globally defined again by Theorem \ref{TG}) and $\|u_c\|_{S(\dot{H}^{s_c})}=+\infty$ (see Proposition \ref{ECS} below). On the other hand, we prove that the critical solution $u_c$ has the property that $K=\{u_c(t):t\in[0,+\infty)\}$ is precompact in $H^1(\mathbb{R}^3)$ (see Proposition \ref{PSC} below). Finally, the rigidity theorem (Theorem \ref{RT} below) will imply that such critical solution is identically zero, which contradicts the fact that $\|u_c\|_{S(\dot{H}^{s_c})}=+\infty$.

\section{Cauchy Problem}

\ In this section we show a miscellaneous of results for the Cauchy problem \eqref{INLS}. These results will be useful in the next sections. We start stating the following two lemmas. To this end, we use the following numbers
\begin{equation}\label{PA1}  %pares Hs admissiveis1
\widehat{q}=\frac{4(4-\theta)}{6+2b-\theta(1+b)},\;\;\;\widehat{r}\;=\;\frac{6(4-\theta)}{2(3-b)-\theta(2-b)},
\end{equation}
and
\begin{equation}\label{PA2}
\widetilde{a}\;=\;\frac{2(4-\theta)}{(7+2b-3\theta)-(2-b)(1-\theta)},\;\;\;  \widehat{a}=\frac{2(4-\theta)}{1-b}.
\end{equation}
It is easy to see that $(\widehat{q},\widehat{r})$ is $L^2$-admissible, $(\widehat{a},\widehat{r})$ is $\dot{H}^{s_c}$-admissible and $(\widetilde{a},\widehat{r})$ is $\dot{H}^{-s_c}$-admissible.

\begin{lemma}\label{LG1} %lema global 1 
 Let $0<b<1$, then there exist $c>0$ and a positive number $\theta<2$ such that
\begin{itemize}
\item [(i)] $\left \||x|^{-b}|u|^2 v \right\|_{S'(\dot{H}^{-s_c})} \leq c \| u\|^{\theta}_{L^\infty_tH^1_x}\|u\|^{2-\theta}_{S(\dot{H}^{s_c})} \|v\|_{S(\dot{H}^{s_c})}$,
\item [(ii)] $\left\||x|^{-b}|u|^2 v \right\|_{S'(L^2)}\leq c\| u\|^{\theta}_{L^\infty_tH^1_x}\|u\|^{2-\theta}_{S(\dot{H}^{s_c})} \| v\|_{S(L^2)}$.
\end{itemize} 
\begin{proof} % See Guzm\'an \cite[Lemma $4.1$ and Lemma $4.2$, with $(N,\alpha,s)=(3,2,1)$]{CARLOS}.
(i) We divide the estimate in $B$ and $B^C$, indeed
$$
\left \||x|^{-b}|u|^2 v \right\|_{S'(\dot{H}^{-s_c})} \leq \left \||x|^{-b}|u|^2 v \right\|_{S'\left(\dot{H}^{-s_c}(B)\right)} + \left \||x|^{-b}|u|^2 v \right\|_{S'\left(\dot{H}^{-s_c}(B^C)\right)}. 
$$

\ We first consider the estimate on $B$. By the H\"older inequality we deduce
\begin{eqnarray}\label{LG1Hs1}
\left \|  |x|^{-b}|u|^2 v \right\|_{L^{\widehat{r}'}_x(B)} &\leq& \||x|^{-b}\|_{L^\gamma(B)}  \|u\|^{\theta}_{L^{\theta r_1}_x}   \|u\|^{2-\theta}_{L_x^{(2-\theta)r_2}}  \|v\|_{L^{\widehat{r}}_x} \nonumber  \\
&=&\||x|^{-b}\|_{L^\gamma(B)}  \|u\|^{\theta}_{L^{\theta r_1}_x}   \|u\|^{2-\theta}_{L_x^{\widehat{r}}} \|v\|_{L^{\widehat{r}}_x},
\end{eqnarray}
where
\begin{equation}\label{LG1Hs2}
\frac{1}{\widehat{r}'}=\frac{1}{\gamma}+\frac{1}{r_1}+\frac{1}{r_2}+\frac{1}{\widehat{r}}\;\;\textnormal{and}\;\;\widehat{r}=(2-\theta)r_2.
\end{equation}
In order to have the norm $\||x|^{-b}\|_{L^\gamma(B)}$ bounded we need $\frac{3}{\gamma}>b$ (see Remark \ref{RIxb}). In fact, observe that \eqref{LG1Hs2} implies
\begin{equation*}
\frac{3}{\gamma}=3-\frac{3(4-\theta)}{\widehat{r}}-\frac{3}{r_1},
\end{equation*}
and from \eqref{PA1} it follows that
\begin{equation}\label{LG1Hs3}
\frac{3}{\gamma}-b=\frac{\theta(2-b)}{2}-\frac{3}{r_1}.
\end{equation}
Choosing $r_1>1$ such that $\theta r_1=6$ we obtain $\frac{3}{\gamma}-b=\theta(1-b)>0$ since $b<1$, that is, $|x|^{-b}\in L^\gamma (B)$. Moreover, using the Sobolev embedding \eqref{SEI22} (with $s=1$) and \eqref{LG1Hs1} we get  
\begin{equation}\label{LG1Hs4}
\left \|  |x|^{-b}|u|^2 v \right\|_{L^{\widehat{r}'}_x(B)} \leq c\|u\|^{\theta}_{H^1_x}  \|u\|^{2-\theta}_{L_x^{\widehat{r}}} \|v\|_{L^{\widehat{r}}_x}.
\end{equation}
On the other hand, we claim that
\begin{equation}\label{LG1Hs41}
\left \|  |x|^{-b}|u|^2 v \right\|_{L^{\widehat{r}'}_x(B^C)} \leq c\|u\|^{\theta}_{H^1_x}  \|u\|^{2-\theta}_{L_x^{\widehat{r}}} \|v\|_{L^{\widehat{r}}_x}.
\end{equation}
Indeed, Arguing in the same way as before we deduce
\begin{eqnarray*}
\left \|  |x|^{-b}|u|^2 v \right\|_{L^{\widehat{r}'}_x(B^C)} \leq \||x|^{-b}\|_{L^\gamma(B^C)}  \|u\|^{\theta}_{L^{\theta r_1}_x}   \|u\|^{2-\theta}_{L_x^{\widehat{r}}} \|v\|_{L^{\widehat{r}}_x},
\end{eqnarray*}
where the relation \eqref{LG1Hs3} holds. By Remark \ref{RIxb}, to show that $\||x|^{-b}\|_{L^\gamma(B^C)}$ is finite we need to verify that $\frac{3}{\gamma}-b<0$.  Indeed, choosing $r_1>1$ such that $\theta r_1=2$ and using $\eqref{LG1Hs3}$ we have $\frac{3}{\gamma}-b=-\frac{\theta(1+b)}{2}$, which is negative. 
%In the first case, we choose $r_1$ such that
%\begin{equation}\label{L1naBC} %bola complementar
% \theta r_1\in\left(2,\frac{N\alpha}{2-b}\right)
%\end{equation}
Therefore the Sobolev inequality \eqref{SEI22} implies \eqref{LG1Hs41}. This completes the proof of the claim.

\ Now, inequalities \eqref{LG1Hs4} and \eqref{LG1Hs41} yield
\begin{equation}\label{LG1Hs5}
\left \|  |x|^{-b}|u|^2 v \right\|_{L^{\widehat{r}'}_x} \leq c\|u\|^{\theta}_{H^1_x}  \|u\|^{2-\theta}_{L_x^{\widehat{r}}} \|v\|_{L^{\widehat{r}}_x}
\end{equation}
and the H\"older inequality in the time variable leads to
\begin{eqnarray*}
\left \|  |x|^{-b}|u|^2 v \right\|_{L_t^{\widetilde{a}'}L^{\widehat{r}'}_x}&\leq& c \|u\|^{\theta}_{L^\infty_tH^1_x} \|u\|^{2-\theta}_{ L_t^{(2-\theta)a_1} L_x^{\widehat{r}}} \|v\|_{L^{\widehat{a}}_tL^{\widehat{r}}_x} \nonumber  \\
&=& c \|u\|^{\theta}_{L^\infty_tH^1_x} \|u\|^{2-\theta}_{ L_t^{\widehat{a}} L_x^{\widehat{r}}} \|v\|_{L^{\widehat{a}}_tL^{\widehat{r}}_x},
\end{eqnarray*}
 where
 \begin{equation}\label{L4i}
\frac{1}{\widetilde{a}'}=\frac{2-\theta}{\widehat{a}}+\frac{1}{\widehat{a}}.
\end{equation}
 Since $\widehat{a}$ and $\widetilde{a}$ defined in \eqref{PA2} satisfy \eqref{L4i} we conclude the proof of item\footnote{Recall that $(\widehat{a},\widehat{r})$ is $\dot{H}^{s_c}$-admissible and $(\widetilde{a},\widehat{r})$ is $\dot{H}^{-s_c}$-admissible.} (i).
 
 \ (ii) In the previous item we already have \eqref{LG1Hs5}, then applying H\"older's inequality in the time variable we obtain
 \begin{equation}\label{LGHsii}
\left\|  |x|^{-b}|u|^2 v\right \|_{L_t^{\widehat{q}'}L^{\widehat{r}'}_x}\leq  c \|u\|^{\theta}_{L^\infty_tH^s_x}\|u\|^{2-\theta}_{L_t^{\widehat{a}} L_x^{\widehat{r}}} \|v\|_{L^{\widehat{q}}_tL^{\widehat{r}}_x},
\end{equation}
since
\begin{equation}\label{LG2Hs1}
\frac{1}{\widehat{q}'}=\frac{2-\theta}{\widehat{a}}+\frac{1}{\widehat{q}}
\end{equation}
 by \eqref{PA1} and \eqref{PA2}. The proof is finished since $(\widehat{q},\widehat{r})$ is $L^2$-admissible.

\end{proof} 
\end{lemma}

\begin{remark}\label{RGP} In the perturbation theory we use the following estimate
$$
\left \||x|^{-b}|u| v  w\right\|_{S'(L^2)} \leq c \| u\|^{\theta}_{L^\infty_tH^1_x}\|u\|^{1-\theta}_{S(\dot{H}^{s_c})} \|v\|_{S(\dot{H}^{s_c})}\|  w\|_{S(L^2)},
$$
where $\theta\in(0,1)$ is a sufficiently small number.
Its proof follows from the ideas of Lemma \ref{LG1} (ii), that is, we can repeat all the computations replacing $|u|^2 v$ by $|u|vw$ or, to be more precise, replacing $|u|^2 v=|u|^\theta |u|^{2-\theta }v$ by $|u| vw=|u|^\theta |u|^{1-\theta }vw$.
\end{remark}

\begin{lemma}\label{LG3} Let $0<b<1/2$. There exist $c>0$ and $\theta\in (0,2)$ sufficiently small such that
 \begin{equation*}
 \left\|\nabla (|x|^{-b}|u|^2 u)\right\|_{S'(L^2)}\leq c\| u\|^{\theta}_{L^\infty_tH^1_x}\|u\|^{2-\theta}_{S(\dot{H}^{s_c})} \|\nabla u\|_{S(L^2)}.
 \end{equation*}
 \begin{proof} Since $(2,6)$ is $L^2$-admissible in 3D and applying the product rule for derivatives we have
 \begin{eqnarray*}
\left\|\nabla\left(|x|^{-b}|u|^2 u \right)\right\|_{S'(L^2)}&\leq& \left\||x|^{-b}\nabla\left(|u|^2 u \right)\right\|_{S'(L^2)}+\left\|\nabla\left(|x|^{-b}\right)|u|^2 u \right\|_{S'(L^2)}\\
&\leq & \left\||x|^{-b}\nabla\left(|u|^2 u \right)\right\|_{L_t^{\widehat{q}'}L^{\widehat{r}'}_x} +  \left\|\nabla\left(|x|^{-b}\right)|u|^2 u \right\|_{L_t^{2'}L^{6'}_x}\\
&\leq & N_1+N2.
\end{eqnarray*}
 
\ First, we estimate $N_1$ (dividing in $B$ and $B^C$). It follows from H\"older's inequality that
\begin{eqnarray}\label{LG3Hs3}
\left\||x|^{-b}\nabla\left(|u|^2 u \right)\right\|_{L^{\widehat{r}'}_x(B)}  &\leq& \||x|^{-b}\|_{L^\gamma(B)}  \|u\|^{\theta}_{L^{\theta r_1}_x}   \|u\|^{2-\theta}_{L_x^{(2-\theta)r_2}}  \|\nabla u\|_{ L^{\widehat{r}}_x}\nonumber \\
&=&   \||x|^{-b}\|_{L^\gamma(B)}   \|u\|^{\theta}_{L^{\theta r_1}_x}  \|u\|^{2-\theta}_{L_x^{\widehat{r}}}  \|\nabla u\|_{ L^{\widehat{r}}_x},
\end{eqnarray}
where 
\begin{equation*}\label{LG3Hs4}
\frac{1}{\widehat{r}'}=\frac{1}{\gamma}+\frac{1}{r_1}+\frac{1}{r_2}+\frac{1}{\widehat{r}}\;\;\;\;\textnormal{and}\;\;\;\; \widehat{r}=(2-\theta)r_2.
\end{equation*}  
Notice that the right hand side of \eqref{LG3Hs3} is the same as the right hand side of \eqref{LG1Hs1}, with $v=\nabla u$. Thus, arguing in the same way as in Lemma \ref{LG1} (i) we obtain 
\begin{equation*}
\left \|  |x|^{-b}\nabla\left(|u|^2 u \right) \right\|_{L^{\widehat{r}'}_x(B)} \leq c\|u\|^{\theta}_{H^1_x}  \|u\|^{2-\theta}_{L_x^{\widehat{r}}} \|\nabla u\|_{L^{\widehat{r}}_x}.
\end{equation*}
We also obtain, by Lemma \ref{LG1} (i)
$$
\left \|  |x|^{-b}\nabla\left(|u|^2 u \right) \right\|_{L^{\widehat{r}'}_x(B^C)} \leq c\|u\|^{\theta}_{H^1_x}  \|u\|^{2-\theta}_{L_x^{\widehat{r}}} \|\nabla u\|_{L^{\widehat{r}}_x}.
$$
Moreover, the H\"older inequality in the time variable leads to (since $\frac{1}{\widetilde{q}'}=\frac{2-\theta}{\widehat{a}}+\frac{1}{\widehat{q}}$)
\begin{eqnarray}\label{SLG0}
N_1=\left \|  |x|^{-b}|u|^2 \nabla u \right\|_{L_t^{\widetilde{q}'}L^{\widehat{r}'}_x}&\leq& c \|u\|^{\theta}_{L^\infty_tH^1_x} \|u\|^{2-\theta}_{ L_t^{(2-\theta)a_1} L_x^{\widehat{r}}} \|\nabla u\|_{L^{\widehat{q}}_tL^{\widehat{r}}_x} \nonumber  \\
&=& c \|u\|^{\theta}_{L^\infty_tH^1_x} \|u\|^{2-\theta}_{ L_t^{\widehat{a}} L_x^{\widehat{r}}} \|\nabla u\|_{L^{\widehat{q}}_tL^{\widehat{r}}_x}.
\end{eqnarray}
 
\ To estimate $N_2$ we use the pairs $(\bar{a},\bar{r})=\left(8(1-\theta), \frac{12(1-\theta)}{3-2b-\theta(4-2b)}\right)$ $\dot{H}^{s_c}$-admissible and $(q,r)=\left(\frac{8(1-\theta)}{2-3\theta},\frac{12(1-\theta)}{4-3\theta}\right)$ $L^2$-admissible.\footnote{Note that $\frac{6}{2-b}=\frac{6}{3-2s_c}<\bar{r}<6$ (condition of $H^s$-admissible pair \eqref{Hsadmissivel}). Indeed, it is easy to check that $\bar{r}>\frac{6}{2-b}$. On the other hand, $\bar{r}<6\Leftrightarrow \theta(2-2b)<1-2b$, which is true by the assumption $b<1/2$ and $\theta>0$ is a small number. Moreover it is easy to see that $2<r<6$, i.e., $r$ satisfies the condition of admissible pair \eqref{L2admissivel}.}           . Let $A\subset \mathbb{R}^N$ such that $A=B$ or $A=B^C$. The H\"older inequality and the Sobolev embedding \eqref{SEI}, with $s=1$ imply
\begin{eqnarray}\label{SLG32}
\left\|  \nabla \left(|x|^{-b}\right)|u|^2 u \right \|_{L^{6'}_x(A)} & \leq& c\left\||x|^{-b-1} \right\|_{L^\gamma(A)} \|u\|^{\theta}_{L^{\theta r_1}_x}   \|u\|^{2-\theta}_{L_x^{(2-\theta)r_2}} \| u\|_{L_x^{ r_3}} \nonumber \\ 
&\leq & c  \left\||x|^{-b-1}\right\|_{L^\gamma(A)}  \|u\|^{\theta}_{L^{\theta r_1}_x} \| u\|^{2-\theta}_{L_x^{\bar{r}}}  \|\nabla u\|_{L_x^r},
\end{eqnarray}
where 
\begin{equation}\label{SLG3}
\frac{1}{6'}=\frac{1}{\gamma}+\frac{1}{r_1}+\frac{1}{r_2}+\frac{1}{r_3};\;\;\;\;1=\frac{3}{r}-\frac{3}{r_3};\;\;\;\bar{r}=(2-\theta)r_2. 
\end{equation}   
%From the H\"older inequality and the Sobolev embedding \eqref{SEI}, with $s=1$, one has
%\begin{eqnarray}\label{SLG32}
% M_2(t,A)& \leq& \||x|^{-b-1} \|_{L^d(A)} \|u\|^{\theta}_{L^{\theta r_1}_x}   \|u\|^{2-\theta}_{L_x^{(2-\theta)r_2}} \| u\|_{L_x^{ r_3}} \nonumber \\ 
%&\leq &   \||x|^{-b-1}\|_{L^d(A)}  \|u\|^{\theta}_{L^{\theta r_1}_x} \| u\|^{2-\theta}_{L_x^{\bar{r}}}  \|\nabla u\|_{L_x^r}
%\end{eqnarray}
%if
%\begin{equation*}
%\left\{\begin{array}{cl}
%\frac{1}{e}=&\frac{1}{r_1}+\frac{1}{r_2}+\frac{1}{r_3}\\ 
% 1=&\frac{3}{r}-\frac{3}{r_3}\\ 
%\bar{r}=&(\alpha-\theta)r_2.
%\end{array}\right.
%\end{equation*}
Note that the second equation in \eqref{SLG3} is valid since $r<3$. On the other hand, in order to show that $\||x|^{-b-1}\|_{L^d(A)}$ is bounded, we need $\frac{3}{d}-b-1>0$ when $A$ is the ball $B$ and $\frac{3}{d}-b-1<0$ when $A=B^C$, by Remark \ref{RIxb}. Indeed, using \eqref{SLG3} and the values of $q$, $r$, $\bar{q}$ and $\bar{r}$ defined above one has
\begin{eqnarray}\label{SLG33}
\frac{3}{\gamma}-b-1&=& \frac{5}{2}-b-\frac{3}{r_1}-\frac{3(2-\theta)}{\bar{r}}-\frac{3}{r}=\frac{\theta(2-b)}{2}-\frac{3}{r_1}.  
\end{eqnarray}
Now choosing $r_1$ such that 
$$
\theta r_1>\frac{6}{2-b} \textrm{ when } A=B \quad \textrm{and} \quad \theta r_1<\frac{6}{2-b} \textrm{ when } A=B^C,
$$
we get $\frac{3}{d}-b-1>0$ when $A=B$ and $\frac{3}{d}-b-1<0$ when $A=B^C$, so $|x|^{-b-1}\in L^d(A)$. In addition, we have by the Sobolev embedding \eqref{SEI22} (since $2<\frac{6}{2-b}<6$) and \eqref{SLG32}   
\begin{equation*}
 \left\|  \nabla \left(|x|^{-b}\right)|u|^2 u \right \|_{L^{6'}_x(A)} \leq c \|u\|^{\theta}_{H^1_x}\|u\|^{2-\theta}_{L^{\bar{r}}_x}   \|\nabla  u\|_{L_x^r}.
\end{equation*}
Finally, by H\"older's inequality in the time variable and the fact that $\frac{1}{2^{'}}=\frac{2-\theta}{\bar{a}}+\frac{1}{q}$,
%$$
%\frac{1}{2^{'}}=\frac{2-\theta}{\bar{a}}+\frac{1}{q},
%$$
we conclude 
 \begin{equation}\label{SLG35}
N_2= \left\|  \nabla \left(|x|^{-b}\right)|u|^2 u \right\|_{{L^{2'}_t L^{6'}_x}} \leq  c \|u\|^{\theta}_{L^\infty_tH^1_x}\|u\|^{2-\theta}_{L^{\bar{a}}_tL^{\bar{r}}_x}   \|\nabla u\|_{L^q_tL_x^r}.
\end{equation}	 
The proof is completed combining \eqref{SLG0} and \eqref{SLG35}.
\end{proof}
\end{lemma} 
\begin{remark} We notice that in Lemma \ref{LG1} and Remark \ref{RGP} we assume $0<b<1$. On the other hand, in Lemma \ref{LG3} the required assumption is $0<b<1/2$ (see footnote $5$). For this reason in our main result, Theorem \ref{SCATTERING}), the restriction on $b$ is different than the one in Theorem \ref{TG}.   
\end{remark}

\begin{remark}\label{RSglobal} A consequence of the previous lemma is the following estimate
$$
\left\| |x|^{-b-1}|u|^2 v  \right\|_{S'(L^2)}\lesssim  \| u\|^{\theta}_{L^\infty_tH^1_x}\|u\|^{2-\theta}_{S(\dot{H}^{s_c})}  \|\nabla v\|_{S(L^2)}.
$$
\end{remark} 

\ Our first result in this section concerning the IVP \eqref{INLS} is the following
\begin{proposition}\label{GWPH1}{\bf (Small data global theory in $H^1$)}
Let $0<b<1/2$ and $u_0 \in H^1(\mathbb{R}^3)$. Assume $\|u_0\|_{H^1}\leq A$. There there exists $\delta=\delta(A)>0$ such that if $\|U(t)u_0\|_{S(\dot{H}^{s_c})}<\delta$, then there exists a unique global solution $u$ of \eqref{INLS} such that
\begin{equation*}\label{NGWP3}
\|u\|_{S(\dot{H}^{s_c})}\leq 2\|U(t)u_0\|_{S(\dot{H}^{s_c})} 
\end{equation*}
and
\begin{equation*}\label{NGWP4}
\|u\|_{S\left(L^2\right)}+\|\nabla  u\|_{S\left(L^2\right)}\leq 2c\|u_0\|_{H^1}.
\end{equation*}
\begin{proof} %See Guzm\'an \cite[Corollary 1.11, with $(N,\alpha)=(3,2)$]{CARLOS}.
To this end, we use the contraction mapping principle. Define
$$
B=\{ u:\;\|u\|_{S(\dot{H}^{s_c})}\leq 2\|U(t)u_0\|_{S(\dot{H}^{s_c})}\;\textnormal{and}\;\|u\|_{S(L^2)}+\|\nabla u\|_{S(L^2)}\leq 2c\|u_0\|_{H^1}\}.
$$ We prove that $G$ defined below
\begin{equation*}\label{OPERATOR} % IE= integral equation
G(u)(t)=U(t)u_0+i \int_0^t U(t-t')F(x,u)(t')dt',
\end{equation*}
where $F(x,u)=|x|^{-b}|u|^2 u$ is a contraction on $B$ equipped with the metric 
$$
d(u,v)=\|u-v\|_{S(L^2)}+\|u-v\|_{S(\dot{H}^{s_c})}.
$$
		
\ Indeed, we deduce by the Strichartz inequalities (\ref{SE1}), (\ref{SE2}), \eqref{SE3} and \eqref{SE5}
\begin{equation}\label{GHs1}
\|G(u)\|_{S(\dot{H}^{s_c})}\leq \|U(t)u_0\|_{S(\dot{H}^{s_c})}+ c\| F \|_{S'(\dot{H}^{-s_c})}
\end{equation}
\begin{equation}\label{GHs2}
\|G(u)\|_{S(L^2)}\leq c\|u_0\|_{L^2}+ c\| F \|_{S'(L^2)}
\end{equation}
%and
\begin{equation}\label{GHs3}
 \|\nabla G(u)\|_{S(L^2)}\leq c \|\nabla  u_0\|_{L^2}+ c\|\nabla F\|_{S'(L^2)}.
\end{equation}
On the other hand, it follows from Lemmas \ref{LG1} and \ref{LG3} that
\begin{eqnarray*}
\|F\|_{S'(\dot{H}^{-s_c})}&\leq & c\| u \|^\theta_{L^\infty_tH^1_x}\| u \|^{2-\theta}_{S(\dot{H}^{s_c})}\| u \|_{S(\dot{H}^{s_c})}\\
\|F\|_{S'(L^2)}&\leq& c\| u \|^\theta_{L^\infty_tH^1_x}\| u \|^{2-\theta}_{S(\dot{H}^{s_c})}\| u \|_{S(L^2)}
\end{eqnarray*}
%and	
\begin{eqnarray*}
\|\nabla F\|_{S'(L^2)}&\leq& c\| u\|^{\theta}_{L^\infty_tH^1_x}\|u\|^{2-\theta}_{S(\dot{H}^{s_c})} \|\nabla u\|_{S(L^2)}.
\end{eqnarray*}
Combining \eqref{GHs1}-\eqref{GHs3} and the last inequalities, we get for $u\in B$
\begin{align*}\label{TGHS}
\|G(u)\|_{S(\dot{H}^{s_c})}\leq& \|U(t)u_0\|_{S(\dot{H}^{s_c})} +c\| u \|^\theta_{L^\infty_tH^1_x}\| u \|^{2-\theta}_{S(\dot{H}^{s_c})}\| u \|_{S(\dot{H}^{s_c})} \nonumber \\
\leq & \|U(t)u_0\|_{S(\dot{H}^{s_c})}+8c^{\theta+1}\|u_0\|^\theta_{H^1}\| U(t)u_0 \|^{3-\theta}_{S(\dot{H}^{s_c})}.
\end{align*}
In addition, setting $X=\|\nabla u\|_{S(L^2)}+\| u\|_{S(L^2)}$ then
\begin{eqnarray*}
\|G(u)\|_{S(L^2)}+\|\nabla G(u)\|_{S(L^2)}&\leq & c\|u_0\|_{H^1}+c\| u \|^\theta_{L^\infty_tH^1_x}\| u \|^{2-\theta}_{S(\dot{H}^{s_c})}X \nonumber\\
&\leq & c\|u_0\|_{H^1}+16c^{\theta+2}\|u_0\|_{H^1}^{\theta+1}\| U(t)u_0 \|^{2-\theta}_{S(\dot{H}^{s_c})},
\end{eqnarray*}
where we have have used the fact that $X\leq 2^2c\|u_0\|_{H^1}$ since $u\in B$.

\ Now if $\| U(t)u_0 \|_{S(\dot{H}^{s_c})}<\delta$ with 
\begin{equation}\label{WD}
\delta\leq \min\left\{\sqrt[2-\theta]{\frac{1}{16c^{\theta+1}A^\theta}}     , \sqrt[2-\theta]{ \frac{1}{32c^{\theta+1}A^\theta}}\right\},
\end{equation}
where $A>0$ is a number such that $\|u_0\|_{H^1}\leq A$, we get 
$$\|G(u)\|_{S(\dot{H}^{s_c})}\leq 2\|U(t)u_0\|_{S(\dot{H}^{s_c})}$$ and  $$\|G(u)\|_{S(L^2)}+\|\nabla G(u)\|_{S(L^2)}\leq 2c\|u_0\|_{H^1},$$ that is  $G(u)\in B$. The contraction property can be obtained by similar arguments. Therefore, by the Banach Fixed Point Theorem, $G$ has a unique fixed point $u\in B$, which is a global solution of \eqref{INLS}.

\end{proof}
\end{proposition}

\ We now show Proposition \ref{SCATTERSH1} (this result gives us the criterion to establish scattering).
	
\begin{proof}[\bf{Proof of Proposition \ref{SCATTERSH1}}] First, we claim that 
 \begin{equation}\label{SCATTER1}
 \|u\|_{S(L^2)}+\|\nabla u\|_{S(L^2)}<+\infty.
 \end{equation}
 
 \ Indeed, since $\|u\|_{S(\dot{H}^{s_c})}<+\infty$, given $\delta>0$ we can decompose $[0,\infty)$ into $n$ many intervals $I_j=[t_j,t_{j+1})$ such that $\|u\|_{S(\dot{H}^{s_c};I_j)}<\delta$ for all $j=1,...,n$. On the time interval $I_j$ we consider the integral equation
 \begin{equation*}\label{SCATTER2}
  u(t)=U(t-t_j)u(t_j)+i\int_{t_j}^{t_{j+1}}U(t-s)(|x|^{-b}|u|^2 u)(s)ds.
 \end{equation*}
 It follows from the Strichartz estimates \eqref{SE1} and \eqref{SE3} that
\begin{equation}\label{SCATTER3} %scattering 1
  \|u\|_{S(L^2;I_j)}\leq c\|u(t_j)\|_{L^2_x}+c\left\||x|^{-b}|u|^2 u \right\|_{S'(L^2;I_j)}	
\end{equation}	 
 %and
 \begin{equation}\label{SCATTER4}
  \|\nabla u\|_{S(L^2;I_j)}\leq c\|\nabla u(t_j)\|_{L^2_x}+c\left\|\nabla(|x|^{-b}|u|^2 u) \right\|_{S'(L^2;I_j)}.
 \end{equation}
From Lemmas \ref{LG1} (ii) and \ref{LG3} we have
\begin{eqnarray*}
 \left\||x|^{-b}|u|^2 u \right\|_{S'(L^2;I_j)}  &\leq& c \| u \|^\theta_{L^\infty_{I_j}H^1_x}\| u \|^{2-\theta}_{S(\dot{H}^{s_c};I_j)}\|u\|_{S(L^2;I_j)},
\end{eqnarray*}
 \begin{equation*}
\|\nabla(|x|^{-b}|u|^2 u)\|_{S'(L^2;I_j)}\leq c\| u\|^{\theta}_{L^\infty_{I_j}H^1_x}\|u\|^{2-\theta}_{S(\dot{H}^{s_c};I_j)}  \|\nabla u\|_{S(L^2;I_j)} .
\end{equation*}
 Thus, using \eqref{SCATTER3}, \eqref{SCATTER4} and the last two estimates we get	
$$
\|u\|_{S(L^2;I_j)}\leq c B+cB^\theta\delta^{2-\theta}\|u\|_{S(L^2;I_j)}
$$
and
\begin{equation}\label{nablau}
\|\nabla u\|_{S(L^2;I_j)}\leq c B+ cB^{\theta+1}\delta^{2-\theta}+cB^\theta\delta^{2-\theta}\|\nabla u\|_{S(L^2;I_j)},
\end{equation}
where we have used the assumption $\sup\limits_{t\in \mathbb{R}}\|u(t)\|_{H^1}\leq B$. Taking $\delta>0$ such that $c B^\theta\delta^{2-\theta}<\frac{1}{2}$ we obtain $\| u\|_{S(L^2;I_j)}+	\|\nabla u\|_{S(L^2;I_j)} \leq cB$,
%$$
%\| u\|_{S(L^2;I_j)}+	\|\nabla u\|_{S(L^2;I_j)} \leq cB,
%$$
and by summing over the $n$ intervals, we conclude the proof of \eqref{SCATTER1}.	
	
 \ Returning to the proof of the proposition, let
$$
\phi^+=u_0+i\int\limits_{0}^{+\infty}U(-s)|x|^{-b}(|u|^2 u)(s)ds,
$$
Note that, $\phi^+ \in H^1(\mathbb{R}^3)$. Indeed, by the same arguments as ones used before we deduce 
 \begin{equation*}
 \|\phi^+\|_{L^2}+\|\nabla\phi^+\|_{L^2}\leq c\|u_0\|_{H^1}+c \| u \|^\theta_{L^\infty_tH^1_x}\| u \|^{2-\theta}_{S(\dot{H}^{s_c})}\left(  \|u\|_{S(L^2)}+\| \nabla u\|_{S(L^2)} \right).
\end{equation*}	 
%and
%\begin{eqnarray*}
%\|\nabla\phi^+\|_{L^2}&\leq & c\|\nabla u_0\|_{L^2}+c\| u\|^{\theta}_{L^\infty_tH^1_x}\|u\|^{2-\theta}_{S(\dot{H}^{s_c})} \| \nabla u\|_{S(L^2)} .  
%\end{eqnarray*}
Therefore, \eqref{SCATTER1} yields $\|\phi\|_{H^1}<+\infty$.
  
\ On the other hand, since $u$ is a solution of \eqref{INLS} we get
$$
 u(t)-U(t)\phi^+=-i\int\limits_{t}^{+\infty}U(t-s)|x|^{-b}(|u|^2 u)(s)ds.
$$
Similarly as before, we have 
$$
 \|u(t)-U(t)\phi\|_{H^1_x}  \leq c \| u \|^\theta_{L^\infty_tH^1_x}\| u \|^{2-\theta}_{S(\dot{H}^{s_c};[t,\infty))} \left( \|u\|_{S(L^2)}+ \| \nabla u\|_{S(L^2)}  \right)
$$
%and
%\begin{eqnarray*}
%\|\nabla(u(t)-U(t)\phi)\|_{L^2_x}   &\leq & c  \| u\|^{\theta}_{L^\infty_tH^1_x}\|u\|^{2-\theta}_{S(\dot{H}^{s_c};[t,\infty))}  \| \nabla u\|_{S(L^2)}   
%\end{eqnarray*}
The proof is completed after using \eqref{SCATTER1} and  $\|u\|_{S(\dot{H}^{s_c};[t,\infty))}\rightarrow 0$ as $t \rightarrow +\infty$. 
\end{proof}	
\begin{remark} In the same way we define 
$$
\phi^-=u_0+i\int_0^{-\infty}U(-s)|x|^{-b}(|u|^2 u)(s)ds,
$$
and using the same argument as before we have $\phi^-\in H^1$ and 
%$$
% u(t)-U(t)\phi^-=i\int\limits_{-\infty}^{t}U(t-s)|x|^{-b}(|u|^2 u)(s)ds,
%$$
%which also satisfies (using the same argument as before) 
$$
\|u(t)-U(t)\phi^-\|_{H^1_x}\rightarrow 0 \,\,\textnormal{as}\,\,t\rightarrow -\infty.
$$
\end{remark}

\ Next, we study the perturbation theory for the IVP \eqref{INLS} following the exposition in Killip-Kwon-Shao-Visan \cite[Theorem $3.1$]{KKSV}. We first obtain a short-time perturbation which can be iterated to obtain a long-time perturbation result. 
\begin{proposition}\label{STP}{\bf (Short-time perturbation theory for the INLS)} Let $I\subseteq \mathbb{R}$ be a time interval containing zero and let $\widetilde{u}$ defined on $I\times \mathbb{R}^3$ be a solution (in the sense of the appropriated integral equation) to 
\begin{equation*}\label{PE}%perturb eq
i\partial_t \widetilde{u} +\Delta \widetilde{u} + |x|^{-b} |\widetilde{u}|^2 \widetilde{u} =e,
\end{equation*}  
with initial data $\widetilde{u}_0\in H^1(\mathbb{R}^3)$, satisfying 
\begin{equation}\label{PC11}  %pertub cond 1
\sup_{t\in I}  \|\widetilde{u}(t)\|_{H^1_x}\leq M \;\; \textnormal{and}\;\; \|\widetilde{u}\|_{S(\dot{H}^{s_c}; I)}\leq \varepsilon,
\end{equation}
for some positive constant $M$ and some small $\varepsilon>0$.
	
\indent  Let $u_0\in H^1(\mathbb{R}^3)$ such that 
\begin{equation}\label{PC22}
 \|u_0-\widetilde{u}_0\|_{H^1}\leq M'\;\; \textnormal{and}\;\; \|U(t)(u_0-\widetilde{u}_0)\|_{S(\dot{H}^{s_c}; I)}\leq \varepsilon,\;\;\textnormal{for }\; M'>0.
\end{equation}
In addition, assume the following conditions
\begin{equation}\label{PC33}
\|e\|_{S'(L^2; I)}+\|\nabla e\|_{S'(L^2; I)}+  \|e\|_{S'(\dot{H}^{-s_c}; I)}\leq \varepsilon.
\end{equation}
\indent There exists $\varepsilon_0(M,M')>0$ such that if $\varepsilon<\varepsilon_0$, then there is a unique solution $u$ to \eqref{INLS} on $I\times \mathbb{R}^3$ with initial data $u_0$, at the time $t=0$, satisfying 
 \begin{equation}\label{C} %conlusao 
 \|u\|_{S(\dot{H}^{s_c}; I)}\lesssim \varepsilon %\;\;\;\;\;\; \textnormal{and}
 \end{equation}
 and
  \begin{equation}\label{C1}
  \|u\|_{S(L^2; I)}+\|\nabla u\|_{S(L^2; I)}\lesssim c(M,M').
  \end{equation}
	\begin{proof} We use the following claim (we will show it later): there exists $\varepsilon_0>0$ sufficiently small such that, if $\|\widetilde{u}\|_{S(\dot{H}^{s_c};I)}\leq \varepsilon_0$ then 
  \begin{equation}\label{ST1} % short time 1
	\|\widetilde{u}\|_{S(L^2;I)} \lesssim M\;\;\;\textnormal{and}\;\;\;\;\|\nabla\widetilde{u}\|_{S(L^2;I)} \lesssim M.
  \end{equation}
   
   \ We may assume, without loss of generality, that $0=\inf I$. Let us first prove the existence of a solution $w$ for the 
    %$u=\widetilde{u}+w$, then $w$ solves the 
    following initial value problem  
	
  \begin{equation}\label{IVPP} %IVP perturbado
  \left\{\begin{array}{cl}
  i\partial_tw +\Delta w + H(x,\widetilde{u},w)+e= 0,&  \\
  w(0,x)= u_0(x)-\widetilde{u}_0(x),& 
  \end{array}\right.
  \end{equation}
  where $H(x,\widetilde{u},w)=|x|^{-b} \left(|\widetilde{u}+w|^2 (\widetilde{u}+w)-|\widetilde{u}|^2 \widetilde{u}\right)$.	
	
\ To this end, let  
\begin{equation}\label{IEP} % IE=integral equation perturbado
G (w)(t):=U(t)w_0+i \int_0^t U(t-s)(H(x,\widetilde{u},w)+e)(s)ds
\end{equation}
and define
$$
B_{\rho,K}=\{ w\in C(I;H^1(\mathbb{R}^3)):\;\|w\|_{S(\dot{H}^{s_c};I)}\leq \rho\;\textnormal{and}\;\|w\|_{S(L^2;I)}+\|\nabla w\|_{S(L^2;I)}\leq K    \}.
 $$
For a suitable choice of the parameters $\rho>0$ and $K>0$, we need to show that $G$ in \eqref{IEP} defines a contraction on $B_{\rho,K}$.
%under the metric given by
%$$
%d(w,z)=\|w-z\|_{S(L^2)}+\|w-z\|_{S(\dot{H}^{s_c})}.
%$$
Indeed, applying Strichartz inequalities (\ref{SE1}), (\ref{SE2}), (\ref{SE3}) and \eqref{SE5} we have
\begin{equation}\label{SP1}% short pertub 1
\|G(w)\|_{S(\dot{H}^{s_c};I)}\lesssim \|U(t)w_0\|_{S(\dot{H}^{s_c};I)}+ \| H(\cdot,\widetilde{u},w) \|_{S'(\dot{H}^{-s_c};I)}+\|e \|_{S'(\dot{H}^{-s_c};I)}
\end{equation}
\begin{equation}\label{SP2}
\|G(w)\|_{S(L^2;I)}\lesssim \|w_0\|_{L^2}+ \| H(\cdot,\widetilde{u},w) \|_{S'(L^2;I)}+\|e\|_{S'(L^2;I)}
\end{equation}
%and	
\begin{equation}\label{SP3}
\|\nabla G(w)\|_{S(L^2;I)}\lesssim  \|\nabla w_0\|_{L^2}+ \| \nabla H(\cdot,\widetilde{u},w)\|_{S'(L^2;I)}+\|\nabla e\|_{S'(L^2;I)}.
\end{equation}	
On the other hand, since 
\begin{equation}\label{EI} % elelmenta ineq
\left| |\widetilde{u}+w|^2(\widetilde{u}+w)-|\widetilde{u}|^2\widetilde{u}\right|\lesssim |\widetilde{u}|^2|w|+|w|^{3}
\end{equation}
by \eqref{FEI}, we get
$$
\|H(\cdot,\widetilde{u},w)\|_{S'(\dot{H}^{-s_c};I)}\leq \||x|^{-b}|\widetilde{u}|^2 w\|_{S'(\dot{H}^{-s_c};I)}+\||x|^{-b}|w|^2 w\|_{S'(\dot{H}^{-s_c};I)},
$$
which implies using Lemma \ref{LG1} (i) that
\begin{align}\label{SP4}
 \|H(\cdot,\widetilde{u},w)\|_{S'(\dot{H}^{-s_c};I)}\lesssim  \left(\| \widetilde{u} \|^\theta_{L^\infty_tH^1_x}\| \widetilde{u} \|^{2-\theta}_{S(\dot{H}^{s_c};I)}+ \| w \|^\theta_{L^\infty_tH^1_x} \| w\|^{2-\theta}_{S(\dot{H}^{s_c};I)}    \right) \| w \|_{S(\dot{H}^{s_c};I)}.
\end{align}
The same argument and Lemma \ref{LG1} (ii) also yield
\begin{align}\label{SP5}
\|H(\cdot,\widetilde{u},w)\|_{S'(L^2;I)}\lesssim \left(\| \widetilde{u} \|^\theta_{L^\infty_tH^1_x}\| \widetilde{u} \|^{2-\theta}_{S(\dot{H}^{s_c};I)} + \| w \|^\theta_{L^\infty_tH^1_x}\| w\|^{2-\theta}_{S(\dot{H}^{s_c};I)} \right)\| w \|_{S(L^2;I)}.
\end{align}
 Now, we estimate $\|\nabla H(\cdot,\widetilde{u},w)\|_{S'(L^2;I)}$. It follows from \eqref{SECONDEI} and \eqref{EI} that
 \begin{equation*} %\label{SP6}
 |\nabla H(x,\widetilde{u},w)| \lesssim |x|^{-b-1}(|\widetilde{u}|^{2}+|w|^{2})|w|+|x|^{-b}(|\widetilde{u}|^2+|w|^2) |\nabla w| +E,
 \end{equation*}
 where $E \lesssim 
 |x|^{-b}\left(|\widetilde{u}|+|w|\right)|w||\nabla \widetilde{u}|. 
$
%\begin{eqnarray*} %\label{NONLI10}
% E &\lesssim& 
% |x|^{-b}\left(|\widetilde{u}|+|w|\right)|w||\nabla \widetilde{u}| 
%\end{eqnarray*}
Thus, Lemma \ref{LG1} (ii), Remark \ref{RSglobal} and Remark \ref{RGP} lead to
$$
\|\nabla H(\cdot,\widetilde{u},w)\|_{S'(L^2;I)} \lesssim \left(\| \widetilde{u} \|^\theta_{L^\infty_tH^1_x}\| \widetilde{u} \|^{2-\theta}_{S(\dot{H}^{s_c};I)} + \| w \|^\theta_{L^\infty_tH^1_x}\| w\|^{2-\theta}_{S(\dot{H}^{s_c};I)} \right)\|\nabla w \|_{S(L^2;I)}
$$
\begin{align}\label{SP6}
\hspace{1.5cm}+\left(\| \widetilde{u} \|^\theta_{L^\infty_tH^1_x} \| \widetilde{u} \|^{1-\theta}_{S(\dot{H}^{s_c};I)} + \| w \|^\theta_{L^\infty_tH^1_x} \| w \|^{1-\theta}_{S(\dot{H}^{s_c};I)} \right) \| w \|_{S(\dot{H}^{s_c};I)}  \|\nabla \widetilde{u} \|_{S(L^2;I)}
\end{align}
%$$
%\hspace{2.0cm} +\left(\| \widetilde{u} \|^\theta_{L^\infty_tH^1_x}\| \widetilde{u} \|^{2-\theta}_{S(\dot{H}^{s_c};I)} + \| w \|^\theta_{L^\infty_tH^1_x}\| w\|^{2-\theta}_{S(\dot{H}^{s_c};I)} \right)\| w \|_{L^\infty_tH^1_x}
%$$
%$$
%\hspace{2.4cm}+\left(\| \widetilde{u} \|^\theta_{L^\infty_tH^1_x}\| \widetilde{u} \|^{2-\theta}_{S(\dot{H}^{s_c};I)} + \| w \|^\theta_{L^\infty_tH^1_x}\| w\|^{2-\theta}_{S(\dot{H}^{s_c};I)} \right)\|\nabla w \|_{S(L^2;I)} +E_1,
%$$
%$$
%\hspace{3.2cm}\lesssim \left(\| \widetilde{u} \|^\theta_{L^\infty_tH^1_x}\| \widetilde{u} \|^{2-\theta}_{S(\dot{H}^{s_c};I)} + \| w \|^\theta_{L^\infty_tH^1_x}\| w\|^{2-\theta}_{S(\dot{H}^{s_c};I)} \right)\|\nabla w \|_{S(L^2;I)}
%$$
%\begin{equation}\label{SP6}
%\hspace{2.8cm} +\left(\| \widetilde{u} \|^\theta_{L^\infty_tH^1_x}\| \widetilde{u} \|^{2-\theta}_{S(\dot{H}^{s_c};I)} + \| w \|^\theta_{L^\infty_tH^1_x}\| w\|^{2-\theta}_{S(\dot{H}^{s_c};I)} \right)\| w \|_{L^\infty_tH^1_x}+E_1.
%\end{equation}
%Moreover, using Remark \ref{RGP},
%\begin{align*}
%E_1\lesssim & 
%\left(\| \widetilde{u} \|^\theta_{L^\infty_tH^1_x} \| \widetilde{u} \|^{1-\theta}_{S(\dot{H}^{s_c};I)} + \| w \|^\theta_{L^\infty_tH^1_x} \| w \|^{1-\theta}_{S(\dot{H}^{s_c};I)} \right) \| w \|_{S(\dot{H}^{s_c};I)}  \|\nabla \widetilde{u} \|_{S(L^2;I)}
%\end{align*}
%where $\theta\in (0,1)$.
 
\ Hence, combining \eqref{SP4}, \eqref{SP5} and if $u\in B(\rho,K)$, we have
\begin{eqnarray}\label{SP7}
 \|H(\cdot,\widetilde{u},w)\|_{S'(\dot{H}^{-s_c};I)} \lesssim \left(M^\theta\varepsilon^{2-\theta}+K^\theta \rho^{2-\theta}\right)\rho
\end{eqnarray}
%and
\begin{eqnarray}\label{SP8}
\|H(\cdot,\widetilde{u},w)\|_{S'(L^2;I)}\lesssim \left(M^\theta\varepsilon^{2-\theta}+K^\theta \rho^{2-\theta}\right)K.
\end{eqnarray}
Furthermore, \eqref{SP6} and \eqref{ST1} imply
\begin{align}\label{SP9}
 \|\nabla H(\cdot,\widetilde{u},w)\|_{S'(L^2;I)} \lesssim \left(M^\theta\varepsilon^{2-\theta}+K^\theta \rho^{2-\theta}\right)K +\left( M^\theta \varepsilon^{1-\theta} + K^\theta \rho^{1-\theta} \right) \rho M. %E_1
\end{align}
%where $E_1 \lesssim  
%\left( M^\theta \varepsilon^{1-\theta} + K^\theta \rho^{1-\theta} \right) \rho M
%$.
%\begin{eqnarray*}
%E_1&\lesssim & 
%\left( M^\theta \varepsilon^{1-\theta} + K^\theta \rho^{1-\theta} \right) \rho M
%\end{eqnarray*}
Therefore, we deduce by \eqref{SP1}-\eqref{SP2} together with \eqref{SP7}- \eqref{SP8} that 
$$
\|G(w)\|_{S(\dot{H}^{s_c};I)}\leq  c\varepsilon+ cA\rho
$$
%and
$$
\|G(w)\|_{S(L^2;I)}\leq cM'+c\varepsilon +cAK,
 $$
 where we also used the hypothesis \eqref{PC22}-\eqref{PC33} and $A=M^\theta\varepsilon^{2-\theta}+K^\theta \rho^{2-\theta}$. We also have, using \eqref{SP3}, \eqref{SP9}
 \begin{equation*}
\|\nabla G(w)\|_{S(L^2;I)}\leq cM'+c\varepsilon +cAK+cB \rho M,
\end{equation*}
where $B= M^\theta \varepsilon^{1-\theta} + K^\theta \rho^{1-\theta}$.\\
Choosing $\rho=2c\varepsilon$, $K=3cM'$ and $\varepsilon_0$ sufficiently small such that 
$$
cA<\frac{1}{3}\;\;\;\;\textnormal{and}\;\;\;c(\varepsilon+B \rho M+K^\theta \rho^{2-\theta} M)<\frac{K}{3},
$$ 
we obtain
  \begin{equation*}
    	\|G(w)\|_{S(\dot{H}^{s_c};I)}\leq \rho\;\;\;\textnormal{and}\;\;\;\|G(w)\|_{S(L^2;I)}+\|\nabla G(w)\|_{S(L^2;I)}\leq K.
   \end{equation*}
The above calculations establish that $G$ is well defined on $B(\rho,K)$. The contraction property can be obtained by  similar arguments. Hence, by the Banach Fixed Point Theorem we obtain a unique solution $w$ on $I\times \mathbb{R}^N$ such that 
 $$
 \|w\|_{S(\dot{H}^{s_c};I)}\lesssim \varepsilon \;\;\;\textnormal{and}\;\;\;\|w\|_{S(L^2;I)}+\|w\|_{S(L^2;I)} \lesssim M'.
 $$ 
 Finally, it is easy to see that $u=\widetilde{u}+w$ is a solution to \eqref{INLS} satisfying \eqref{C} and \eqref{C1}. 
 
 \ To complete the proof we now show \eqref{ST1}. Indeed, we first show that 
 \begin{equation}\label{widetilde{u}}
 \|\nabla\widetilde{u}\|_{S(L^2;I)}\lesssim M.
 \end{equation}
 Using the same arguments as before, we have
  $$
  \|\nabla\widetilde{u}\|_{S(L^2;I)}\lesssim  \|\nabla\widetilde{u}_0\|_{L^2}+ \left\|\nabla(|x|^{-b}|\widetilde{u}|^2\widetilde{u})\right\|_{S'(L^2;I)}+\|\nabla e\|_{S'(L^2;I)}.
 $$
 Lemma \ref{LG3} implies
 \begin{eqnarray*}
  \|\nabla\widetilde{u}\|_{S(L^2;I)}&\lesssim& M+ \| \widetilde{u} \|^\theta_{L^\infty_tH^1_x}\| \widetilde{u} \|^{2-\theta}_{S(\dot{H}^{s_c};I)} \|\nabla \widetilde{u} \|_{S(L^2;I)} +\varepsilon\\
 &\lesssim& M+\varepsilon+  M^\theta \varepsilon_0^{2-\theta}\|\nabla \widetilde{u} \|_{S(L^2;I)}.
\end{eqnarray*}
Therefore, choosing $\varepsilon_0$ sufficiently small the linear term $M^\theta \varepsilon_0^{2-\theta}\|\nabla \widetilde{u} \|_{S(L^2;I)}$ may be absorbed by the left-hand term and we conclude the proof of \eqref{widetilde{u}}. Similar estimates also imply $\|\widetilde{u}\|_{S(L^2;I)}\lesssim M$. 
\end{proof}	
\end{proposition}
\begin{remark}\label{RSP} %remark de short perturb
From Proposition \ref{STP}, we also have the following estimates:
\begin{equation}\label{RSP1}
\|H(\cdot,\widetilde{u},w)\|_{S'(\dot{H}^{-s_c}; I)}\leq C(M,M') \varepsilon
\end{equation}
and
\begin{equation}\label{RSP2}
\|H(\cdot,\widetilde{u},w)\|_{S'(L^2; I)}+\|\nabla H(\cdot,\widetilde{u},w)\|_{S'(L^2; I)}\leq C(M,M')\varepsilon^{2-\theta},
\end{equation}
with $\theta>0$ small enough. Indeed, the relations \eqref{SP7}, \eqref{SP8} and \eqref{SP9} imply
\begin{eqnarray*}
\|H(\cdot,\widetilde{u},w)\|_{S'(\dot{H}^{-s_c}; I)} \lesssim \left(M^\theta\varepsilon^{2-\theta}+K^\theta \rho^{2-\theta}\right)\rho,
\end{eqnarray*}
\begin{eqnarray*}
\|H(\cdot,\widetilde{u},w)\|_{S'(L^2; I)}\lesssim \left(M^\theta\varepsilon^{2-\theta}+K^\theta \rho^{2-\theta}\right)K
 \end{eqnarray*}
and
\begin{align*}
\|\nabla H(\cdot,\widetilde{u},w)\|_{S'(L^2; I)} \lesssim \left(M^\theta\varepsilon^{2-\theta}+K^\theta \rho^{2-\theta}\right)K+\left( M^\theta \varepsilon^{1-\theta} + K^\theta \rho^{1-\theta} \right) \rho M.
\end{align*}
%where $E_1\lesssim 
%\left( M^\theta \varepsilon^{1-\theta} + K^\theta \rho^{1-\theta} \right) \rho M.
%$
%\begin{eqnarray*}
% E_1\lesssim & 
%\left( M^\theta \varepsilon^{1-\theta} + K^\theta \rho^{1-\theta} \right) \rho M.
%\end{eqnarray*}
Therefore, the choice $\rho=2c\varepsilon$ and $K=3cM'$ in Proposition \ref{STP} yield \eqref{RSP1} and \eqref{RSP2}. 
\end{remark}

\ In the sequel, we prove the long-time perturbation result.
\begin{proposition}\label{LTP}{\bf (Long-time perturbation theory for the INLS)} 
Let $I\subseteq \mathbb{R}$ be a time interval containing zero and let $\widetilde{u}$ defined on $I\times \mathbb{R}^3$ be a solution (in the sense of the appropriated integral equation) to 
\begin{equation*}
i\partial_t \widetilde{u} +\Delta \widetilde{u} + |x|^{-b} |\widetilde{u}|^2 \widetilde{u} =e,
\end{equation*}  
with initial data $\widetilde{u}_0\in H^1(\mathbb{R}^3)$, satisfying 
\begin{equation}\label{HLP1} %hypot de long perturb 1
\sup_{t\in I}  \|\widetilde{u}\|_{H^1_x}\leq M \;\; \textnormal{and}\;\; \|\widetilde{u}\|_{S(\dot{H}^{s_c}; I)}\leq L,
\end{equation}
for some positive constants $M,L$. 
    
\indent  Let $u_0\in H^1(\mathbb{R}^3)$ such that 
\begin{equation}\label{HLP2}
\|u_0-\widetilde{u}_0\|_{H^1}\leq M'\;\; \textnormal{and}\;\; \|U(t)(u_0-\widetilde{u}_0)\|_{S(\dot{H}^{s_c}; I)}\leq \varepsilon,
\end{equation}
	for some positive constant $M'$ and some $0<\varepsilon<\varepsilon_1=\varepsilon_1(M,M',L)$.
	Moreover, assume also the following conditions
	\begin{equation*}
	\|e\|_{S'(L^2; I)}+\|\nabla e\|_{S'(L^2; I)}+  \|e\|_{S'(\dot{H}^{-s_c}; I)}\leq \varepsilon.
	\end{equation*}
	\indent Then, there exists a unique solution $u$ to \eqref{INLS} on $I\times \mathbb{R}^3$ with initial data $u_0$ at the time $t=0$ satisfying 
	
\begin{equation}\label{CLP} %conclu long per
\|u-\widetilde{u}\|_{S(\dot{H}^{s_c}; I)}\leq C(M,M',L)\varepsilon\;\;\;\;\;\;\;\textnormal{and}
\end{equation}
\begin{equation}\label{CLP1}
\|u\|_{S(\dot{H}^{s_c}; I)} +\|u\|_{S(L^2; I)}+\|\nabla u\|_{S(L^2; I)}\leq C(M,M',L).
\end{equation}
\begin{proof} First observe that since $\|\widetilde{u}\|_{S(\dot{H}^{s_c}; I)}\leq L$, given\footnote{$\varepsilon_0$ is given by the previous result and $\varepsilon$ to be determined later.} $\varepsilon<\varepsilon_0(M,2M')$ we can partition $I$ into $n = n(L,\varepsilon)$ intervals $I_j = [t_j ,t_{j+1})$ such that for each $j$, the quantity $\|\widetilde{u}\|_{S(\dot{H}^{s_c};I_j)}\leq \varepsilon$. Note that $M'$ is  being replaced by $2M'$, as the $H^1$-norm of the difference of two different initial data may increase in each iteration.

\ Again, we may assume, without loss of generality, that $0=\inf I$. Let $w$ be defined by $u = \widetilde{u} + w$, then $w$ solves IVP \eqref{IVPP} with initial time $t_j$. Thus, the integral equation in the interval $I_j = [t_j ,t_{j+1})$ reads as follows
\begin{equation*}
  w(t)=U(t-t_j)w(t_j)+i\int_{t_j}^{t}U(t-s)(H(x,\widetilde{u},w)+e)(s)ds,
\end{equation*}
where $H(x,\widetilde{u},w)=|x|^{-b} \left(|\widetilde{u}+w|^2 (\widetilde{u}+w)-|\widetilde{u}|^2 \widetilde{u}\right)$.	
	
\ Thus, choosing $\varepsilon_1$ sufficiently small (depending on $n$, $M$, and $M'$), we may
	apply Proposition \ref{STP} (Short-time perturbation theory) to obtain for each $0\leq j<n$ and all $\varepsilon<\varepsilon_1$, 
	\begin{equation}\label{LP1}
	\|u-\widetilde{u}\|_{S(\dot{H}^{s_c};I_j)}\leq C(M,M',j)\varepsilon
	\end{equation}
	and
	\begin{equation}\label{LP2}
	\|w\|_{S(\dot{H}^{s_c};I_j)}+\|w\|_{S'(L^2;I_j)}+\|\nabla w\|_{S'(L^2;I_j)}\leq C(M,M',j)
	\end{equation}
 provided we can show
 \begin{equation}\label{LP3}
 \|U(t-t_j)(u(t_j)-\widetilde{u}(t_j))\|_{S(\dot{H}^{s_c};I_j)}\leq C(M,M',j)\varepsilon\leq \varepsilon_0
 \end{equation}
 and
 \begin{equation}\label{LP4}
 \|u(t_j)-\widetilde{u}(t_j)\|_{H^1_x}\leq 2M',
 \end{equation}
 For each $0\leq j<n$.
 
  \ Indeed, by the Strichartz estimates \eqref{SE2} and \eqref{SE5}, we have   
  \begin{eqnarray*}
  \|U(t-t_j)w(t_j)\|_{S(\dot{H}^{s_c};I_j)}&\lesssim& \|U(t)w_0\|_{S(\dot{H}^{s_c}; I)}+\|H(\cdot,\widetilde{u},w)\|_{S'(\dot{H}^{-s_c};[0,t_j])}\\
  &&+\|e\|_{S'(\dot{H}^{-s_c};I)},
  \end{eqnarray*}
which implies by \eqref{RSP1} that	
 $$
  \|U(t-t_j)(u(t_j)-\widetilde{u}(t_j))\|_{S(\dot{H}^{s_c}; I_j)}\lesssim \varepsilon+\sum_{k=0}^{j-1}C(k,M,M')\varepsilon.
 $$
  
 \ Similarly, it follows from Strichartz estimates \eqref{SE1}, \eqref{SE3} and \eqref{RSP2} that
 \begin{eqnarray*}
 \|u(t_j)-\widetilde{u}(t_j)\|_{H^1_x}&\lesssim & \|u_0-\widetilde{u}_0\|_{H^1}+\|e\|_{S'(L^2; I)}+\|\nabla e\|_{S'(L^2;I)}\\
 &&+\| H(\cdot,\widetilde{u},w)\|_{S'(L^2;[0,t_j])}+\|\nabla H(\cdot,\widetilde{u},w)\|_{S'(L^2;[0,t_j])}\\
 &\lesssim& M'+\varepsilon+\sum_{k=0}^{j-1}C(k,M,M')\varepsilon^{2-\theta}.
 \end{eqnarray*}
 Taking $\varepsilon_1=\varepsilon(n,M,M')$ sufficiently small, we see that \eqref{LP3} and \eqref{LP4} hold and so, it implies \eqref{LP1} and \eqref{LP2}.
 
 \ Finally, summing this over all subintervals $I_j$ we obtain 
 $$
 \|u-\widetilde{u}\|_{S(\dot{H}^{s_c}; I)}\leq C(M,M',L)\varepsilon
 $$
 and
 $$
  \|w\|_{S(\dot{H}^{s_c}; I)}+\|w\|_{S'(L^2; I)}+\|\nabla w\|_{S'(L^2; I)}\leq C(M,M',L).
 $$
 This completes the proof.
 \end{proof} 
\end{proposition}

\section[Ground state and wave operator]{Properties of the ground state, energy bounds and wave operator}%GSS=ground state solution

\ In this section, we recall some properties that are related to our problem. In \cite{LG} the first author proved the following Gagliardo-Nirenberg inequality
\begin{equation}\label{GNI} %GNI= GAGLIARDO NIRENBERG INEQUALITY
\left\||x|^{-b}|u|^{4} \right\|_{L^1_x}\leq C_{GN}\|\nabla u\|^{3+b}_{L^2_x}\|u\|^{1-b}_{L^2_x},
\end{equation}
with the sharp constant (recalling $s_c=\frac{1+b}{2}$)
\begin{equation}\label{GNI1}%GNE1: gagliardo nerinberg inequality1
C_{GN}=\frac{4}{3+b}\left(\frac{1-b}{3+b}\right)^{s_c}\frac{1}{\|Q\|^{2}_{L^2}}
\end{equation}
where $Q$ is the ground state solution of \eqref{GSE}. Moreover, $Q$ satisfies the following relations 
\begin{equation}\label{GS1}
\|\nabla Q\|^2_{L^2}=\frac{3+b}{1-b}\|Q\|^2_{L^2}
\end{equation}
and 
\begin{equation}\label{GS2}
\left\||x|^{-b}|Q|^{4} \right\|_{L^1}=\frac{4}{3+b}\|\nabla Q\|^2_{L^2}.
\end{equation}
Note that, combining \eqref{GNI1}, \eqref{GS1} and \eqref{GS2} one has
\begin{equation}\label{GNI2}%GNE: gagliardo nerinberg inequality2
C_{GN}=\frac{4}{(3+b)\|\nabla Q\|^{2 s_c}_{L^2}\|Q\|^{2(1-s_c)}_{L^2}}.
\end{equation}
On the other hand, we also have 
\begin{equation}\label{EGS} % energy ground state
E[Q]=\frac{1}{2}\|\nabla Q\|^2_{L^2}-\frac{1}{4}\left\||x|^{-b}|Q|^{4}\right\|_{L^1}=\frac{s_c}{3+b}\|\nabla Q\|^2_{L^2}.
\end{equation}

\ The next lemma provides some estimates that will be needed for the compactness and rigidity results. 

\begin{lemma}\label{LGS} %{\bf (Lower bound on the convexity of the variance).} 
Let $v \in H^1(\mathbb{R}^3)$ such that 
 \begin{equation}\label{LGS1}
 \|\nabla v\|^{s_c}_{L^2}\|v\|_{L^2}^{1-s_c}\leq \|\nabla Q\|^{s_c}_{L^2}\|Q\|_{L^2}^{1-s_c}.
\end{equation}	
Then, the following statements hold
\begin{itemize}
\item [(i)] $\frac{s_c}{3 +b}\|\nabla v\|^2_{L^2}\leq E(v)\leq \frac{1}{2} \|\nabla v\|^{2}_{L^2}$,
\item [(ii)] $\|\nabla v\|^{s_c}_{L^2}\|v\|^{1-s_c}_{L^2}\leq w^{\frac{1}{2}} \|\nabla Q\|^{s_c}_{L^2}\|Q\|^{1-s_c}_{L^2}$, 
\item [(iii)] $16A E[v]\leq 8A \|\nabla v\|_{L^2}^2\leq 8 \|\nabla v\|^2_{L^2}-2(3+b)\left\||x|^{-b}|v|^{4}\right\|_{L^1}$,
\end{itemize}
where $w=\frac{E[v]^{s_c}M[v]^{1-s_c}}{E[Q]^{s_c}M[Q]^{1-s_c}}$ and $A=(1-w)$.
\begin{proof}
 (i) The second inequality is immediate from the definition of Energy \eqref{energy}. The first one is obtained by observing that
\begin{eqnarray*}
E[v] &\geq& \frac{1}{2}\|\nabla v\|^2_{L^2} -  \frac{C_{GN}}{4}\|\nabla v\|^{3+b}_{L^2}\|v\|^{1-b}_{L^2}\\
&=&\frac{1}{2}\|\nabla v\|^2_{L^2}\left(1- \frac{C_{GN}}{2} \|\nabla v\|^{2 s_c}_{L^2}\|v\|^{2(1-s_c)}_{L^2}  \right)\\
&\geq& \frac{1}{2}\|\nabla v\|^2_{L^2}\left(1- \frac{C_{GN}}{2} \|\nabla Q\|^{2 s_c}_{L^2}\|Q\|^{2(1-s_c)}_{L^2}  \right)\\
&=&\frac{1+b}{2(3+b)}\|\nabla v\|^2_{L^2}=\frac{ s_c}{3+b}\|\nabla v\|^2_{L^2},
\end{eqnarray*}	  
where we have used \eqref{GNI}, \eqref{GNI2} and \eqref{LGS1}.
	
 (ii) The first inequality in (i) yields $\|\nabla v\|^2_{L^2}\leq \frac{3 +b}{ s_c}E(v)$, multiplying it by $M[v]^\sigma=\|v\|_{L^2}^{2\sigma}$, where $\sigma=\frac{1-s_c}{s_c}$, we have
\begin{eqnarray*}
\|\nabla v\|^2_{L^2}\| v\|^{2\sigma}_{L^2}&\leq& \frac{3 +b}{s_c}E[v]M[v]^\sigma\\
 & = & \frac{3 +b}{ s_c}\frac{E[v]M[v]^\sigma}{E[Q]M[Q]^\sigma} E[Q]M[Q]^\sigma\\
&=& w \|\nabla Q\|^2\|Q\|^{2\sigma}_{L^2},
\end{eqnarray*}
where we have used \eqref{EGS}.
		
(iii) The first inequality obviously holds. Next, let $B=8 \|\nabla v\|^2_{L^2}-2(3+b)\left\||x|^{-b}|v|^{4}\right\|_{L^1}$. Applying the Gagliardo-Niremberg inequality \eqref{GNI} and item (ii) we deduce  
\begin{eqnarray*}
B&\geq& 8\|\nabla v\|^2_{L^2}- 2(3+b)C_{GN}\|\nabla v\|^{3+b}_{L^2}\|v\|^{1-b}_{L^2}   \\      %( \|\nabla v\|^{s_c}_{L^2}\|v\|^{1-s_c}_{L^2})^\alpha  \\
&\geq& \|\nabla v\|^2_{L^2}\left(8- 2(3+b)C_{GN}w \|\nabla Q\|^{2 s_c}_{L^2}\|Q\|^{2(1-s_c)}_{L^2}  \right)\\
&=& \|\nabla v\|^2_{L^2}8(1-w),
\end{eqnarray*} 
where in the last equality, we have used \eqref{GNI2}.	
\end{proof}	
\end{lemma}

\ Now, applying the ideas introduced by C\^ote \cite{COTE} for the KdV equation (see also Guevara \cite{GUEVARA} Proposition $2.18$, with $(N,\alpha)=(3,2)$), we show the existence of the Wave Operator. Before stating our result, we prove the following lemma.

\begin{lemma}\label{LEWO}%lema exist de wav opera.
Let $0<b<1$. If $f$ and $g\in H^1(\mathbb{R}^3)$ then 
 \begin{itemize}
\item [(i)] $\left\|  |x|^{-b} |f|^{3}g \right\|_{L^1} 
 \leq c \|f\|^{3}_{L^{4}} \|g\|_{L^{4}} + c\|f\|^{3}_{L^{r}}\|g\|_{L^{r}}$
\item [(ii)] $\left\|  |x|^{-b} |f|^{3}g \right\|_{L^1} 
 \leq c \|f\|^{3}_{{H^1}} \|g\|_{H^1}$
\item [(iii)] $ \lim\limits_{|t|\rightarrow +\infty}  \left\|  |x|^{-b} |U(t) f|^{3}g \right\|_{L^1_x}=0.$ 
 \end{itemize}
where $\frac{12}{3-b}<r<6$.
% $r^*=\frac{2N}{N-2}$ if $N \geq 3$ and $r^*=2(\alpha+2)$ if $N=2$. 
\begin{proof} (i) We divide the estimate in $B^C$ and $B$. Applying the H\"older inequality, since $1=\frac{3}{4}+\frac{1}{4}$, one has
 \begin{eqnarray}\label{LEWO1}
 \left\|  |x|^{-b} |f|^{3}g \right\|_{L^1} 
 &\leq &\left\|  |x|^{-b} |f|^{3}g \right\|_{L^1(B^C)} 
  +\left\|  |x|^{-b} |f|^{3}g \right\|_{L^1(B)} \nonumber\\
 &\leq& \|f\|^{3}_{L^{4}}\|g\|_{L^{4}}+\|x|^{-b}|\|_{L^\gamma(B)}\|f\|^{3}_{L^{3\beta}}\|g\|_{L^r} \nonumber\\
 &=& \|f\|^{3}_{L^{4}}\|g\|_{L^{4}}+\|x|^{-b}|\|_{L^\gamma(B)} \|f\|^{3}_{L^{r}} \|g\|_{L^r},
\end{eqnarray}
where
 \begin{equation}\label{LEWO2}
1=\frac{1}{\gamma}+\frac{1}{\beta}+\frac{1}{r}\;\;\;\;\;\textnormal{and}\;\;\;\;\;\;r=3\beta.
\end{equation}
 To complete the proof we need to check that $\||x|^{-b}\|_{L^\gamma(B)}$ is bounded, i.e., $\frac{3}{\gamma}>b$ (see Remark \ref{RIxb}). In fact, we deduce from \eqref{LEWO2}
 $$
 \frac{3}{\gamma}=3-\frac{12}{r},
 $$
and thus, since $r>\frac{12}{3-b}$ we obtain the desired result ($\frac{3}{\gamma}-b>0$). 

\ (ii) By the Sobolev inequality \eqref{SEI22}, it is easy to see that $H^1 \hookrightarrow  L^{4} $ and $H^1 \hookrightarrow L^{r}$ (where $2<\frac{12}{3-b}<r<6$), then using \eqref{LEWO1} we get (ii).

\ (iii) Similarly as (i) and (ii), we get
 \begin{eqnarray}\label{lematecnico}
 \left\|  |x|^{-b} |U(t) f|^{3}g \right\|_{L^1_x}\leq c\|U(t)f\|^{\alpha+1}_{L^{4}} \|g\|_{H^1}+ c\|U(t)f\|^{3}_{L^{r}}\|g\|_{H^1},
\end{eqnarray}
for $\frac{12}{3-b}<r<6$. We now show that $\|U(t)f\|_{L^{r}_x}$ and $\|U(t)f\|_{L^{4}_x}$ $\rightarrow 0$ as 
$|t|\rightarrow +\infty$. Indeed, since $r$ and $4$ belong to $(2,6)$ then it suffices to show 
\begin{equation}\label{LEWO3}
\lim\limits_{|t|\rightarrow +\infty}\|U(t)f\|_{L^p_x}=0, 
\end{equation}
where $2<p<6$. Let $\widetilde{f}\in H^1\cap L^{p'}$, the Sobolev embedding \eqref{SEI22} and Lemma \ref{ILE} yield
 $$
 \|U(t)f\|_{L^{p}_x}\leq c \|f-\widetilde{f}\|_{{H}^1}+c|t|^{-\frac{3(p-2)}{2p}}\|\widetilde{f}\|_{L^{p'}}.
 $$
Since $p>2$ then the exponent of $|t|$ is negative and so approximating $f$ by $\widetilde{f}\in C^\infty_0$ in $H^1$, we deduce \eqref{LEWO3}.
\end{proof}
\end{lemma}

\begin{proposition}{\bf (Existence of Wave Operator)}\label{PEWO} Suppose $\phi \in H^1(\mathbb{R}^3)$ and, for some\footnote{Note that $(\frac{2 s_c}{3+b})^{\frac{s_c}{2}}<1$.} $0<\lambda \leq (\frac{2 s_c}{3+b})^{\frac{s_c}{2}}$, 
\begin{equation}\label{HEWO}
\|\nabla \phi\|^{2s_c}_{L^2}\|\phi\|^{2(1-s_c)}_{L^2}<\lambda^2\left(\frac{3+b}{ s_c}\right)^{s_c}E[Q]^{s_c}M[Q]^{1-s_c}.
\end{equation}
Then, there exists $u^+_0\in H^1(\mathbb{R}^3)$ such that $u$ solving \eqref{INLS} with initial data $u^+_0$ is global in $H^1(\mathbb{R}^3)$ with
\begin{itemize}
\item [(i)]$M[u]=M[\phi]$,
\item [(ii)] $E[u]=\frac{1}{2}\|\nabla \phi\|^2_{L^2}$,
\item [(iii)] $\lim\limits_{t\rightarrow +\infty} \|u(t)-U(t)\phi\|_{H^1_x}=0$,
\item [(iv)] $\|\nabla u(t)\|^{s_c}_{L^2}\| u(t)\|^{1-s_c}_{L^2}\leq \lambda\|\nabla Q\|^{s_c}_{L^2}\|Q\|^{1-s_c}_{L^2}$.
\end{itemize}
\end{proposition}

\begin{proof} We will divide the proof in two parts. First, we construct the wave operator for large time. Indeed, let $I_T=[T,+\infty)$ for $T\gg 1$ and define  
\begin{equation*}\label{IEWO1} % IE=integral equation wave oper 1
G(w)(t)=- i \int_t^{+\infty} U(t-s)(|x|^{-b}|w+U(t)\phi|^2 (w+U(t)\phi)(s)ds,\;\;t\in I_T
\end{equation*}
and
$$
B(T,\rho)=\{w\in C\left(I_T;H^1(\mathbb{R}^3)\right): \|w\|_{T}\leq \rho   \},
$$
where
$$
\|w\|_T=\|w\|_{S(\dot{H}^{s_c};I_T)}+\|w\|_{S(L^2;I_T)}+\|\nabla w\|_{S(L^2;I_T)}. 
$$ 
Our goal is to find a fixed point for $G$ on $B(T,\rho)$. 
  
\ Applying the Strichartz estimates \eqref{SE3} \eqref{SE5} and Lemmas \ref{LG1}-\ref{LG3}, we deduce  
\begin{align}\label{EWO1}    % esisten wave opera 1
\| G(w) \|_{S(\dot{H}^{s_c};I_T)} \lesssim &  \| w+U(t)\phi\|^\theta_{L^\infty_{T}H^1_x}\| w+U(t)\phi \|^{2-\theta}_{S(\dot{H}^{s_c};I_T)}\| w+U(t)\phi \|_{S(\dot{H}^{s_c};I_T)}
\end{align}
\begin{align}\label{EWO2}
\| G(w) \|_{S(L^2;I_T)} \lesssim &  \| w+U(t)\phi\|^\theta_{L^\infty_{T}H^1_x}\| w+U(t)\phi \|^{2-\theta}_{S(\dot{H}^{s_c};I_T)}\| w+U(t)\phi \|_{S(L^2;I_T)}
\end{align}
and
\begin{align}\label{EWO3}
\|\nabla G(w) \|_{S(L^2;I_T)} \lesssim &  \| w+U(t)\phi\|^\theta_{L^\infty_{T}H^1_x}\| w+U(t)\phi \|^{2-\theta}_{S(\dot{H}^{s_c};I_T)}\| \nabla(w+U(t)\phi) \|_{S(L^2;I_T)}  
\end{align}	  
Thus,
\begin{eqnarray*}\label{EWO4}
\|G(w)\|_{T} &\lesssim &  \| w+U(t)\phi\|^\theta_{L^\infty_{T}H^1_x}\| w+U(t)\phi \|^{2-\theta}_{S(\dot{H}^{s_c};I_T)}\| w+U(t)\phi \|_{T}. \nonumber  \\
\end{eqnarray*}
   
\ Since\footnote{Note that \eqref{U(t)phi} is possible not true using the norm $L^{\infty}_{I_T}L^{\frac{6}{3-2s_c}}_x$ and for this reason we remove the pair $\left(\infty,\frac{6}{3-2s_c}\right)$ in the definition of $\dot{H}^{s}$-admissible pair.} 
\begin{equation}\label{U(t)phi}
\| U(t)\phi \|_{S(\dot{H}^{s_c};I_T)}\rightarrow 0
\end{equation}
as $T\rightarrow +\infty$, we can find $T_0>0$ large enough and $\rho>0$ small enough such that $G$ is well defined on $B(T_0,\rho)$. The same computations show that $G$ is a contraction on $B(T_0,\rho)$. Therefore, $G$ has a unique fixed point, which we denote by $w$. 
 
\ On the other hand, from \eqref{EWO1} and since 
$$
\| w+U(t)\phi\|_{L^\infty_{T}H^1_x}\leq \|w\|_{H^1} +\|\phi\|_{H^1}<+\infty,
$$ 
one has (recalling $G(w)=w$)
\begin{eqnarray*}    % esisten wave opera 1
\| w \|_{S(\dot{H}^{s_c};I_T)} &\lesssim &  \| w+U(t)\phi \|^{2-\theta}_{S(\dot{H}^{s_c};I_T)}\| w+U(t)\phi \|_{S(\dot{H}^{s_c};I_T)}\\
 &\lesssim &  A\| w\|_{S(\dot{H}^{s_c};I_T)} +A\|U(t)\phi \|_{S(\dot{H}^{s_c};I_T)}
\end{eqnarray*}
where $A=\| w+U(t)\phi \|^{2-\theta}_{S(\dot{H}^{s_c};I_T)}$. In addition, if $\rho$ has been chosen small enough and since $\|U(t)\phi\|_{S(\dot{H}^{s_c};I_T)}$ is also sufficiently small for $T$ large, we deduce
$$
A\leq c\| w\|^{2-\theta}_{S(\dot{H}^{s_c};I_T)}+c\| U(t)\phi \|^{2-\theta}_{S(\dot{H}^{s_c};I_T)}<\frac{1}{2},
$$
and so (using the last two inequalities)
$$
\frac{1}{2} \| w \|_{S(\dot{H}^{s_c};I_T)} \lesssim A \|U(t)\phi \|_{S(\dot{H}^{s_c};I_T)},
$$
which implies,
\begin{equation}\label{EWO5}
  \| w \|_{S(\dot{H}^{s_c};I_T)}\rightarrow 0\;\;\;\;\textnormal{as}\;\;\;\; T\rightarrow +\infty.
\end{equation} 
 Hence, \eqref{EWO2}, \eqref{EWO3} and $\eqref{EWO5}$ also yield that\footnote{Observe that $\| w+U(t)\phi \|_{S(\dot{H}^{s_c};I_T)}\leq \| w \|_{S(\dot{H}^{s_c};I_T)}+\| U(t)\phi \|_{S(\dot{H}^{s_c};I_T)} \rightarrow 0$ as $T\rightarrow +\infty$ by \eqref{EWO5} and $\| w+U(t)\phi\|^\theta_{L^\infty_{T}H^1_x}, \| w+U(t)\phi \|_{S(L^2;I_T)}, \| \nabla(w+U(t)\phi) \|_{S(L^2;I_T)} <\infty$ since $w\in B(T,\rho)$ and $\phi \in H^1(\mathbb{R}^3)$.}
$$
 \| w \|_{S(L^{2};I_T)}\;,\,\|\nabla
 w \|_{S(L^{2};I_T)}\rightarrow 0\;\;\;\;\textnormal{as}\;\;\;\; T\rightarrow +\infty,
 $$
and finally
\begin{equation}\label{EWO6}
   \|w\|_{T}\rightarrow 0 \;\; \textnormal{as}\;\; T\rightarrow +\infty.  
 \end{equation}
   
  \ Next, we claim that $u(t)=U(t)\phi+w(t)$ satisfies \eqref{INLS} in the time interval $[T_0,\infty)$. To do this, we need to show that
 \begin{equation}\label{CWO} %calim de wave operat
 u(t)=U(t-T_0)u(T_0)+i\int_{T_0}^{t}U(t-s)(|x|^{-b}|u|^2 u)(s)ds,
 \end{equation}
 for all $t\in [T_0,\infty)$. Indeed, since
  	$$
  	w(t)=- i \int_t^\infty U(t-s)|x|^{-b}|w+U(t)\phi|^2 (w+U(t)\phi)(s)ds,
  $$
then 
  \begin{eqnarray*}
   	U(T_0-t)w(t)&=&- i \int_t^\infty U(T_0-s)|x|^{-b}|w+U(t)\phi|^2 (w+U(t)\phi)(s)ds\\
   	&=& i\int_{T_0}^t U(T_0-s)|x|^{-b}|w+U(t)\phi|^2 (w+U(t)\phi)(s)ds+w(T_0),
  \end{eqnarray*}
  and so applying $U(t-T_0)$ on both sides, we get
  $$
  w(t)=U(t-T_0)w(T_0)+ i \int_{T_0}^t U(t-s)|x|^{-b}|w+U(t)\phi|^2 (w+U(t)\phi)(s)ds.
  $$
  Finally, adding $U(t)\phi$ in both sides of the last equation, we deduce \eqref{CWO}.
  
 \ Now we show relations (i)-(iv). Since $u(t)=U(t)\phi+w$ then  
 \begin{align}\label{EWO7}
 \|u(t)-U(t)\phi\|_{L^\infty_TH^1_x}=\|w\|_{L^\infty_TH^1_x}\leq c\|w\|_{S(L^2;I_T)}+c\|\nabla w\|_{S(L^2;I_T)}\leq c\|w\|_{T}
  \end{align}
 and so from $\eqref{EWO2}$ we obtain (iii). Furthermore, using \eqref{EWO7} it is clear that
 \begin{equation}\label{EWO81}
 \lim_{t\rightarrow\infty}\|u(t)\|_{L^2_x}=\| \phi\|_{L^2}.
 \end{equation}
 and
 \begin{equation}\label{EWO8}
 \lim_{t\rightarrow\infty}\|\nabla u(t)\|_{L^2_x}=\|\nabla \phi\|_{L^2}.
 \end{equation}
By the mass conservation \eqref{mass} we have $\|u(t)\|_{L^2}=\|u(T_0)\|_{L^2}$ for all $t$, so from \eqref{EWO81} we deduce $\|u(T_0)\|_{L^2}=\|\phi\|_{L^2}$, i.e., item (i) holds. On the other hand, it follows from Lemma \ref{LEWO} (ii) 
 \begin{eqnarray*}
\left\| |x|^{-b}|u(t)|^{4} \right\|_{L^1_x}&\leq& c\left\| |x|^{-b}|u(t)-U(t)\phi|^{4} \right\|_{L^1_x}+c\left\| |x|^{-b}|U(t)\phi|^{4} \right\|_{L^1_x}\\
&\leq & c\left\|u(t)-U(t)\phi| \right\|^{4}_{H^1_x}+c\left\| |x|^{-b}|U(t)\phi|^{4} \right\|_{L^1_x},
\end{eqnarray*}
which goes to zero as $t\rightarrow +\infty$, by item (iii) and Lemma \ref{LEWO} (iii), i.e.
\begin{equation}\label{EWO9}
 \lim_{t\rightarrow\infty}\left\| |x|^{-b}|u(t)|^{4} \right\|_{L^1_x}=0.
 \end{equation} 
Combining \eqref{EWO8} and \eqref{EWO9}, it is easy to deduce (ii).
 
 \ Next, in view of \eqref{HEWO}, (i) and (ii) we have
 $$
 E[u]^{s_c}M[u]^{1-s_c}=\frac{1}{2^{s_c}}\|\nabla \phi\|^{2s_c}_{L^2}\|\phi\|^{2(1-s_c)}_{L^2}<\lambda^2\left(\frac{3+b}{2 s_c}\right)^{s_c}E[Q]^{s_c}M[Q]^{1-s_c}
  $$
 and by our choice of $\lambda$ we conclude
 \begin{equation*}\label{EWO10}
 E[u]^{s_c}M[u]^{1-s_c}<E[Q]^{s_c}M[Q]^{1-s_c}.
 \end{equation*}
 Moreover, from \eqref{EWO81}, \eqref{EWO8} and \eqref{HEWO}
\begin{eqnarray*}
\lim_{t\rightarrow \infty}\|\nabla u(t)\|^{2s_c}_{L^2_x}\|u(t)\|^{2(1-s_c)}_{L^2_x}&=&\|\nabla \phi\|^{2s_c}_{L^2}\|\phi\|^{2(1-s_c)}_{L^2}\\
&<& \lambda^2\left(\frac{3+b}{ s_c}\right)^{s_c}E[Q]^{s_c}M[Q]^{1-s_c}\\
&=&\lambda^2\|\nabla Q\|^{2s_c}_{L^2}\|Q\|^{2(1-s_c)}_{L^2},
\end{eqnarray*}
where we have used \eqref{EGS}. Thus, one can take $T_1>0$ sufficiently large such that
\begin{equation*}\label{EWO11}
\|\nabla u(T_1)\|^{s_c}_{L^2_x}\|u(T_1)\|^{1-s_c}_{L^2_x}<\lambda \|\nabla Q\|^{s_c}_{L^2}\|Q\|^{1-s_c}_{L^2}.
\end{equation*}
Therefore, since $\lambda<1$, we deduce that relations \eqref{EMC} and \eqref{GFC} hold with $u_0=u(T_1)$ and so, by Theorem \ref{TG}, we have in fact that $u(t)$ constructed above is a global solution of \eqref{INLS}.  %we can envolve $u(t)$ from back to the initial time $0$. 
\end{proof}		
\begin{remark}\label{backward}
A similar Wave Operator construction also holds when the time limit is taken as $t\rightarrow -\infty$ (backward in time).
\end{remark}

%%%%%%%%%%%%%%%%%%%%%%%%%%%%%%%%%%%%%%%%%%%%%%%%%%%%%%%%%%%%%%%%%%%%%%%%%%%%%%%%%%%%%%%%%%%%%%%%%%%%%%%%%%%%%5
%   WAVE OPERATORS
%%%%%%%%%%%%%%%%%%%%%%%%%%%%%%%%%%%%%%%%%%%%%%%%%%%%%%%%%%%%%%%%%%%%%%%%%%%%%%%%%%%%%%%%%%%%%%%%%%%%%%%%%%%%%5

%%%%%%%%%%%%%%%%%%%%%%%%%%%%%%%%%%%%%%%%%%%%%%%%%%%%%%%%%%%%%%%%%%%%%%%%%%%%%%%%%%%%%%%%%%%%%%%%%%%%%%%%%%%%%%%%%%%%%%%%

%%%%%%%%%%%%%%%%%%%%%%%%%%%%%%%%%%%%%%%%%%%%%%%%%%%%%%%%%%%%%%%%%%%%%%%%%%%%%%%%%%%%%%%%%%%%%%%%%%%%%%%%%%%%%%%%%%%%%%%%

%%%%%%%%%%%%%%%%%%%%%%%%%%%%%%%%%%%%%%%%%%%%%%%%%%%%%%%%%%%%%%%%%%%%%%%%%%%%%%%%%%%%%%%%%%%%
%%%%%%%%%%%%%%%%%%%%%%%%%%%%%%%%%%%%%%%%%%%%%%%%%%%%%%%%%%%%%%%%%%%%%%%%%%%%%%%%%%%%%%%%%%%%

\section{Existence and compactness of a critical solution}

\ The goal of this section is to construct a critical solution (denoted by $u_c$) of \eqref{INLS}. We divide the study in two parts, first we establish a profile decomposition result and also an Energy Pythagorean expansion for such decomposition. In the sequel, using the results of the first part we construct $u_c$ and discuss some of its properties.

\ We start this section recalling some elementary inequalities (see G\'erard \cite{Ge98} inequality (1.10) and Guevara \cite{GUEVARA} page 217). Let $(z_j)\subset\mathbb{C}^M$ with $M\geq 2$. For all $q>1$ there exists $C_{q,M}>0$ such that 
\begin{equation}\label{FI}% first inequality
 \left|\;\left | \sum_{j=1}^M z_j \right|^q-\sum_{j=1}^M|z_j|^q \right|  \leq C_{q,M}\sum_{j\neq k}^{M} |z_j| |z_k|^{q-1},
\end{equation}  
and for $\beta>0$ there exists a constant $C_{\beta,M}>0$ such that 
%for all $(z_j)_{1\leq j\leq M}\subset \mathbb{C}^M$,
\begin{equation}\label{EIerror} %elemet ineq error
\left| \left|\sum_{j=1}^{M}z_j\right|^\beta\sum_{j=1}^{M}z_j-\sum_{j=1}^{M}  |z_j|^\beta z_j\right|\leq C_{\beta,M}\sum_{j=1}^{M}\sum_{1\leq j\neq k\leq M}|z_j|^\beta|z_k|.
\end{equation} 

%%%%%%%%%%%%%%%%%%%%%%%%%%%%%%%%%%%%%%%%%%%%%%%%%%%%%%%%%%%%%%%%%%%%%%%%%%%%%%%%%%%%%%%%%%%%%%%%%%%%%%%%%%%%%%%%%%%%%%%%%%%%%%%%%%%%%%%%%%%%%555555555555555
% PROFILE DECOMPOSITION
%%%%%%%%%%%%%%%%%%%%%%%%%%%%%%%%%%%%%%%%%%%%%%%%%%%%%%%%%%%%%%%%%%%%%%%%%%%%%%%%%%%%%%%%%%%%%%%%%%%%%%%%%%%%%%%%%%%%%%%%%%%%%%%%%%%%%%%%%%%%%%%5555555555555

\subsection{Profile expansion}

This subsection contains a profile decomposition and  an energy Pythagorean expansion results. We use similar arguments as the ones in Holmer-Roudenko \cite[Lemma $5.2$]{HOLROU} (see also Fang-Xie-Cazenave \cite[Theorem 5.1]{JIANCAZENAVE}, with $(N,\alpha)=(3,2)$) and, for the sake of completeness, we provide the details here. 
\begin{proposition}\label{LPD} {\bf (Profile decomposition)}% LPD=LINEAR PROF DECOMP
	Let $\phi_n(x)$ be a radial uniformly bounded sequence in $H^1(\mathbb{R}^3)$. Then for each $M\in \mathbb{N}$ there exists a  subsequence of $\phi_n$ (also denoted by $\phi_n$), such that, for each $1\leq j\leq M$, there exist a profile $\psi^j$ in $H^1(\mathbb{R}^3)$, a sequence $t_n^j$ of time shifts and a sequence $W_n^M$ of remainders in $H^1(\mathbb{R}^3)$, such that	
	\begin{equation}\label{Aproximation}
	\phi_n(x)=\sum_{j=1}^{M}U(-t_n^j)\psi^j(x)+W_n^M(x)
	\end{equation}
	with the properties: 	  	
 \begin{itemize}
  \item \textsl{Pairwise divergence} for the time sequences. For $1\leq k\neq j\leq M$,
 \begin{equation}\label{PD} %PD pairse diverg
  \lim\limits_{n \rightarrow +\infty}|t_n^j-t_n^k|=+\infty.
  \end{equation}  % AS=asymp small
  \item \textsl{Asymptotic smallness} for the remainder sequence\footnote{Recalling that $s_c=\frac{1+b}{2}$.} 
		
\begin{equation}\label{AS}
\lim\limits_{M \rightarrow +\infty}\left(\lim\limits_{n \rightarrow +\infty}\|U(t)W_n^M\|_{S(\dot{H}^{s_c})}\right)=0.
\end{equation}
		
\item \textsl{Asymptotic Pythagoream expansion}. For fixed $M\in \mathbb{N}$ and any $s\in [0,1]$, we have
		
\begin{equation}\label{PDNHs}%PDNHs prof deco norm Hs  
\|\phi_n\|^2_{\dot{H}^s}=\sum_{j=1}^{M}\|\psi^j\|^2_{\dot{H}^s}+\|W_n^M\|^2_{\dot{H}^s}+o_n(1)
 \end{equation}
 where $o_n(1) \rightarrow 0$ as $n\rightarrow +\infty$.
 \end{itemize}
  \begin{proof}  Let $C_1>0$ such that $\|\phi_n\|_{H^1}\leq C_1$.
  For every $(a,r)$ $\dot{H}^{s_c}$-admissible we can define $r_1=2r$ and $a_1=\frac{4r}{r(3-2s_c)-3}$. Note that $(a_1,r_1)$ is also $\dot{H}^{s_c}$-admissible, then combining the interpolation inequality with $\eta=\frac{3}{r(3-2s_c)-3}\in (0,1)$ and the Strichartz estimate \eqref{SE2}, we have
 \begin{eqnarray}\label{LPD1} %linear prof decomp 1
  \|U(t)W_n^M\|_{L_t^aL^r_x}&\leq&\|U(t)W_n^M\|^{1-\eta}_{L_t^{a_1}L^{r_1}_x}\|U(t)W_n^M\|^\eta_{L_t^\infty L^{\frac{6}{3-2s_c}}_x}\nonumber\\
  & \leq & \|W_n^M\|^{1-\eta}_{\dot{H}^{s_c}}\|U(t)W_n^M\|^\eta_{L_t^\infty L^{\frac{6}{3-2s_c}}_x}.
\end{eqnarray}
 Since we will have $\|W_n^M\|_{\dot{H}^{s_c}}\leq C_1$, then we need to show that the second norm in the right hand side of \eqref{LPD1} goes to zero as $n$ and $M$ go to infinite, that is
 \begin{equation}\label{LPD2}
\lim\limits_{M \rightarrow +\infty}\left(\limsup\limits_{n \rightarrow +\infty}\|U(t)W_n^M\|_{L_t^\infty L^{\frac{6}{3-2s_c}}_x}\right)=0.
\end{equation}

\ First we construct $\psi^1_n$, $t_n^1$ and $W_n^1$. Let
 $$
 A_1=\limsup \limits_{n\rightarrow +\infty} \|U(t)\phi_n\|_{L_t^\infty L^{\frac{6}{3-2s_c}}_x}.
 $$
If $A_1=0$, the proof is complete with $\psi^j=0$ for all $j=1,\dots,M$.  
Assume that $A_1>0$. Passing to a subsequence, we may consider $A_1=\lim \limits_{n\rightarrow +\infty} \|U(t)\phi_n\|_{L_t^\infty L^{\frac{6}{3-2s_c}}_x}$. We claim that there exist a time sequence $t_n^1$ and $\psi^1$ such that $U(t_n^1)\phi_n \rightharpoonup \psi^1$ and
\begin{equation}\label{LPD22}
\beta C_1^ { \frac{3-2s_c}{2s_c(1-s_c)} }\|\psi^1\|_{\dot{H}^{s_c}}\geq A_1^{\frac{3-2s_c^2}{2s_c(1-s_c)}}, 
\end{equation} 
where $\beta>0$ is independent of $C_1$, $A_1$ and $\phi_n$.
Indeed, let $\zeta\in C^\infty_0(\mathbb{R}^3)$ a real-valued and radially symmetric function such that $0\leq \zeta \leq 1$, $\zeta(\xi)=1$ for $|\xi|\leq 1$ and $\zeta(\xi)=0$ for $|\xi|\geq 2$. Given $r>0$, define $\chi_r$ by $\widehat{\chi_r}(\xi)=\zeta(\frac{\xi}{r})$. From the Sobolev embedding \eqref{SEI} and since the operator $U(t)$ is an isometry in $H^{s_c}$, we deduce (recalling $0<s_c<1$)    
\begin{align*}\label{LPD23}
 \|U(t)\phi_n -\chi_r*U(t)\phi_n\|^2_{L^\infty_tL_x^{\frac{6}{3-2s_c}}} &\leq c\|U(t)\phi_n -\chi_r*U(t)\phi_n\|^2_{L^\infty_tH_x^{s_c} } \nonumber \\
 &\leq c \int |\xi|^{2s_c}|(1-\widehat{\chi_r})^2|\widehat{\phi}_n(\xi)|^2d\xi  \nonumber \\
 &\leq c\int_{|\xi|>r} |\xi|^{-2(1-s_c)} |\xi|^2 |\widehat{\phi}_n(\xi)|^2d\xi \nonumber \\
 &\leq c r^{-2(1-s_c)}\|\phi\|^2_{\dot{H}^1} \leq c r^{-2(1-s_c)}C_1^2.
 \end{align*}
Choosing    
\begin{equation}\label{LPD24}
 r=\left(\frac{4\sqrt{c}C_1}{A_1}\right)^{\frac{1}{1-s_c}}
\end{equation}
 and for $n$ large enough we have
 \begin{equation}\label{LPD3}
\|\chi_r*U(t)\phi_n\|_{L^\infty_tL_x^{\frac{6}{3-2s_c}}}\geq \frac{A_1}{2}.
\end{equation} 
Note that, from the standard interpolation in Lebesgue spaces
 \begin{eqnarray}\label{LPD4}
 \|\chi_r*U(t)\phi_n\|^3_{L^\infty_tL_x^{\frac{6}{3-2s_c}}}&\leq&\|\chi_r*U(t)\phi_n\|^{3-2s_c}_{L^\infty_tL_x^2} \|\chi_r*U(t)\phi_n\|^{2s_c}_{L^\infty_tL_x^\infty}\nonumber  \\
 &\leq & C_1^{3-2s_c}\|\chi_r*U(t)\phi_n\|^{2s_c}_{L^\infty_tL_x^\infty},
\end{eqnarray}
 thus inequalities \eqref{LPD3} and \eqref{LPD4} lead to
 $$
 \|\chi_r*U(t)\phi_n\|_{L^\infty_tL_x^\infty}\geq \left(\frac{A_1}{2C_1^{\frac{3-2s_c}{3}}}  \right)^{\frac{3}{2s_c}}.
 $$
 It follows from the radial Sobolev Gagliardo-Nirenberg inequality (since all $\phi_n$ are radial functions and so are $\chi_r*U(t)\phi_n$) that
\begin{eqnarray*}
\|\chi_r*U(t)\phi_n\|_{L^\infty_tL_x^\infty(|x|\geq R)}&\leq &\frac{1}{R} \|\chi_r*U(t)\phi_n\|^{\frac{1}{2}}_{L^2_x}\|\nabla(\chi_r*U(t)\phi_n)\|^{\frac{1}{2}}_{L^2_x}\leq \frac{C_1}{R},
\end{eqnarray*} 
 which implies for $R>0$ sufficiently large
 $$
 \|\chi_r*U(t)\phi_n\|_{L^\infty_tL_x^\infty(|x|\leq R)}\geq \frac{1}{2} \left(\frac{A_1}{2C_1^{\frac{3-2s_c}{3}}}  \right)^{\frac{3}{2s_c}},
 $$
 where we have used the two last inequalities. Now, let $t_n^1$ and $x_n^1$, with $|x_n^1|\leq R$, be sequences such that for each $n\in \mathbb{N}$  
 $$
 \left|\chi_r*U(t_n^1)\phi_n(x_n^1)\right|\geq \frac{1}{4}\left(\frac{A_1}{2C_1^{\frac{3-2s_c}{3}}}  \right)^{\frac{3}{2s_c}}
 $$
 or 
\begin{equation}\label{A_1}
 \frac{1}{4}\left(\frac{A_1}{2C_1^{\frac{3-2s_c}{3}}}  \right)^{\frac{3}{2s_c}}\leq \left|\int \chi_r(x_n^1-y)U(t_n^1)\phi_n(y)dy\right|. 
 \end{equation}
 On the other hand, since $\|U(t_n^1)\phi_n\|_{H^1}=\|\phi_n\|_{H^1}\leq C_1$  then $U(t^1_n)\phi_n$ converges weakly in $H^1$, i.e., there exists $\psi^1$ a radial function such that (up to a subsequence) $U(t_n^1)\phi_n \rightharpoonup \psi^1$ in $H^1$ and $\|\psi^1\|_{H^1}\leq \limsup \limits_{n\rightarrow +\infty}\|\phi_n\|_{H^1}\leq C_1$. In addition, $x_n^1\rightarrow x^1$ (also up to a subsequence) since $x_n^1$ is bounded. Hence the inequality \eqref{A_1}, the Plancherel formula and the Cauchy-Schwarz inequality yield
 $$ 
 \frac{1}{8}\left(\frac{A_1}{2C_1^{\frac{3-2s_c}{3}}}  \right)^{\frac{3}{2s_c}}\leq \left|\int \chi_r(x^1-y) \psi^1(y)dy \right|\leq \|\chi_r\|_{\dot{H}^{-s_c}}\|\psi^1\|_{\dot{H}^{s_c}},
 $$
 which implies $\frac{1}{8}\left(\frac{A_1}{2C_1^{\frac{3-2s_c}{3}}}  \right)^{\frac{3}{2s_c}}\leq cr^{\frac{3-2s_c}{2}}\|\psi^1\|_{\dot{H}^{s_c}}$, where we have used 
 $$
 \|\chi_r\|_{\dot{H}^{-s_c}}=\left(\int_{0<|\xi|<2r}|\xi|^{-2s_c}|\widehat{\chi_r}(\xi)|^2d\xi\right)^\frac{1}{2}\leq c\left(\int_{0}^{2r}\rho^{-2s_c}\rho^{2}d\rho\right)^\frac{1}{2}\leq cr^{\frac{3-2s_c}{2}}.
 $$
 Therefore in view of our choice of $r$ (see \eqref{LPD24}) we obtain \eqref{LPD22}, concluding the claim. 
 
 \ Next, define $W^1_n =\phi_n-U(-t_n^1)\psi^1$. It is easy to see that, for any $0\leq s\leq 1$,  
 \begin{itemize}
\item $U(t_n^1)W^1_n \rightharpoonup 0$ in $H^1$ (since $U(t_n^1)\phi_n \rightharpoonup \psi^1$),
\item  $\langle \phi_n,U(-t^1_n)\psi^1 \rangle_{\dot{H}^s}=\langle U(t^1_n)\phi_n,\psi^1 \rangle_{\dot{H}^s}\rightarrow \|\psi^1\|^2_{\dot{H}^s}$,
\item $\|W_n^1\|^2_{\dot{H}^s}=\|\phi_n\|^2_{\dot{H}^s}-\|\psi^1\|^2_{\dot{H}^s}+o_n(1)$.
\end{itemize}
%The group property immediately gives \eqref{PDNHs} in the case $M=1$.
The last item, with $s=0$ and $s=1$, implies $\|W_n^1\|_{H^1}\leq C_1$.

\ Let $A_2=\limsup \limits_{n\rightarrow +\infty}\|U(t)W_n^1\|_{L_t^\infty L^{\frac{6}{3-2s}}_x}$. If $A_2=0$ the result follows taking $\psi^j=0$ for all $j=2,\dots,M$.. Let $A_2>0$, repeating the above argument with $\phi_n$ replaced by $W_n^1$ we obtain a sequence $t_n^2$ and a function $\psi^2$ such that $U(t_n^2)W_n^1\rightharpoonup \psi^2$ in $H^1$ and $\beta C_1^ { \frac{3-2s_c}{2s_c(1-s_c)} }\|\psi^2\|_{\dot{H}^{s_c}}\geq A_2^{\frac{3-2s_c^2}{2s_c(1-s_c)}}.
$ 

\ We now prove that $|t_n^2-t_n^1|\rightarrow +\infty$. In fact, if we suppose (up to a subsequence) $t_n^2-t_n^1\rightarrow t^*$ finite, then 
$$
 U(t_n^2-t_n^1)\left(U(t_n^1)\phi_n-\psi^1 \right)=U(t_n^2)\left(\phi_n-U(-t_n^1)\psi^1 \right)=U(t_n^2)W_n^1\rightharpoonup \psi^2.
$$
On the other hand, since $U(t_n^1)\phi_n\rightharpoonup \psi^1$, the left side of the above expression converges weakly to $0$, and thus $\psi^2=0$, a contradiction. Define $W_n^2=W_n^1-U(-t_n^2)\psi^2$. For any $0\leq s\leq 1$, since $|t_n^1-t_n^2|\rightarrow +\infty$, we deduce
\begin{eqnarray*}
\langle \phi_n,U(-t_n^2)\psi^2 \rangle_{\dot{H}^{s}}&=&\langle U(t_n^2)\phi_n,\psi^2 \rangle_{\dot{H}^{s}}\\
&=&\langle U(t_n^2)\left(W_n^1+U(-t_n^1)\psi^1\right),\psi^2 \rangle_{\dot{H}^{s}}\\%+o_n(1)\rightarrow \|\psi^2\|_{\dot{H}^{s}}\\
&=&\langle U(t_n^2)W_n^1,\psi^2 \rangle_{\dot{H}^{s}}+\langle U(t_n^2-t_n^1)\psi^1,\psi^2 \rangle_{\dot{H}^{s}} \\
&\rightarrow& \|\psi^2\|^2_{\dot{H}^{s}}.
\end{eqnarray*}
In addition, the definition of $W_n^2$ implies that 
$$
\|W_n^2\|^2_{\dot{H}^s}=\|W_n^1\|^2_{\dot{H}^{s_c}}-\|\psi^2\|^2_{\dot{H}^s}+o_n(1)
$$
and $\|W_n^2\|_{H^1}\leq C_1$.

\ By induction we can construct $\psi^M$, $t_n^M$ and $W_n^M$ such that $U(t_n^M)W_n^{M-1}\rightharpoonup \psi^M$ in $H^1$ and 
\begin{equation}\label{LPD5}
\beta C_1^ { \frac{3-2s_c}{2s_c(1-s_c)} }\|\psi^M\|_{\dot{H}^{s_c}}\geq A_M^{\frac{3-2s_c^2}{2s_c(1-s_c)}},
\end{equation} 
where $A_M=\lim \limits_{n\rightarrow +\infty}\|U(t)W_n^{M-1}\|_{L_t^\infty L^{\frac{6}{3-2s_c}}_x}$. 

\ Next, we show \eqref{PD}. Suppose $1\leq j<M$, we prove that $|t^M_n-t_n^j|\rightarrow +\infty$ by induction assuming $|t^M_n-t_n^k|\rightarrow +\infty$ for $k=j+1, \dots, M-1$. Indeed, let $t^M_n-t_n^j\rightarrow t_0$ finite (up to a subsequence) then it is easy to see
$$
U(t_n^M-t_n^j)\left(U(t_n^j)W_n^{j-1}-\psi^j\right)-U(t_n^M-t_n^{j+1})\psi^{j+1}-...-U(t_n^M-t_n^{M-1})\psi^{M-1}
$$
$$
=U(t_n^M)W_n^{M-1}\rightharpoonup \psi^M.
$$
 Since the left side converges weakly to $0$, we have $\psi^M=0$, a contradiction. 
 
 We now consider 
$$
W_n^M=\phi_n-U(-t_n^1)\psi^1-U(-t_n^2)\psi^2-...-U(-t_n^M)\psi^M.
$$
Similarly as before, by \eqref{PD} we get for any $0\leq s\leq 1$
$$
\langle \phi_n,U(-t_n^M)\psi^M \rangle_{\dot{H}^s}=\langle U(t_n^M)W_n^{M-1},\psi^M \rangle_{\dot{H}^s}+o_n(1),
$$
and so $\langle \phi_n,U(-t_n^M)\psi^M \rangle_{\dot{H}^s}\rightarrow \|\psi^M\|^2_{\dot{H}^s}$. Thus expanding $\|W_n^M\|^2_{\dot{H}^s}$ we deduce that \eqref{PDNHs} also holds. 

 \ Finally, the inequality \eqref{LPD5} together with the relation \eqref{PDNHs} yield 
$$
\sum_{M\geq 1} \left(\frac{A_M^{\frac{3-2s_c^2}{s_c(1-s_c)}}}{\beta^2C_1^{ \frac{3-2s_c}{s_c(1-s_c)}  }}\right)\leq \lim_{n\rightarrow+\infty}\|\phi_n\|^2_{\dot{H}^{s_c}}<+\infty,
$$
 which implies that $A_M\rightarrow 0$ as $M\rightarrow +\infty$ i.e., \eqref{LPD2} holds\footnote{ Note that $3-2s_c^2>0$ since $s_c=\frac{1+b}{2}<1$.}. Therefore, from \eqref{LPD1} we get \eqref{AS}. This completes the proof.
\end{proof} 
\end{proposition}

\begin{remark}\label{RLPD} %RLPD = remark linear prof decomp 
It follows from the proof of Proposition \ref{LPD} that 
\begin{equation}\label{RLPD1}
\lim\limits_{M,n\rightarrow \infty} \|W_n^{M}\|_{L^{p}} =0,
\end{equation}
where $2<p<6$. Indeed, first we show
\begin{equation}\label{RLPD2}
\lim\limits_{M \rightarrow +\infty}\left(\lim\limits_{n \rightarrow +\infty}\|U(t)W_n^M\|_{L_t^\infty L^p_x}\right)=0.
\end{equation}
Note that, $\dot{H}^{s}\hookrightarrow L^{p}$ where $s=\frac{3}{2}-\frac{3}{p}$ (see inequality \eqref{SEI}). Since $2<p<6$ then $0<s<1$, thus repeating the argument used for showing \eqref{LPD2} with $\frac{6}{3-2s_c}$ replaced by $p$ and $s_c$ by $s$, we obtain \eqref{RLPD2}. On the other hand, \eqref{RLPD1} follows directly from \eqref{RLPD2} and the inequality
$$
\|W_n^{M}\|_{L_x^{p}}\leq \|U(t)W_n^{M}\|_{L^\infty_tL_x^{p}},
$$
since $W_n^{M}=U(0)W_n^{M}$.
\end{remark}

\begin{proposition}\label{EPE}{\bf (Energy Pythagoream Expansion)} Under the hypothesis of Proposition \ref{LPD} we obtain
\begin{equation}\label{PDE} %pdf prof decomp energy
E[\phi_n]=\sum_{j=1}^{M}E[U(-t_n^j)\psi^j]+E[W_n^M]+o_n(1).
\end{equation}

\begin{proof} By definition of $E[u]$ and \eqref{PDNHs} with $s=1$, we have
	$$
E[\phi_n]-\sum_{j=1}^{M}E[U(-t_n^j)\psi^j]-E[W_n^M]=-\frac{A_n}{\alpha+2}+o_n(1),
$$
where 
$$
A_n=\left\|  |x|^{-b}|\phi_n|^{4}  \right\|_{L^1}-\sum_{j=1}^{M}\left\|  |x|^{-b} |U(-t_n^j)\psi^j|^{4}  \right\|_{L_x^1}-\left\|  |x|^{-b}|W_n^M|^{4}  \right\|_{L^1}.
$$

\ For a fixed $M\in \mathbb{N}$, if $A_n \rightarrow 0$ as $n\rightarrow +\infty$ then \eqref{PDE} holds. To prove this fact, pick $M_1\geq M$ and rewrite the last expression as
\begin{eqnarray*}
A_n&=& \int \left(  |x|^{-b}|\phi_n|^{4}  -  \sum_{j=1}^{M}   |x|^{-b}  |U(-t_n^j)\psi^j|^{4} -  |x|^{-b} |W_n^M|^{4}   \right)dx     \\
&=&  I^1_n+I^2_n+I^3_n,     
\end{eqnarray*}   
where 
\begin{eqnarray*}
I^1_n&=&  \int   |x|^{-b} \left[ |\phi_n|^{4}-|\phi_n-W_n^{M_1}|^{4}\right]dx, \\
I^2_n&=& \int |x|^{-b} \left[ |W_n^{M_1}-W_n^M|^{4} -|W_n^M|^{4}\right]dx, 
\end{eqnarray*}
\begin{eqnarray*}
I^3_n&=& \int  |x|^{-b}\left[ |\phi_n-W_n^{M_1}|^{4}- \sum_{j=1}^{M} |U(-t_n^j)\psi^j|^{4}- |W_n^{M_1}-W_n^M|^{4}   \right] dx .
\end{eqnarray*}

\ We first estimate $I^1_n$. Combining \eqref{FI} and Lemma \ref{LEWO} (i)-(ii) we have %%%%5(recaling $r$ is defined in )
\begin{eqnarray*}
|I^1_n|&\lesssim& \int |x|^{-b}\left( |\phi_n|^{3}|W_n^{M_1}|+|\phi_n||W_n^{M_1}|^{3}  + |W_n^{M_1}|^{4}  \right)dx\\
&\lesssim&  \left(  \|\phi_n\|^{3}_{L^r}\|W_n^{M_1}\|_{L^r}+ \|\phi_n\|_{L^r}\|W_n^{M_1}\|^{3}_{L^{r}} +  \|W_n^{M_1}\|^{4}_{L^{r}} \right)+\\
& & \left(  \|\phi_n\|^{3}_{L^{4}}\|W_n^{M_1}\|_{L^{4}}+ \|\phi_n\|_{L^{4}}\|W_n^{M_1}\|^{3}_{L^{4}} + \|W_n^{M_1}\|^{4}_{{L^{4}}} \right)\\
&\lesssim&    \|\phi_n\|^{3}_{H^1}\|W_n^{M_1}\|_{L^r}+ \|\phi_n\|_{H^1}\|W_n^{M_1}\|^{3}_{L^r} +  \|W_n^{M_1}\|^{3}_{L^r} +\\
 & &  \|\phi_n\|^{3}_{H^1}\|W_n^{M_1}\|_{L^{4}}+ \|\phi_n\|_{H^1}\|W_n^{M_1}\|^{3}_{L^{4}} +  \|W_n^{M_1}\|^{4}_{L^{4}} ,
\end{eqnarray*}
where $\frac{12}{3-b}<r<6$. In view of inequality \eqref{RLPD1} and since $\{\phi_n\}$ is uniformly bounded in $H^1$, we conclude that\footnote{We can apply Remark \ref{RLPD} since $r$ and $4 \in (2,6)$.} 
$$
I^1_n\rightarrow +\infty\;\;\textnormal{as}\;\;n, M_1\rightarrow +\infty.
$$  

\ Also, by similar arguments (replacing $\phi_n$ by $W_n^{M}$) we have $$
I^2_n\rightarrow +\infty\;\;\textnormal{as}\;\;n, M_1\rightarrow +\infty,
$$
where we have used that $W_n^{M}$ is uniformly bounded by \eqref{PDNHs}.

\ Finaly we consider the term $I^3_n$. Since,
$$
\phi_n-W_n^{M_1}=\sum\limits_{j=1}^{M_1}U(-t_n^j)\psi^j\,\,\,\textnormal{and}\,\,\,W_n^{M}-W_n^{M_1}=\sum\limits_{j=M+1}^{M_1}U(-t_n^j)\psi^j,
$$
we can rewrite $I^3_n$ as
\begin{eqnarray*}
I^3_n=\int |x|^{-b} \left( \left| \sum\limits_{j=1}^{M_1} U(-t_n^j)\psi^j \right|^{4}- \sum\limits_{j=1}^{M_1}| U(-t_n^j)\psi^j |^{4}\right)dx
\end{eqnarray*}
$$
-\int |x|^{-b} \left( \left| \sum\limits_{j=M+1}^{M_1} U(-t_n^j)\psi^j \right|^{4}- \sum\limits_{j=M+1}^{M_1}| U(-t_n^j)\psi^j |^{4}\right)dx.
$$

To complete the prove we make use of the following claim.

 \textit{Claim.} For a fixed $M_1\in \mathbb{N}$ and for some $j_0\in \mathbb{N}$ ($j_0< M_1$), we get
$$
D_n= \left\||x|^{-b} \left|\sum_{j=j_0}^{M_1}  U(-t_n^j)\psi \right|^{4}  \right\|_{L^1_x}  -    \sum_{j=j_0}^{M_1} \left\|   |x|^{-b} |U(-t_n^j)\psi^j|^{4}   \right\|_{L^1_x}\rightarrow 0\;,\;\;\textnormal{as}\;\;\;n\rightarrow +\infty.
$$

\ Indeed, it is clear that the last limit implies that $I^3_n\rightarrow 0\;\textnormal{as}\;n\rightarrow +\infty$ completing the proof of relation \eqref{PDE}.

To prove the claim note that \eqref{FI} implies
$$
D_n\leq \sum_{j\neq k}^{M_1}\int |x|^{-b}|U(-t_n^j)\psi^j||U(-t_n^k)\psi^k|^{3}dx.
$$
Thus, from Lemma \ref{LEWO} (i) one has
\begin{equation*}\label{lematecnico1}
E^{j,k}_n\leq  c  \| U(-t_n^k)\psi^k\|^{3}_{L^{4}_x}\| U(-t_n^j)\psi^j\|_{L^{4}_x}+c\| U(-t_n^k)\psi^k\|^{3}_{L^r_x}\| U(-t_n^j)\psi^j\|_{L^r_x},
\end{equation*}
where $2<\frac{12}{3-b}< r<6$ and $E^{j,k}_n=\int |x|^{-b}|U(-t_n^j)\psi^j||U(-t_n^k)\psi^k|^{3}dx$. In view of \eqref{PD} we can consider that $t_n^k$, $t_n^j$ or both go to infinite as $n$ goes to infinite. If $t_n^j\rightarrow +\infty$ as $n\rightarrow +\infty$, so it follow from the last inequality and since $H^{1}\hookrightarrow L^{4}$ and $H^{1}\hookrightarrow L^{r}$ that %or $t_n^k\rightarrow +\infty$ as $n\rightarrow +\infty$ we have, respectively
\begin{eqnarray*}
E^{j,k}_n & \leq&  c \|\psi^k\|^{3}_{H^1}\| U(-t_n^j)\psi^j\|_{L^{4}_x} +c\| \psi^k\|^{3}_{H^1} \| U(-t_n^j)\psi^j\|_{L^r_x}\\
& \leq&  c \| U(-t_n^j)\psi^j\|_{L^{4}_x} +c\| U(-t_n^j)\psi^j\|_{L^r_x},
\end{eqnarray*}
where in the last inequality we have used that $(\psi^k)_{k\in\mathbb{N}}$ is a uniformly bounded sequence in $H^1$. Hence, if $n\rightarrow +\infty$ we have $t_n^j\rightarrow +\infty$ and using \eqref{LEWO3} with $t=t_n^j$ and $f=\psi^j$ we conclude that $E^{j,k}_n\rightarrow 0$ as $n\rightarrow +\infty$. Similarly for the case $t^k_n\rightarrow +\infty$ as $n\rightarrow +\infty$, we have 
\begin{eqnarray*}
E^{j,k}_n & \leq&  c \| U(-t_n^k)\psi^k\|^{3}_{L^{4}_x} \|\psi^j\|_{H^1}+c\| U(-t_n^k)\psi^k\|^{3}_{L^r_x}\| \psi^j\|_{H^1}\\
& \leq&  c  \| U(-t_n^k)\psi^k\|^{3}_{L^{4}_x}+c\| U(-t_n^k)\psi^k\|^{3}_{L^r_x},
\end{eqnarray*}
which implies that $E^{j,k}_n\rightarrow 0$ as $n\rightarrow +\infty$ by \eqref{LEWO3} with $t=t_n^k$ and $f=\psi^k$. Finally, since $D_n$ is a finite sum of terms in the form of $E^{j,k}$ we deduce $D_n\rightarrow 0$ as $n\rightarrow +\infty$. 
%By \eqref{PD} we can assume without loss of generality that $t_n^k\rightarrow +\infty$. Then, in view of Lemma \ref{LEWO} (iii) with $f=\psi^k$ and $g_n=U(-t_n^j)\psi^j$, we have
%\begin{eqnarray*}
 %  \lim\limits_{n\rightarrow +\infty} \int |x|^{-b}|U(-t_n^j)\psi^j||U(-t_n^k)\psi^k|^{\alpha+1}dx =0.
%\end{eqnarray*}
%Therefore, $D_n\rightarrow 0$ as $n\rightarrow +\infty$ . 
\end{proof}
\end{proposition}

\subsection{Critical solution}\label{CCS}

In this subsection, assuming that $\delta_c<E[u]^{s_c}M[u]^{1-s_c}$ (see \eqref{deltac}), we construct a global solution, denoted by $u_c$, of \eqref{INLS} with infinite Strichartz norm $\|\cdot\|_{S(\dot{H}^{s_c})}$ and satisfying 
$$
E[u_c]^{s_c}M[u_c]^{1-s_c}=\delta_c.
$$ 
Next, we show that the flow associated to this critical solution is precompact in $H^1(\mathbb{R}^3)$. 

\begin{proposition}\label{ECS}{\bf (Existence of a critical solution)} If $\delta_c<E[Q]^{s_c}M[Q]^{1-s_c},$
%\begin{equation*}
%\delta_c<E[Q]^{s_c}M[Q]^{1-s_c},
%\end{equation*}
then there exists a radial function $u_{c,0}\in H^1(\mathbb{R}^3)$ such that the corresponding solution $u_c$ of the IVP \eqref{INLS} is global in $H^1(\mathbb{R}^3)$. Moreover the following properties hold
\begin{itemize}
	\item [(i)] $M[u_c]=1$,
    \item [(ii)] $E[u_c]^{s_c}=\delta_c$,
    \item [(iii)] $\|  \nabla u_{c,0} \|_{L^2}^{s_c} \|u_{c,0}\|_{L^2}^{1-s_c}<\|\nabla Q \|_{L^2}^{s_c} \|Q\|_{L^2}^{1-s_c}$,
    \item [(iv)] $\|u_{c}\|_{S(\dot{H}^{s_c})}=+\infty$.
\end{itemize}
\begin{proof}
 \  Recall from Subsection \ref{SPMR} that there exists a sequence of solutions $u_n$ to \eqref{INLS} with $H^1$
 initial data $u_{n,0}$, with $\|u_n\|_{L^2} = 1$ for all $n\in \mathbb{N}$, such that 
 \begin{equation}\label{PCS1}   %   prop critical condition 1
 \|\nabla u_{n,0}\|^{s_c}_{L^2} <  \|\nabla Q\|^{s_c}_{L^2}\|Q\|^{1-s_c}_{L^2}
 \end{equation}
 and
 \begin{equation*}\label{PCS2}
 E[u_n] \searrow \delta_c^{\frac{1}{s_c}}\;\; \textnormal{as}\;\; n \rightarrow +\infty.
 \end{equation*} 	
Moreover 
\begin{equation}\label{un}
\|u_n\|_{S(\dot{H}^{s_c})}=+\infty
\end{equation}
for every $n\in \mathbb{N}$. Note that, in view of the assumption $\delta_c<E[Q]^{s_c}M[Q]^{1-s_c}$, there exists $a \in (0,1)$ such that
\begin{equation}\label{PCS21}
 E[u_n]\leq  a E[Q]M[Q]^{\sigma},
\end{equation}
where $\sigma=\frac{1-s_c}{s_c}$. Furthermore, \eqref{PCS1} implies by Lemma \ref{LGS} (ii) that
\begin{equation*}
 \|\nabla u_{n,0}\|^{2}_{L^2} \leq w^{\frac{1}{s_c}} \|\nabla Q\|^{2}_{L^2}\|Q\|^{2\sigma}_{L^2},
\end{equation*}
where $w=\frac{E[u_n]^{s_c}M[u_n]^{1-s_c}}{E[Q]^{s_c}M[Q]^{1-s_c}}$, thus we deduce from \eqref{PCS21} and $\|u_n\|_{L^2} = 1$ that $w^{\frac{1}{s_c}}\leq a$ which implies
\begin{equation}\label{PCS22}
\|\nabla u_{n,0}\|^{2}_{L^2} \leq a \|\nabla Q\|^{2}_{L^2}\|Q\|^{2\sigma}_{L^2}.
\end{equation}

\ On the other hand, the linear profile decomposition (Proposition \ref{LPD}) applied to $u_{n,0}$, which is a uniformly bounded sequence in $H^1(\mathbb{R}^3)$ by \eqref{PCS22}, yields
 \begin{equation}\label{PCS3}
 u_{n,0}(x)=\sum_{j=1}^{M}U(-t_n^j)\psi^j(x)+W_n^M(x),
 \end{equation} 
where $M$ will be taken large later. 
It follows from the Pythagorean expansion \eqref{PDNHs}, with $s=0$, that for all $M\in \mathbb{N}$
\begin{equation}\label{PCS4}
\sum_{j=1}^{M}\|\psi^j\|^2_{L^2}+\lim_{n\rightarrow +\infty}\|W_n^M\|^2_{L^2}\leq \lim_{n\rightarrow +\infty}\|u_{n,0}\|^2_{L^2}= 1,
\end{equation}
this implies that 
\begin{equation}\label{PCS41}
\sum_{j=1}^{M}\|\psi^j\|^2_{L^2}\leq 1.
\end{equation}
In addition, another application of \eqref{PDNHs}, with $s=1$, and \eqref{PCS22} lead to
\begin{equation}\label{Sumpsij}
\sum_{j=1}^{M}\|\nabla \psi^j\|^2_{L^2}+\lim_{n\rightarrow +\infty}\|\nabla W_n^M\|^2_{L^2}\leq \lim_{n\rightarrow +\infty}\|\nabla u_{n,0}\|^2_{L^2}\leq a\|\nabla Q\|^2_{L^2}\|Q\|^{2\sigma}_{L^2},
\end{equation}
and so
\begin{equation}\label{PCS5}
\|\nabla \psi^j\|^{s_c}_{L^2}\leq a^{\frac{s_c}{2}}\|\nabla Q\|^{s_c}_{L^2}\| Q\|^{1-s_c}_{L^2},\;\;j=1,\dots,M.
\end{equation}

 \ Let $\{t^j_n\}_{n\in \mathbb{N}}$ be the sequence given by Proposition \ref{LPD}. From \eqref{PCS41}, \eqref{PCS5} and the fact that $U(t)$ is an isometry in $L^2(\mathbb{R}^3)$ and $\dot{H}^1(\mathbb{R}^3)$ we deduce
 \begin{equation*}\label{U(-t_n^j)}
 \|U(-t_n^j)\psi^j\|^{1-s_c}_{L^2_x}\|\nabla U(-t_n^j)\psi^j\|^{s_c}_{L^2_x}\leq a^{\frac{s_c}{2}} \|\nabla Q\|^{s_c}_{L^2}\| Q\|^{1-s_c}_{L^2}.
\end{equation*}
Now, Lemma \ref{LGS} (i) yields
\begin{equation}\label{PCS6}
E[U(-t_n^j)\psi^j]\geq c(b)\|\nabla \psi^j\|_{L^2}\geq 0
\end{equation}
 
\ A complete similar analysis also gives, for all $M\in \mathbb{N}$,
$$
\lim_{n\rightarrow +\infty}\|W_n^M\|^2_{L^2}\leq 1,
$$	
$$
\lim_{n\rightarrow +\infty}\|\nabla W_n^M\|^{s_c}_{L^2}\leq a^{\frac{s_c}{2}}\|\nabla Q\|^{s_c}_{L^2}\|Q\|^{1-s_c}_{L^2},
$$	
and for $n$ large enough (depending on $M$)
\begin{equation}\label{PCS7}
E[W_n^M]\geq 0.
\end{equation}

\ The energy Pythagorean expansion (Proposition \ref{EPE}) allows us to deduce
$$
\sum_{j=1}^{M}\lim_{n\rightarrow+\infty}E[U(-t_n^j)\psi^j]+\lim_{n\rightarrow+\infty}E[W_n^M]=\lim_{n\rightarrow+\infty}E[u_{n,0}]=\delta_c^{\frac{1}{s_c}},
$$
which implies, by \eqref{PCS6} and \eqref{PCS7}, that  
\begin{equation}\label{PCS8}
\lim_{n\rightarrow\infty}E[U(-t_n^j)\psi^j]\leq \delta_c^{\frac{1}{s_c}},\;\textnormal{for all}\;\;j=1,...,M.
\end{equation}

 \ Now, if more than one $\psi^j\neq 0$, we show a contradiction and thus the profile expansion given by \eqref{PCS3} is reduced to the case that only one profile is nonzero. In fact, if more than one $\psi^j\neq 0$, then by \eqref{PCS4} we must have $M[\psi^j]<1$ for each $j$. Passing to a subsequence, if necessary, we have two cases to consider: 
  
  \ {\bf Case $1$}. If for a given $j$, $t^j_n\rightarrow t^*$ finite (at most only one such $j$ exists by \eqref{PD}), then the continuity of the linear flow in $H^1(\mathbb{R}^3)$ yields   
   \begin{equation}\label{widetildepsi}
  U(-t_n^j)\psi^j\rightarrow U(-t^*)\psi^j\;\;\;\;\textnormal{strongly in}\;H^1.
   \end{equation}

Let us denote the solution of \eqref{INLS} with initial data $\psi$ by INLS$(t)\psi$. Set $\widetilde{\psi}^j=\textnormal{INLS}(t^*)(U(-t^*)\psi^j)$ so that $\mbox{INLS}(-t^*)\widetilde{\psi}^j=U(-t^*)\psi^j$. Since the set
\begin{equation*}
\mathcal{K}:=\left\{u_0\in H^1(\mathbb{R}^3):\;  \textrm{relations} \; \eqref{EMC} \; \textrm{and} \; \eqref{GFC} \; \textrm{hold}\;\right\}
\end{equation*}
is closed in $H^1(\mathbb{R}^3)$ then $\widetilde{\psi}^j\in \mathcal{K}$ and therefore INLS$(t)\widetilde{\psi}^j$ is a global solution by Theorem \ref{TG}. Moreover from \eqref{U(-t_n^j)}, \eqref{PCS8} and the fact that $M[\psi^j]<1$ we have
$$
\|\widetilde{\psi}^j\|^{1-s_c}_{L^2_x}\|\nabla \widetilde{\psi}^j\|^{s_c}_{L^2_x}\leq \|\nabla Q\|^{s_c}_{L^2}\| Q\|^{1-s_c}_{L^2}
$$
and 
$$
E[\widetilde{\psi}^j]^{s_c}M[\widetilde{\psi}^j]^{1-s_c}< \delta_c.
$$
So, the definition of $\delta_c$ (see \eqref{deltac}) implies
  \begin{equation}\label{CSCP} %crit sol conlu do profi
  \|\textnormal{INLS}(t)\widetilde{\psi}^j\|_{S(\dot{H}^{s_c})}<+\infty.
  \end{equation}
  Finally, from \eqref{widetildepsi} it is easy to see
\begin{equation}\label{CSWO3}
   \lim_{n\rightarrow+\infty}\|\textnormal{INLS}(-t_n^j)\widetilde{\psi}^j-U(-t_n^j)\psi^j\|_{H^1_x}=0.
  \end{equation}

 \ {\bf Case $2$.} If $|t^j_n|\rightarrow+\infty$ then by Lemma \ref{LEWO} (iii), $\left\||x|^{-b}|U(-t_n^j)\psi^j|^{4}\right\|_{L^1_x}\rightarrow 0$.
% $$
% \left\||x|^{-b}|U(-t_n^j)\psi^j|^{\alpha+2}\right\|_{L^1_x}\rightarrow 0,
% $$
 Thus, by the definition of Energy \eqref{energy} and the fact that $U(t)$ is an isometry in $\dot{H}^1(\mathbb{R}^3)$, we deduce
 \begin{equation}\label{CS0}
 \left( \frac{1}{2}\|\nabla \psi^j\|^2_{L^2}\right)^{s_c}= \lim_{n \rightarrow \infty}  E [ U(-t_n^j) \psi^j]^{s_c} \leq  \delta_c < E[Q]^{s_c}M[Q]^{1-s_c},
 \end{equation}
  where we have used \eqref{PCS8}. Therefore, by the existence of wave operator, Proposition \ref{PEWO} with $\lambda=(\frac{2 s_c}{3+b})^{\frac{s_c}{2}}<1$ (see also Remark \ref{backward}), there exists $\widetilde{\psi}^j\in H^1(\mathbb{R}^3)$ such that
  \begin{equation}\label{CSWO1}
  M[\widetilde{\psi}^j]=M[\psi^j]\;\;\;\textrm{ and }\;\;\;\;E[\widetilde{\psi}^j]=\frac{1}{2}\|\nabla \psi^j\|^2_{L^2},
  \end{equation}
  \begin{equation}\label{CSWO2}
  \|\nabla \textnormal{INLS}(t)\widetilde{\psi}^j\|^{s_c}_{L^2_x}\|\widetilde{\psi}^j\|^{1-s_c}_{L^2}<\|\nabla Q\|^{s_c}_{L^2}\| Q\|^{1-s_c}_{L^2} 
  \end{equation}
  and \eqref{CSWO3} also holds in this case.
  
   Since $M[{\psi}^j]<1$ and using \eqref{CS0}-\eqref{CSWO1}, we get $E[\widetilde{\psi}^j]^{s_c}M[\widetilde{\psi}^j]^{1-s_c}<\delta_c$. Hence, the definition of $\delta_c$  together with \eqref{CSWO2} also lead to \eqref{CSCP}.
 
  To sum up, in either case, we obtain a new profile $\widetilde{\psi}^j$ for the given $\psi^j$ such that \eqref{CSWO3} \eqref{CSCP} hold. \\
	
Next, we define $u_n(t)=\textnormal{INLS}(t)u_{n,0}$; $v^j(t)=\textnormal{INLS}(t)\widetilde{\psi}^j$; $\widetilde{u}_n(t)=\sum_{j=1}^{M}v^j(t-t_n^j)$ and 
 %$$
 % u_n(t)=\textnormal{INLS}(t)u_{n,0},
 % $$	
 % $$
 % v^j(t)=\textnormal{INLS}(t)\widetilde{\psi}^j,
 % $$
 % $$
 % \widetilde{u}_n(t)=\sum_{j=1}^{M}v^j(t-t_n^j),
 % $$
 %and
\begin{equation}\label{CSR}%crit sol remaind
\widetilde{W}_n^M=\sum_{j=1}^{M}\left[ U(-t_n^j)\psi^j-\textnormal{INLS}(-t_n^j)\widetilde{\psi}^j \right]+W_n^M.
  \end{equation}
 Then $\widetilde{u}_n(t)$ solves the following equation
   \begin{equation}\label{widetildeun}
  i\partial_t\widetilde{u}_n+\Delta \widetilde{u}_n+|x|^{-b}|\widetilde{u}_n|^{2}\widetilde{u}_n=e_n^M,
\end{equation}
where
\begin{equation}\label{CSR1}
 e_n^M=|x|^{-b}\left( |\widetilde{u}_n|^{2}\widetilde{u}_n-\sum_{j=1}^{M}|v^j(t-t_n^j)|^{2}v^j(t-t_n^j) \right).
  \end{equation}
Also note that by definition of $\widetilde{W}_n^M$ in \eqref{CSR} and \eqref{PCS3}we can write
  \begin{equation*}\label{aproximation1}
    u_{n,0}=\sum_{j=1}^{M}\textnormal{INLS}(-t_n^j)\widetilde{\psi}^j+\widetilde{W}_n^M,
  \end{equation*}
so it is easy to see $u_{n,0}-\widetilde{u}_n(0)=\widetilde{W}_n^M$, then combining  \eqref{CSR} and the Strichartz inequality \eqref{SE2}, we estimate
  \begin{equation*}
  \|U(t)\widetilde{W}_n^M\|_{S(\dot{H}^{s_c})}\leq c\sum_{j=1}^{M}\|\textnormal{INLS}(-t_n^j)\widetilde{\psi}^j-U(-t_n^j)\psi^j\|_{H^1}+\|U(t)W_n^M\|_{S(\dot{H}^{s_c})},
  \end{equation*}
  which implies 
 \begin{equation}\label{CSR2}
 \lim_{M\rightarrow +\infty} \left[\lim_{n\rightarrow +\infty}  \|U(t)(u_{n,0}-\widetilde{u}_{n,0})\|_{S(\dot{H}^{s_c})}\right]=0, 
   \end{equation}
   where we used \eqref{AS} and \eqref{CSWO3}. 
  
  \ The idea now is to approximate $u_n$ by $\widetilde{u}_n$. Therefore, from the long time perturbation theory (Proposition \ref{LTP}) and \eqref{CSCP} we conclude $\|u_n\|_{S(\dot{H}^{s_c})}<+\infty$, 
  %$$
  %\|u_n\|_{S(\dot{H}^{s_c})}<+\infty,
  %$$    
for $n$ large enough, which is a contradiction with \eqref{un}. Indeed, we assume the following two claims for a moment to conclude the proof.\\%  first prove the following two claims to reach the contradiction.
{\bf Claim $1$.} For each $M$ and $\varepsilon>0$, there exists $n_0=n_0(M,\varepsilon)$ such that
 \begin{equation}\label{claim2}
n>n_0\;\; \Rightarrow\;\;   \|e_n^M\|_{S'(\dot{H}^{-s_c})}+\|e_n^M\|_{S'(L^2)}+\|\nabla e_n^M\|_{S'(L^2)}\leq\varepsilon.
\end{equation}
\ {\bf Claim $2$.} There exist $L>0$ and $S>0$ independent of $M$ such that for any $M$, there exists $n_1=n_1(M)$ such that
 \begin{equation}\label{claim1}
 n>n_1\;\; \Rightarrow\;\;   \|\widetilde{u}_n\|_{S(\dot{H}^{s_c})}\leq L\;\;\textnormal{and}\;\;\|\widetilde{u}_n\|_{L^\infty_tH^1_x}\leq S.
 \end{equation}

  \ Note that by \eqref{CSR2}, there exists $M_1=M_1(\varepsilon)$ such that for each $M>M_1$ there exists $n_2=n_2(M)$ such that
  $$ 
n>n_2\;\; \Rightarrow\;\;  \|U(t)(u_{n,0}-\widetilde{u}_{n,0})\|_{S(\dot{H}^{s_c})}\leq \varepsilon,
  $$
  with $\varepsilon<\varepsilon_1$ as in Proposition \ref{LTP}. Thus, if the two claims hold true, by Proposition \ref{LTP}, for $M$ large enough and $n>\max\{n_0,n_1,n_2\}$, we obtain  $\|u_n\|_{S(\dot{H}^{s_c})}<+\infty$, reaching the desired contradiction .
  
 \ Up to now, we have reduced the profile expansion to the case where $\psi^1\neq 0$ and $\psi^j= 0$ for all $j\geq 2$. We now begin to show the existence of a critical solution. From the same arguments as the ones in the previous case (the case when more than one $\psi^j \neq 0 $), we can find $\widetilde{\psi}^1$ such that 
\begin{equation*}
u_{n,0}=\textnormal{INLS}(-t_n^1)\widetilde{\psi}^1+\widetilde{W}_n^M,
\end{equation*}   
with
\begin{equation}\label{CSWO11}
M[\widetilde{\psi}^1]=M[\psi^1]\leq 1%\;\;\;\textrm{ and }\;\;\;E[\widetilde{\psi}^1]^{s_c}M[\widetilde{\psi}^1]^{1-s_c}\leq\delta_c,
\end{equation}
\begin{equation}\label{CSWO221}
 E[\widetilde{\psi}^1]^{s_c}=\left(\frac{1}{2} \|\nabla \psi^1 \|^2_{L^2}\right)^{s_c}\leq \delta_c
    \end{equation}
    \begin{equation}\label{CSWO22}
    \|\nabla \textnormal{INLS}(t)\widetilde{\psi}^1\|^{s_c}_{L^2_x}\|\widetilde{\psi}^1\|^{1-s_c}_{L^2}< \|\nabla Q\|^{s_c}_{L^2}\| Q\|^{1-s_c}_{L^2}
    \end{equation}
   and 
   \begin{equation}\label{CSCP1} %crit sol conlu do profi
   \lim_{n\rightarrow +\infty} \|U(t)(u_{n,0}-\widetilde{u}_{n,0})\|_{S(\dot{H}^{s_c})}=\lim_{n\rightarrow +\infty}\|U(t)\widetilde{W}_{n}^M\|_{S(\dot{H}^{s_c})}=0.
    \end{equation}
   
     \ Let $\widetilde{\psi}^1=u_{c,0}$ and $u_c$ be the global solution to \eqref{INLS} (in view of Theorem \ref{TG} and inequalities \eqref{CSWO11}-\eqref{CSWO22}) with initial data $\widetilde{\psi}^1$, that is, $u_c(t)=\textnormal{INLS}(t)\widetilde{\psi}^1$. We claim that
     \begin{equation}\label{claimfinal}
     \|u_c\|_{S(\dot{H}^{s_c})}=+\infty.
     \end{equation}    
	
	\ Indeed, suppose, by contradiction, that $\|u_c\|_{S(\dot{H}^{s_c})}<+\infty$. Let,
	$$
    \widetilde{u}_n(t)=\textnormal{INLS}(t-t_n^j)\widetilde{\psi}^1,
	$$ 
	then 
	$$
	\|\widetilde{u}_n(t)\|_{S(\dot{H}^{s_c})}=\|\textnormal{INLS}(t-t_n^j)\widetilde{\psi}^1\|_{S(\dot{H}^{s_c})}=\|\textnormal{INLS}(t)\widetilde{\psi}^1\|_{S(\dot{H}^{s_c})}=\|u_c\|_{S(\dot{H}^{s_c})}<+\infty.
	$$
Furthermore, it follows from \eqref{CSWO11}-\eqref{CSCP1} that   
$$
\sup_{t\in \mathbb{R}}\|\widetilde{u}_n\|_{H^1_x}=\sup_{t\in \mathbb{R}}\|u_c\|_{H^1_x}<+\infty.
$$
and 
$$
\|U(t)(u_{n,0}-\widetilde{u}_{n,0})\|_{S(\dot{H}^{s_c})}\leq \varepsilon,
$$
for $n$ large enough. Hence, by the long time perturbation theory (Proposition \ref{LTP}) with $e=0$, we obtain $\|u_n\|_{S(\dot{H}^{s_c})}<+\infty$, which is a contradiction with \eqref{un}. 
	
\ On the other hand, the relation \eqref{claimfinal} implies $E[u_c]^{s_c}M[u_c]^{1-s_c}=\delta_c$ (see \eqref{deltac}). Thus, we conclude from \eqref{CSWO11} and \eqref{CSWO221}
$$
M[u_c]=1\;\;\;\;\textnormal{and}\;\;\;\;E[u_c]^{s_c}=\delta_c.
$$ 
Also note that \eqref{CSWO22} implies (iii) in the statement of the Proposition \ref{ECS}.
   
\ To complete the proof it remains to establish Claims $1$ and $2$ (see \eqref{claim1} and \eqref{claim2}). 
%To show these claims we use the same admissible pairs and some estimates already used and proved in \cite[Section $4$, with $(N,\alpha)=(3,2)$]{CARLOS}, to show global well-posedness.             %Subsection \ref{Hstheory}.
%\begin{equation*}\label{PHsA1}  %pares Hs admissiveis1
%\widehat{q}=\frac{4(4-\theta)}{6+2b-\theta(1+b)},\;\;\;\widehat{r}\;=\;\frac{6(4-\theta)}{2(3-b)-\theta(2-b)},
%\end{equation*}
%and
%\begin{equation*}\label{PHsA2}
%\widetilde{a}\;=\;\frac{2(4-\theta)}{(7+2b-3\theta)-(2-b)(1-\theta)},\;\;\;  \widehat{a}=\frac{2(4-\theta)}{1-b}.
%\end{equation*}
%Recall that $(\widehat{q},\widehat{r})$ is $L^2$-admissible, $(\widehat{a},\widehat{r})$ is $\dot{H}^{s_c}$-admissible and $(\widetilde{a},\widehat{r})$ is $\dot{H}^{-s_c}$-admissible.

\ {\bf Proof of Claim $1$.} First, we show that for each $M$ and $\varepsilon>0$, there exists $n_0=n_0(M,\varepsilon)$ such that $\|e_n^M\|_{S'(\dot{H}^{-s_c})}< \frac{\varepsilon}{3}$. From \eqref{CSR1} and \eqref{EIerror} (with $\beta=2$), we deduce 
\begin{equation}\label{ec21} %error claim 2 eq 1
\|e_n^M\|_{S'(\dot{H}^{-s_c})}\leq C_{\alpha,M}\sum_{j=1}^{M}\sum_{1\leq j\neq k\leq M}  \left\||x|^{-b}|v^k|^2|v^j|\right\|_{L^{\widetilde{a}'}_tL^{\widehat{r}'}_x}.
\end{equation}
We claim that the norm in the right hand side of \eqref{ec21} goes to $0$ as $n\rightarrow +\infty$. Indeed, by the relation \eqref{LG1Hs5} one has
 \begin{align}\label{ec22}
  \left\||x|^{-b}|v^k|^2|v^j|\right\|_{L^{\widetilde{a}'}_tL^{\widehat{r}'}_x} \leq& c \|v^k\|^{\theta}_{L^\infty_tH^1_x} \left\|\|v^k(t-t_n^k)\|^{2-\theta}_{L_x^{\widehat{r}}} \|v^j(t-t_n^j)\|_{L^{\widehat{r}}_x}\right\|_{L^{\widetilde{a}'}_t}.   
 \end{align}
  Fix $1\leq j\neq k\leq M$. Note that, $\|v^k\|_{H^1_x}<+\infty$ (see \eqref{CSWO1} - \eqref{CSWO2}) and by \eqref{CSCP} $v^j$, $v^k\in S(\dot{H^{s_c}})$ and , so we can approximate $v^j$ by functions of $C_0^\infty(\mathbb{R}^{3+1})$. Hence, defining
 $$
 g_n(t)= \|v^k(t)\|^{2-\theta}_{L_x^{\widehat{r}}} \|v^j(t-(t_n^j-t_n^k))\|_{L^{\widehat{r}}_x},
 $$
 we deduce
 \begin{itemize}
\item [(i)] $g_n\in L^{\widetilde{a}'}_t$. Indeed, applying the H\"older inequality since $\frac{1}{\widetilde{a}'}=\frac{\alpha-\theta}{\widehat{a}}+\frac{1}{\widehat{a}}$ we get 
$$
\|g_n\|_{L^{\widetilde{a}'}_t}\leq \|v^k\|^{2-\theta}_{L^{\widehat{a}}_t L_x^{\widehat{r}}} \|v^j\|_{L^{\widehat{a}}_tL^{\widehat{r}}_x}\leq \|v^k\|^{2-\theta}_{S(\dot{H}^{s_c})} \|v^j\|_{S(\dot{H}^{s_c})}<+\infty.
$$
Furthermore, \eqref{PD} implies that $g_n(t)\rightarrow 0$ as $n\rightarrow +\infty$.
\item [(ii)] $|g_n(t)|\leq  KI_{supp(v^j)}\|v^k(t)\|^{2-\theta}_{L_x^{\widehat{r}}}\equiv g(t)$ for all $n$, where $K>0$ and $I_{supp(v^j)}$ is the characteristic function of $supp(v^j)$. Similarly as (i), we obtain 
$$
\|g\|_{L^{\widetilde{a}'}_t}\leq \|v^k\|^{2-\theta}_{L^{\widehat{a}}_t L_x^{\widehat{r}}}\|I_{supp(v^j)}\|_{L^{\widehat{a}}_t L_x^{\widehat{r}}}<+\infty.
$$
That is, $g\in L^{\widetilde{a}'}_t$.
\end{itemize}
 Then, the Dominated Convergence Theorem yields $\|g_n\|_{L^{\widetilde{a}'}_t}\rightarrow 0$ as $n\rightarrow +\infty$, and so combining this result with \eqref{ec22} we conclude the proof of the first estimate. 
 
   \ Next, we prove $\|e_n^M\|_{S'(L^2)}<\frac{\varepsilon}{3}$. Using again the elementary inequality \eqref{EIerror} we estimate
 \begin{equation*}
   \|e_n^M\|_{S'(L^2)}\leq C_{\alpha,M}\sum_{j=1}^{M}\sum_{1\leq j\neq k\leq M}  \left\||x|^{-b}|v^k|^2|v^j|\right\|_{L^{\widehat{q}'}_tL^{\widehat{r}'}_x}.
  \end{equation*}   
On the other hand, we have (see proof of Lemma \ref{LG1} (ii)) 
 \begin{eqnarray*}
	\left\|  |x|^{-b}|v^k|^2|v^j \right \|_{L_t^{\widehat{q}'}L^{\widehat{r}'}_x}&\leq&  c \|v^k\|^{\theta}_{L^\infty_tH^1_x}\left \| \|v^k(t-t_n^k)\|^{2-\theta}_{ L_x^{\widehat{r}}} \|v^j(t-t_n^j)\|_{L^{\widehat{r}}_x}\right\|_{L_t^{\widehat{q}'}}\\
  &\leq & c\|v^k\|^{\theta}_{L^\infty_tH^1_x} \|v^k\|^{2-\theta}_{L^{\widehat{a}}_t L_x^{\widehat{r}}} \|v^j\|_{L_t^{\widehat{q}}L^{\widehat{r}}_x}\\
 &\leq & c\|v^k\|^{\theta}_{L^\infty_tH^1_x}  \|v^k\|^{2-\theta}_{S(\dot{H}^{s_c})}  \|v^j\|_{S(L^2)}.
\end{eqnarray*}
Since $v^j\in S(\dot{H}^{s_c})$ then by \eqref{SCATTER1} the norms $\|v^j\|_{S(L^2)}$ and $\|\nabla v^j\|_{S(L^2)}$ are bounded quantities. This implies that the right hand side of the last inequality is finite. Therefore, using the same argument as in the previous case we get
  $$
   \left \| \|v^k(t-t_n^k) \|^{2-\theta}_{ L_x^{\widehat{r}}} \|v^j(t-t_n^j) \|_{L^{\widehat{r}}_x}\right\|_{L_t^{\widehat{q}'}}\rightarrow 0, 
  $$
  as $n\rightarrow +\infty$, which lead to % completes the proof. % we obtain the desired result ($\|e_n^M\|_{S'(L^2)}<\frac{\varepsilon}{3}$).
$
\left\|  |x|^{-b}|v^k|^2|v^j \right \|_{L_t^{\widehat{q}'}L^{\widehat{r}'}_x}\rightarrow 0.
$
  
\ Finally, we prove $\|\nabla e_n^M\|_{S'(L^2)}<\frac{\varepsilon}{3}$. Note that
\begin{eqnarray}\label{ec23}
\nabla e_n^M&=&\nabla (|x|^{-b})\left( f(\widetilde{u}_n)-\sum_{j=1}^M f(v^j)  \right)+|x|^{-b}\nabla \left(  f(\widetilde{u}_n)-\sum_{j=1}^M f(v^j)\right)\nonumber \\
&\equiv& R^1_n+R^2_n,
\end{eqnarray}
where $f(v)=|v|^2 v$. First, we consider $R^1_n$. The estimate \eqref{EIerror} yields
$$
\| R^1_n\|_{S'(L^2)}\leq c\; C_{\alpha,M}\sum_{j=1}^{M}\sum_{1\leq j\neq k\leq M}  \left\||x|^{-b-1}|v^k|^2|v^j|\right\|_{L^{\widehat{q}'}_tL^{\widehat{r}'}_x}
$$
and by Remark \ref{RSglobal} we deduce that $\left\||x|^{-b-1}|v^k|^2|v^j|\right\|_{L^{\widehat{q}'}_tL^{\widehat{r}'}_x}$ is finite, then by the same argument as before we have 
$$
 \left\||x|^{-b-1}|v^k(t-t_n^k)|^2|v^j(t-t_n^j)|\right\|_{L^{\widehat{q}'}_tL^{\widehat{r}'}_x}\rightarrow 0\;\;\textnormal{as}\;\;n\rightarrow+\infty.
$$
Therefore, the last two relations yield $\|R^1_n\|_{S'(L^2)}\rightarrow 0$ as $n\rightarrow+\infty$.

\ On the other hand, observe that
\begin{eqnarray}\label{ec231}
\nabla (  f(\widetilde{u}_n)-\sum_{j=1}^M f(v^j))&=&f'(\widetilde{u}_n)\nabla \widetilde{u}_n-\sum_{j=1}^M f'(v^j)\nabla v^j\nonumber\\
&=&\sum_{j=1}^M (f'(\widetilde{u}_n)- f'(v^j))\nabla v^j.
\end{eqnarray}
Since $|f'(\widetilde{u}_n)- f'(v^j)|\leq C_{\alpha,M}\sum_{1\leq k\neq j\leq M}|v^k|(|v^j|+|v^k|)$, %(by Remark \ref{nonlinerity})
%$$
%|f'(\widetilde{u}_n)- f'(v^j)|\leq C_{\alpha,M}\sum_{1\leq k\neq j\leq M}|v^k|(|v^j|^{\alpha-1}+|v^k|^{\alpha-1})\;\;\;\textnormal{if}\;\;\;\alpha>1
%$$
we deduce using the last two relations together with \eqref{ec23} and \eqref{ec231} 
$$
\|R_n^2\|_{S'(L^2)}\lesssim \sum_{j=1}^M\sum_{1\leq k\neq j\leq M} \left\||x|^{-b}  |v^k|(|v^j|+|v^k|)|\nabla v^j|\right\|_{S'(L^2)}.
$$
Therefore, from Lemma \ref{LG1} (ii) (see also Remark \ref{RGP}) we have that the right hand side of the last two inequalities are finite quantities and, by an analogous argument as before, we conclude that 
$$
\|R_n^2\|_{S'(L^2)}\rightarrow 0\;\;\;\textnormal{as}\;\;\;n\rightarrow+\infty.
$$
This completes the proof of Claim $1$.
   
  \ {\bf Proof of Claim $2.$} First, we show that $\|\widetilde{u}_n\|_{L^\infty_tH^1_x}$ and $\|\widetilde{u}_n\|_{L^\gamma_tL^\gamma_x}$ are bounded quantities where $\gamma=\frac{10}{3}$. Indeed, we already know (see \eqref{PCS41} and \eqref{Sumpsij}) that there exists $C_0$ such that
$$
\sum_{j=1}^{\infty}\|\psi^j\|^2_{H^1_x}\leq C_0,
$$
then we can choose $M_0\in \mathbb{N}$ large enough such that
\begin{equation}\label{SP} %soma pequena
\sum_{j=M_0}^{\infty}\|\psi^j\|^2_{H^1_x}\leq \frac{\delta}{2},
\end{equation}
where $\delta>0$ is a sufficiently small.\\
Fix $M\geq M_0$. From \eqref{CSWO3}, there exists $n_1(M)\in \mathbb{N}$ where for all $n> n_1(M)$, we obtain
$$
  \sum_{j=M_0}^{M}\|\textnormal{INLS}(-t_n^j)\widetilde{\psi}^j\|^2_{H^1_x}\leq \delta,
$$    
where we have used \eqref{SP}. This is equivalent to 
\begin{equation}\label{claim11}
\sum_{j=M_0}^{M}\|v^j(-t_n^j)\|^2_{H^1_x}\leq \delta.
\end{equation}
Therefore, by the Small Data Theory (Proposition \ref{GWPH1})\footnote{Recall that the pair $(\infty,2)$ is $L^2$-admissible.}
$$
\sum_{j=M_0}^{M}\|v^j(t-t_n^j)\|^2_{L_t^{\infty}H^1_x}\leq c\delta\;\;\textnormal{for}\;n\geq n_1(M).
$$
Note that,  
$$
\left\|\sum_{j=M_0}^{M} v^j(t-t_n^j)\right\|^2_{H^1_x}=\sum_{j=M_0}^{M}\|v^j(t-t_n^j)\|^2_{H_x^1}+2\sum_{M_0\leq l \neq k\leq M}\langle v^l(t-t_n^l),v^k(t-t_n^k)\rangle_{H^1_x},
$$ 
so, for $l\neq k$ we deduce from $\eqref{PD}$ that (see \cite[Corollary $4.4$]{JIANCAZENAVE} for more details)
$$ \sup_{t\in\mathbb{R}} |\langle v^l(t-t_n^l),v^k(t-t_n^k)\rangle_{H^1_x}|\rightarrow 0\;\;\textnormal{as}\;\;n\rightarrow +\infty.
$$
Hence, since $\|v^j\|_{L^{\infty}_tH_x^1}$ is bounded (see \eqref{CSWO1} - \eqref{CSWO2}), by definition of $\widetilde{u}_n$ there exists $S>0$ (independent of $M$) such that
 \begin{equation}\label{claim12}
 \sup_{t\in\mathbb{R}}\|\widetilde{u}_n\|^2_{H^1_x}\leq S \;\,\textnormal{for}\;\;n>n_1(M).
\end{equation}
   
   \ We now show $\|\widetilde{u}_n\|_{L^\gamma_tL^\gamma_x}\leq L_1$. Using again \eqref{claim11} with $\delta$ small enough and the Small Data Theory (noting that $(\gamma,\gamma)$ is $L^2$-admissible and $\gamma >2$), we have 
\begin{equation}\label{claim111}
 \sum_{j=M_0}^{M}\|v^j(t-t_n^j)\|^{\gamma}_{L^\gamma_tL^\gamma_x}\leq c  \sum_{j=M_0}^{M}\|v^j(-t_n^j)\|^{\gamma}_{H^1_x}\leq c  \sum_{j=M_0}^{M}\|v^j(-t_n^j)\|^2_{H^1_x}\leq c\delta,
\end{equation}
for $n\geq n_1(M)$.

 On the other hand, in view of \eqref{FI} 
  $$
   \left\|\sum_{j=M_0}^{M} v^j(t-t_n^j)\right\|^\gamma_{L^\gamma_tL^\gamma_x}\leq \sum_{j=M_0}^{M}\|v^j\|^\gamma_{L^\gamma_tL^\gamma_x}+C_M\sum_{M_0\leq j \neq k\leq M}\int_{\mathbb{R}^{3+1}} |v^j||v^k||v^k|^{\gamma-2}
  $$ 
 for all $M>M_0$. Observe that, given $j$ such that $M_0\leq j\neq k\leq M$, the H\"older inequality yields 
 \begin{eqnarray}\label{claim112}
 \int_{\mathbb{R}^{3+1}} |v^j||v^k||v^k|^{\gamma-2}& \leq &\|v^k(t-t_n^k)\|_{L^\gamma_{t}L^\gamma_x}\left( \int_{\mathbb{R}^{3+1}} |v^j|^{\frac{\gamma}{2}}|v^k|^{\frac{\gamma}{2}} \right)^{\frac{2}{\gamma}} \nonumber \\
 &\leq& c \|v^j(-t_n^j)\|_{H^1_x}\left( \int_{\mathbb{R}^{3+1}} |v^j|^{\frac{\gamma}{2}}|v^k|^{\frac{\gamma}{2}} \right)^{\frac{2}{\gamma}}.
\end{eqnarray}
Since $v^j$ and $v^k\in L^\gamma_tL^\gamma_x$ we have that the right hand side of \eqref{claim112} is bounded and so by similar arguments as in the previous claim, we deduce from \eqref{PD} that the integral in the right hand side of the previous inequality goes to $0$ as $n\rightarrow +\infty$ (another proof of this fact can be found in \cite[Lemma $4.5$]{JIANCAZENAVE}). This implies that there exists $L_1$ (independent of $M$) such that
 \begin{equation}\label{claim113}
  \|\widetilde{u}_n\|_{L^\gamma_tL^\gamma_x} \leq \sum_{j=1}^{M_0}\|v^j\|_{L^\gamma_tL^\gamma_x}+\left\|\sum_{j=M_0}^{M} v^j\right\|_{L^\gamma_tL^\gamma_x}\leq L_1\;\;\;\textnormal{for}\;n\geq n_1(M),
\end{equation}
where we have used \eqref{claim111}.  

\ To complete the proof of the Claim $2$ we will show the following inequality % Next, we obtain by interpolation inequality 
\begin{equation}\label{ineq2}
\|\widetilde{u}_n\|_{L^{\widehat{a}}_tL^{\widehat{r}}_x}\leq \|\widetilde{u}_n\|^{1-\frac{\gamma}{\widehat{a}}}_{L^\infty_tH^1_x}\|\widetilde{u}_n\|^{\frac{\gamma}{\widehat{a}}}_{L^\gamma_tL^\gamma_x},
\end{equation}
where $\widehat{a}$ and $\widehat{r}$ are defined in \eqref{PA1}-\eqref{PA2}.

\ Assuming the last inequality for a moment let us conclude the proof of the Claim $2$. Indeed combining \eqref{claim12} and \eqref{claim113} we deduce from \eqref{ineq2} that
$$
\|\widetilde{u}_n\|_{L^{\widehat{a}}_tL^{\widehat{r}}_x}\leq S^{1-\frac{\gamma}{\widehat{a}}}L_1^{\frac{\gamma}{\widehat{a}}}=L_2,\;\;\;\textnormal{for}\;n\geq n_1(M).
$$
Then, since $\widetilde{u}_n$ satisfies the perturbed equation \eqref{widetildeun} we can apply the Strichartz estimates to the integral formulation and conclude (using also Claim $1$)
\begin{eqnarray*}
\|\widetilde{u}_n\|_{S(\dot{H}^{s_c})}&\leq &c\|\widetilde{u}_{n,0}\|_{H^1_x}+c\left\|  |x|^{-b}  |\widetilde{u}_n|^2 \widetilde{u}_n \right\|_{L^{\widetilde{a}'}_tL_x^{\bar{\widehat{r}}'}}+\|e^M_n\|_{S'(\dot{H}^{-s_c})}\\
&\leq & cS+cL_2+\varepsilon = L,
\end{eqnarray*}
for $n\geq n_1(M)$, which completes the proof of the Claim $2$.  

\ We now prove the inequality \eqref{ineq2}. Indeed, the interpolation inequality implies 
$$
\|\widetilde{u}_n\|_{L^{\widehat{a}}_tL^{\widehat{r}}_x}\leq \|\widetilde{u}_n\|^{1-\frac{\gamma}{\widehat{a}}}_{L^\infty_tL^p_x}\|\widetilde{u}_n\|^{\frac{\gamma}{\widehat{a}}}_{L^\gamma_tL^\gamma_x},
$$
where $\frac{1}{\widehat{r}}=\left(1-\frac{\gamma}{\widehat{a}}\right)\left(\frac{1}{p}\right)+\frac{1}{\widehat{a}}$. Using the values of $\widehat{r}$, $\widehat{a}$ and $\gamma$ we obtain 
$$p=\frac{14-6\theta+10b}{3+b-\theta(2-b)}.
$$
Choosing $0<\theta <2/3$ and $b<1$ then it is easy to see that $2<p<6$. Thus by the Sobolev embedding $H^1 \hookrightarrow L^p$ the inequality \eqref{ineq2} holds. 
\end{proof}
\end{proposition}

\ In the next proposition, we prove the precompactness of the flow associated to the critical solution $u_c$. The argument is very similar to Holmer-Roudenko \cite[Proposition $5.5$]{HOLROU}. %For the sake of completeness are provide the details.

\begin{proposition}\label{PSC}{\bf (Precompactness of the flow of the critical solution)} Let $u_c$ be as in Proposition \ref{ECS} and define
 $$
 K=\{u_c(t)\;:\;t\in[0,+\infty)\}\subset  H^1.
 $$
 Then $K$ is precompact in $H^1(\mathbb{R}^3)$.
 \begin{proof} 
Let $\{t_n\}\subseteq [0,+\infty )$ a sequence of times and $\phi_n=u_c(t_n)$ be a uniformly bounded sequence in $H^1(\mathbb{R}^3)$. We need to show that $u_c(t_n)$ has a subsequence converging in $H^1(\mathbb{R}^3)$. If $\{t_n\}$ is bounded, we can assume $t_n \rightarrow t^*$ finite, so by the continuity of the solution in $H^1(\mathbb{R}^3)$ the result is clear. Next, assume that $t_n\rightarrow +\infty$. 
	
 \ The linear profile expansion (Proposition \ref{LPD}) implies the existence of profiles $\psi^j$ and a remainder $W_n^M$ such that
  $$
  u_c(t_n)=\sum_{j=1}^{M}U(-t_n^j)\psi^j+W_n^M,
  $$	
  with $|t_n^j-t_n^k|\rightarrow +\infty$ as $n\rightarrow +\infty$ for any $j\neq k$. Then, by the energy Pythagorean expansion (Proposition \ref{EPE}), we get
  \begin{equation}\label{ECCS}
  \sum_{j=1}^{M}\lim_{n\rightarrow +\infty}E[U(-t_n^j)\psi^j]+\lim_{n\rightarrow +\infty} E[W_n^M]=E[u_c]=\delta_c,	
  \end{equation}
 where we have used Proposition \ref{ECS} (ii). This implies that
 \begin{equation*}
 \lim_{n\rightarrow +\infty}E[U(-t_n^j)\psi^j]\leq \delta_c\;\;\;\;\forall\;j, 
 \end{equation*}
 since each energy in \eqref{ECCS} is nonnegative by Lemma \eqref{LGS} (i).	\\
 Moreover, by \eqref{PDNHs} with $s=0$ we obtain
 \begin{equation}\label{MCCS} %massa compa crit sol
 \sum_{j=1}^{M}M[\psi^j]+\lim_{n\rightarrow +\infty}M[W_n^M]=M[u_c]=1,
 \end{equation}
 by Proposition \ref{ECS} (i).
 	
  \ If more than one $\psi^j \neq 0$, similar to the proof in Proposition \ref{ECS}, we have a contradiction
  with the fact that $\|u_c\|_{S(\dot{H}^{s_c})}=+\infty$. Thus, we address the case that only $\psi^j=0$ for all $j\geq 2$, and so
  \begin{equation}\label{CCS1}%compa da solo criti 1
   u_c(t_n)=U(-t_n^1)\psi^1+W_n^M.
  \end{equation}
 Also as in the proof of Proposition \ref{ECS}, we obtain that
 \begin{equation}\label{MECS}
 M[\psi^1] =M[u_c]=1\;\;\;\textnormal{and}\;\;\;\lim_{n\rightarrow +\infty}E[U(-t_n^1)\psi^1]=\delta_c,	
 \end{equation}
 and using \eqref{ECCS}, \eqref{MCCS} together with \eqref{MECS}, we deduce that 
 \begin{equation}
 \lim_{n\rightarrow +\infty}M[W_n^M]=0\;\;\;\textnormal{and}\;\;\;\lim_{n\rightarrow +\infty}E[W_n^M]=0.
 \end{equation}	
  Thus, Lemma \ref{LGS} (i) yields
  \begin{equation}\label{ERCS}
  \lim_{n\rightarrow +\infty}\|W_n^M\|_{H^1}=0.
  \end{equation} 
	
  \ We claim now that $t^1_n$ converges to some finite $t^*$ (up to a subsequence). In this case, since $U(-t_n^1)\psi^1\rightarrow U(-t^*)\psi^1$ in $H^1(\mathbb{R}^3)$ and $\eqref{ERCS}$ holds, the relation \eqref{CCS1} implies that $u_c(t_n)$ converges in $H^1(\mathbb{R}^3)$, concluding the
  proof. 
  
  \ Assume by contradiction that $|t^1_n|\rightarrow +\infty$, then we have two cases to consider. If $t^1_n\rightarrow -\infty$, by \eqref{CCS1}  
  $$
  \|U(t)u_c(t_n)\|_{S(\dot{H}^{s_c};[0,+\infty))}\leq\|U(t-t_n^1)\psi^1\|_{S(\dot{H}^{s_c};[0,+\infty))}+\|U(t)W_n^M\|_{S(\dot{H}^{s_c};[0,+\infty))}.
  $$	
 Next, note that since $t^1_n\rightarrow -\infty$ we obtain
 \begin{equation*}\label{PPCS1}
  \|U(t-t_n^1)\psi^1\|_{S(\dot{H}^{s_c};[0,+\infty))}\leq \|U(t)\psi^1\|_{S(\dot{H}^{s_c};[-t_n^j,+\infty))}\leq \frac{1}{2}\delta,
 \end{equation*}
 and also 
 \begin{equation*}\label{PPCS2}
 \|U(t)W_n^M\|_{S(\dot{H}^{s_c})}\leq \frac{1}{2}\delta,
 \end{equation*}
given $\delta>0$ for $n, M$ sufficiently large, where in the last inequality we have used \eqref{SE2} and \eqref{ERCS}. Hence,   
 $$
 \|U(t)u_c(t_n)\|_{S(\dot{H}^{s_c};[0,+\infty))}\leq \delta.
 $$ 
 Therefore, choosing $\delta>0$ sufficiently small, by the small data theory (Proposition \ref{GWPH1}) we get that $ \|u_c\|_{S(\dot{H}^{s_c})}\leq 2\delta,$
% $$ \|u_c\|_{S(\dot{H}^{s_c})}\leq 2\delta,$$
 which is a contradiction with Proposition \ref{ECS} (iv).
 
 \ On the other hand, if $t^1_n\rightarrow +\infty$, the same arguments also give that for $n$ large,    
 $$
 \|U(t)u_c(t_n)\|_{S(\dot{H}^{s_c};(-\infty,0])}\leq  \delta,
 $$
 and again the small data theory (Proposition \ref{GWPH1}) implies $
 \|u_c\|_{S(\dot{H}^{s_c};(-\infty,t_n])}\leq 2\delta.
 $
 % $$
 %\|u_c\|_{S(\dot{H}^{s_c};(-\infty,t_n])}\leq 2\delta.
 %$$
 Since $t_n\rightarrow +\infty$ as $n\rightarrow +\infty$, from the last inequality we get  $\|u_c\|_{S(\dot{H}^{s_c})}\leq 2\delta$, which is also a contradiction. Thus, $t_n^1$ must converge to some finite $t^*$, completing the proof of Proposition \ref{PSC}.
	
  	\end{proof}

\end{proposition}

%%%%%%%%%%%%%%%%%%%%%%%%%%%%%%%%%%%%%%%%%%%%%%%%%%%%%%%%%%%%%%%%%%%%%%%%%%%%%%%%%%%%%%%%%%%%%%%%%%%

%%%%%%%%%%%%%%%%%%%%%%%%%%%%%%%%%%%%%%%%%%%%%%%%%%%%%%%%%%%%%%%%%%%%%%%%%%%%%%%%%%%%%%%%%%%%%%%%%%%
\section{Rigidity theorem}

\ The main result of this section is a rigidity theorem, which implies that the critical solution $u_c$ constructed in Section \ref{CCS} must be identically zero and so reaching a contradiction in view of Proposition \ref{ECS} (iv). Before proving this result, we begin showing some preliminaries that will help us in the proof. 

\begin{proposition}\label{PFIUL}{\bf (Precompactness of the flow implies uniform localization)} Let $u$ be a solution of \eqref{INLS} such that
$$
K=\{u(t)\;:\;t\in[0,+\infty)\}
$$
is precompact in $H^1(\mathbb{R}^3)$. Then for each $\varepsilon>0$, there exists $R>0$ so that
\begin{equation}\label{PFIUL1}
\int_{|x|>R}|\nabla u(t,x)|^2dx\leq \varepsilon,\;\textnormal{for all}\;0\leq t<+\infty.
	\end{equation}
 \begin{proof} The proof follows the same steps as in Holmer-Roudenko \cite[Lemma $5.6$]{HOLROU}. So we omit the details 
 %The proof is similar to that in Holmer-Roudenko \cite[Lemma $5.6$]{HOLROU}. If \eqref{PFIUL1} does not hold, then there exists $\varepsilon>0$ and a sequence $t_n\rightarrow +\infty$ such that, for each $n\in \mathbb{N}$,
 %\begin{equation}\label{PC1}%Precomp 1
 %\int_{|x|>n}|\nabla u(t,x)|^2dx\geq 2\varepsilon.
%\end{equation}
% The fact that $K$ is precompact yields that there exists some $\phi\in H^1$ such that, up to a subsequence of $t_n$, $u(t_n)\rightarrow \phi$ in $H^1$, which implies
% \begin{equation}\label{PC2}
%  \int |\nabla(u(t_n)-\phi)|^2dx<\frac{1}{4}\varepsilon.
%\end{equation}
% On the other hand, since $\phi\in H^1$, taking $n$ sufficiently large we can get
% \begin{equation}\label{PC3}
% \int_{|x|>n}|\nabla \phi|^2dx\leq \frac{1}{4}\varepsilon.
%\end{equation}
% Thus, \eqref{PC2} and \eqref{PC3} lead to
% $$
% \int_{|x|>n}|\nabla u(t,x)|^2dx\leq  2\int |\nabla(u(t_n)-\phi)|^2dx+ 2\int_{|x|>n}|\nabla \phi|^2dx<\varepsilon,
 %$$
 %which is a contradiction with \eqref{PC1}.
		\end{proof}
\end{proposition}

We will also need the following local virial identity. %VI=virial identity

\begin{proposition}\label{VI}  {\bf (Virial identity)} Let $\phi\in C^\infty_0(\mathbb{R}^3)$, $\phi \geq 0$ and $T>0$. For $R>0$ and $t\in [0,T]$ define
\begin{equation*}\label{DFZ}% DFZ=definition function z
z_R(t)=\int_{\mathbb{R}^3} R^2 \phi\left(\frac{x}{R}\right)|u(t,x)|^2dx,
\end{equation*}  
where $u$ is a solution of \eqref{INLS}. Then we have
\begin{equation}\label{FD}%First derivate
z'_R(t)=2R Im\int_{\mathbb{R}^3} \nabla\phi\left(\frac{x}{R}\right)\cdot\nabla u \bar{u}dx
\end{equation}
and 
\begin{align}\label{SD}
z''_R(t)&= 4\sum_{j,k} Re\int \frac{\partial u}{\partial_{x_k}}  \frac{\partial \bar{u}}{\partial_{x_j}}\frac{\partial^2\phi}{\partial x_k\partial x_j}\left(\frac{x}{R}\right)dx-\frac{1}{R^2}\int |u|^2\Delta^2 \phi\left(\frac{x}{R}\right) dx \nonumber \\
	& -\int |x|^{-b}|u|^{4}\Delta \phi\left(\frac{x}{R}\right) dx+R\int \nabla (|x|^{-b})\cdot\nabla \phi\left(\frac{x}{R}\right) |u|^{4}dx.
 \end{align}

 \begin{proof} 
 We first compute $z'_R$. Note that
 $$
 \partial_t|u|^2=2Re( u_t\bar{u})=2Im(iu_t\bar{u}). 
 $$
Since $u$ satisfies \eqref{INLS} and using integration by parts, we have
\begin{eqnarray*}
z'_R(t)&=&2Im\int  R^2\phi\left(\frac{x}{R}\right)iu_t \bar{u}dx\\
&=&-2Im\int R^2\phi\left(\frac{x}{R}\right)\left(\Delta u\bar{u}+|x|^{-b}|u|^{4}\right)dx\\
&=& -2Im\int R^2\phi\left(\frac{x}{R}\right)\nabla \cdot (\nabla u \bar{u})dx\\
&=&2RIm\int \nabla \phi\left(\frac{x}{R}\right)\cdot\nabla u \bar{u}dx.
\end{eqnarray*}
On the other hand, using again integration by parts and the fact that $z-\bar{z}=2iIm z$, we obtain
\begin{eqnarray*}
z''_R(t)&=& 2RIm\int \nabla \phi\left(\frac{x}{R}\right)\cdot \left(  \bar{u}_t\nabla u +\bar{u}\nabla  u_t\right)dx\\
	&=& 2RIm\left\{\sum_{j}\int  \bar{u}_t\partial_{x_j}u\partial_{x_j}\phi\left(\frac{x}{R}\right)dx- u_t\partial_{x_j}\left(\bar{u}\partial_{x_j}\phi\left(\frac{x}{R}\right)\right)  dx \right\}\\
	&=& 2RIm\left\{\sum_{j}2i Im \int \bar{u}_t\partial_{x_j}u\partial_{x_j}\phi\left(\frac{x}{R}\right) dx-\int\frac{1}{R}u_t \bar{u}\partial^2_{x_j}\phi\left(\frac{x}{R}\right) dx \right\}\\
	&=&4R I_1+2I_2,
\end{eqnarray*}
where
$$
I_1=Im \sum_{j}\int \bar{u}_t\partial_{x_j}u\partial_{x_j}\phi\left(\frac{x}{R}\right)\;\;\textnormal{and}\;\;I_2=-Im \sum_{j}\int u_t\bar{u} \partial^2_{x_j}\phi\left(\frac{x}{R}\right) dx.
$$
We start considering $I_2$. Since $u$ is a solution of \eqref{INLS} we get
\begin{eqnarray*}
I_2&=&-Im\left\{\sum_{j,k}  \int i \partial^2_{x_k}u \bar{u}\partial^2_{x_j}  \phi\left(\frac{x}{R}\right) dx  \right\}-\sum_{j}\int |x|^{-b}|u|^{4}\partial^2_{x_j}\phi\left(\frac{x}{R}\right) dx\\
&= & Im \left\{\sum_{j,k}  \int i \left( |\partial_{x_k}u|^2 \partial^2_{x_j}  \phi\left(\frac{x}{R}\right)+ \frac{1}{R}\partial_{x_k}u\bar{u}\frac{\partial^3\phi}{\partial x_k \partial x^2_j}\left(\frac{x}{R}\right) \right)dx\right\} \\
&  &  - \int |x|^{-b}|u|^{4}\Delta \phi\left(\frac{x}{R}\right) dx\\
&=&\int \left(|\nabla u|^2-|x|^{-b}|u|^{4}\right) \Delta \phi\left(\frac{x}{R}\right) dx+\frac{1}{R} \sum_{j,k} Re\int \partial_{x_k}u\bar{u} \frac{\partial^3\phi}{\partial x_k \partial x^2_j}\left(\frac{x}{R}\right)dx, 
\end{eqnarray*}
where we have used integration by parts and the fact that $Im (iz)=Re (z)$. Furthermore, since $\partial_{x_k}|u|^2=2 Re \left(\partial_{x_k}u\bar{u}\right)$ another integration by parts yields
\begin{align}\label{VI1} %VI1=virial identity 1
I_2=&\int \left(|\nabla u|^2-|x|^{-b}|u|^{4}\right) \Delta \phi\left(\frac{x}{R}\right) dx-\frac{1}{2R^2}\sum_{j,k}\int |u|^2\frac{\partial^4\phi}{\partial x^2_k\partial x^2_j}\left(\frac{x}{R}\right)dx \nonumber \\
 =&\int \left(|\nabla u|^2-|x|^{-b}|u|^{4}\right) \Delta \phi\left(\frac{x}{R}\right) dx-\frac{1}{2R^2}\int |u|^2\Delta^2\phi\left(\frac{x}{R}\right) dx.
\end{align}

\indent Next, we deduce using the equation \eqref{INLS} and $Im (z)=-Im (\bar{z})$ that

\begin{align*}
I_1&=-Im \sum_{j}u_t \partial_{x_j}\bar{u}\partial_{x_j} \phi\left(\frac{x}{R}\right) dx\\
&=-Im i\sum_{j}\left\{\int \left(    \Delta u + |x|^{-b}|u|^{2}u \right)\partial_{x_j}\bar{u}\partial_{x_j} \phi\left(\frac{x}{R}\right) dx   \right\}\\
&= -Re \sum_{j,k}\int \partial^2_{x_k} u\partial_{x_j} \bar{u} \partial_{x_j} \phi\left(\frac{x}{R}\right) dx-\sum_j \int |x|^{-b}\partial_{x_j} \phi\left(\frac{x}{R}\right)|u|^{2}  Re (\partial_{x_j}\bar{u} u) dx\\
&=-Re \sum_{j,k}\int \partial^2_{x_k} u\partial_{x_j} \bar{u} \partial_{x_j} \phi\left(\frac{x}{R}\right)dx-\frac{1}{4}\sum_j \int |x|^{-b}\partial_{x_j}\phi\left(\frac{x}{R}\right)\partial_{x_j}(|u|^{4}) dx\\
&\equiv A+B,
\end{align*}
where we have used $Im (iz)=Re (z)$ and $\partial_{x_j}(|u|^{4})=4|u|^2 Re (\partial_{x_j}\bar{u} u) $. Moreover, since $\partial_{x_j}|\partial_{x_k}u|^2=2Re \left(\partial_{x_k}u \frac{\partial^2\bar{u}}{ \partial x_k \partial x_j}\right) $ and using integration by parts twice, we get

\begin{eqnarray*}
A&=& Re \sum_{j,k} \left\{  \int \left(\partial_{x_j}\phi\left(\frac{x}{R}\right) \partial_{x_k}u \frac{\partial^2\bar{u}}{ \partial x_k \partial x_j}+ \frac{1}{R}  \partial_{x_k}u    \partial_{x_j}\bar{u}  \frac{  \partial^2\phi}{\partial x_j\partial x_k}\left(\frac{x}{R}\right)  \right)dx \right\}\\
&=&-\sum_{j,k}\frac{1}{2R} \int |\partial_{x_k}u|^2\partial^2_{x_j}\phi\left(\frac{x}{R}\right) dx+\frac{1}{R} \sum_{i,j} Re \int \partial_{x_k}u    \partial_{x_j}\bar{u}  \frac{  \partial^2\phi}{\partial x_j\partial x_k}\left(\frac{x}{R}\right) dx \\
&=&-\frac{1}{2R} \int |\nabla u|^2 \Delta \phi\left(\frac{x}{R}\right) dx+\frac{1}{R} \sum_{i,j} Re \int \partial_{x_k}u    \partial_{x_j}\bar{u}  \frac{  \partial^2\phi}{\partial x_j\partial x_k}\left(\frac{x}{R}\right) dx.
\end{eqnarray*}
Similarly, integrating by parts
\begin{align*}
B&=\frac{1}{4} \sum_j \left( \int \partial_{x_j}\phi\left(\frac{x}{R}\right)\partial_{x_j}(|x|^{-b}) |u|^{4} dx+ \frac{1}{R} \int \partial^2_{x_j}\phi\left(\frac{x}{R}\right)|x|^{-b} |u|^{4} dx \right) \\
&=\frac{1}{4}  \int \nabla \phi\left(\frac{x}{R}\right)\cdot \nabla (|x|^{-b}) |u|^{4} dx+ \frac{1}{4R}  \int \Delta \phi\left(\frac{x}{R}\right)|x|^{-b} |u|^{4} dx. \\
\end{align*}
Therefore,
\begin{align}\label{VI2}
I_1&=-\frac{1}{2R} \int |\nabla u|^2 \Delta \phi\left(\frac{x}{R}\right) dx+\frac{1}{R} \sum_{i,j}Re  \int \partial_{x_k}u    \partial_{x_j}\bar{u}  \frac{  \partial^2\phi}{\partial x_j\partial x_k}\left(\frac{x}{R}\right) dx\nonumber \\
& +\frac{1}{4}  \int \nabla \phi\left(\frac{x}{R}\right)\cdot \nabla (|x|^{-b}) |u|^{4} dx+ \frac{1}{4R}  \int \Delta \phi\left(\frac{x}{R}\right)|x|^{-b} |u|^{4} dx.
\end{align}
Finally it is easy to check that combining \eqref{VI1} and \eqref{VI2} we obatin \eqref{SD}, which complete the proof. 
\end{proof} 

\end{proposition}

\indent Finally, we apply the previous results to proof the rigidity theorem.
\begin{theorem}\label{RT}{\bf (Rigidity)} Let $u_0\in H^1(\mathbb{R}^3)$ satisfying
$$
E[u_0]^{s_c}M[u_0]^{1-s_c} <E[Q]^{s_c}M[Q]^{1-s_c}
$$
and 
$$
\|  \nabla u_{0} \|_{L^2}^{s_c} \|u_{0}\|_{L^2}^{1-s_c}<\|\nabla Q \|_{L^2}^{s_c} \|Q\|_{L^2}^{1-s_c}.
$$
If the global $H^1(\mathbb{R}^3)$-solution $u$ with initial data $u_0$ satisfies
	$$
	K=\{u(t)\;:\;t\in[0,+\infty)\}\; \textnormal{is precompact in}\; H^1(\mathbb{R}^3)
$$
then $u_0$ must vanishes, i.e., $u_0=0$.
	
\begin{proof} By Theorem \ref{TG} we have that $u$ is global in $H^1(\mathbb{R}^3)$ and
	
	\begin{equation}\label{TR1} %TR1 teorema rigidez 1
	\|  \nabla u(t) \|_{L^2_x}^{s_c} \|u(t)\|_{L^2_x}^{1-s_c}<\|\nabla Q \|_{L^2}^{s_c} \|Q\|_{L^2}^{1-s_c}.
	\end{equation}
	
	\ On the other hand, let $\phi \in C_0^\infty$ be radial, with
	$$
	\phi(x)=\left\{\begin{array}{cl}
	|x|^2&\textnormal{for}\;|x|\leq 1\\
	0&\textnormal{for}\;|x|\geq 2.
	\end{array}\right.
$$
Then, using \eqref{FD}, the H\"older inequality and \eqref{TR1} we obtain
\begin{eqnarray*}
|z'_R(t)| &\leq & cR\int_{|x|<2R}|\nabla u(t)||u(t)|dx\leq cR\|\nabla u(t)\|_{L^2}\|u(t)\|_{L^2}\lesssim cR.
\end{eqnarray*}
 Hence,
\begin{eqnarray}\label{FDI}%FDI=first deriv ineq
|z'_R(t)-z'_R(0)|\leq |z'_R(t)|+|z'_R(0)|\leq 2cR,\;\;\textnormal{for all }\;t>0. 
%&\leq& cR,\;\;\textnormal{for all }\;t>0.
\end{eqnarray}
  
\ The idea now is to obtain a lower bound for $z''_R(t)$ strictly greater than zero and reach a contradiction. Indeed, from the local virial identity \eqref{SD}
  \begin{align}\label{SDz} %second deriv of Z
	z''_R(t)&= 4\sum_{j,k} Re\int \partial_{x_k} u \partial_{x_j}\bar{u}\frac{\partial^2\phi}{\partial x_k\partial x_j}\left(\frac{x}{R}\right)dx-\frac{1}{R^2}\int |u|^2\Delta^2 \phi\left(\frac{x}{R}\right) dx   \nonumber   \\
    &  -\int |x|^{-b}|u|^{4}\Delta \phi\left(\frac{x}{R}\right) dx+R\int \nabla (|x|^{-b})\cdot\nabla \phi\left(\frac{x}{R}\right) |u|^{4}dx \nonumber  \\
    & = 8 \| \nabla u \|^2_{L^2_x}- 2(3+b) \left\||x|^{-b}|u|^{4}\right\|_{L^1_x}+R(u(t)),
   \end{align}
  where
  \begin{align*}
  R(u(t))&= 4 \sum\limits_{j}Re\int \left( \partial^2_{x_j} \phi\left(\frac{x}{R}\right)-2\right )|\partial_{x_j}u|^2
  + 4 \sum\limits_{j\neq k}Re\int \frac{\partial^2\phi}{\partial x_k\partial x_j}\left(\frac{x}{R}\right)\partial_{x_k}u\partial_{x_j}\bar{u}\\
  &- \frac{1}{R^2}\int |u|^2\Delta^2\phi\left(\frac{x}{R}\right)+R\int \nabla(|x|^{-b})\cdot\nabla \phi\left(\frac{x}{R}\right)|u|^{4}\\
  &+\int \left( -\left( \Delta\phi\left(\frac{x}{R}\right)-6 \right)+2b\right) |x|^{-b}|u|^{4}.
  \end{align*}
 Since $\phi(x)$ is radial and $\phi(x)=|x|^2$ if $|x|\leq 1$, the sum of all terms in the definition of $R(u(t))$ integrating over $|x|\leq R$ is zero. Indeed, for the first three terms this is clear by the definition of $\phi(x)$. In the fourth term we have 
 $$
 2\int_{|x|\leq R} \nabla(|x|^{-b})\cdot x|u|^4dx=2\int_{|x|\leq R} -b|x|^{-b}|u|^4dx,
 $$ 
 and adding the last term (also integrating over $|x|\leq R$) we get zero since $\Delta \phi =6$, if $|x|\leq R$. Therefore, for the integration on the region $|x|> R$, we have the following bound
\begin{eqnarray*}
|R(u(t))|&\leq& c\int_{|x|>R} \left(|\nabla u(t)|^2+\frac{1}{R^2}|u(t)|^2+|x|^{-b}|u(t)|^{4}\right)dx
\end{eqnarray*}
\begin{eqnarray}\label{RESTO}
&\leq & c\int_{|x|>R} \left(|\nabla u(t)|^2+\frac{1}{R^2}|u(t)|^2+\frac{1}{R^b}|u(t)|^{4}\right)dx,
\end{eqnarray}
where we have used that all derivatives of $\phi$ are bounded and  $|R\partial_{x_j}(|x|^{-b})|\leq c|x|^{-b}$ if $|x|>R$. 
 
 \ Next we use that $K$ is precompact in $H^1(\mathbb{R}^3)$. By Proposition \ref{PFIUL}, given $\varepsilon>0$ there exists $R_1>0$ such that $\int_{|x|>R_1} |\nabla u(t)|^2\leq\varepsilon$. Furthermore, by Mass conservation \eqref{mass}, there exists $R_2>0$ such that $\frac{1}{R^2_2}\int_{|x|>R_2} | u(t)|^2\leq \varepsilon$. Finally, by the Sobolev embedding $H^1\hookrightarrow L^{4}$, there exists $R_3$ such that $\frac{1}{R_3^b}\int_{|x|>R_3} |u(t)|^{4}\leq c\varepsilon$ (recall that $\|u(t)\|_{H^1_x}$ is uniformly bounded for all $t>0$ by \eqref{TR1} and Mass conservation \eqref{mass}). Taking $R=\max\{R_1,R_2,R_3\}$ the inequality \eqref{RESTO} implies
 \begin{equation}\label{RESTO1}
 |R(u(t))|\leq c\int_{|x|>R} \left(|\nabla u(t)|^2+\frac{1}{R^2}|u(t)|^2+\frac{1}{R^b}|u(t)|^{4}\right)dx\leq c\varepsilon.
 \end{equation} 
On the other hand, Lemma \ref{LGS} (iii), \eqref{SDz} and \eqref{RESTO1} yield
\begin{equation*}\label{RESTO2}
z''_R(t)\geq 16A E[u]-|R(u(t))|\geq 16AE[u]-c\varepsilon,
\end{equation*} 
where $A=1-w$ and $w=\frac{E[v]^{s_c}M[v]^{1-s_c}}{E[Q]^{s_c}M[Q]^{1-s_c}}$. \\
Now, choosing $\varepsilon=\frac{8A}{c}E[u]$, with $c$ as in \eqref{RESTO1} we have
$$
z''_R(t)\geq 8AE[u].
$$ 
Thus, integrating the last inequality from $0$ to $t$ we deduce 
\begin{equation}\label{SDI}
z'_R(t)-z'_R(0)\geq 8AE[u]t.
\end{equation}
Now sending $t\rightarrow \infty$ the left hand of \eqref{SDI} also goes to $+\infty$, however from \eqref{FDI} it must be bounded. Therefore, we have a contradiction unless $E[u]=0$ which implies $u\equiv 0$ by Lemma \ref{LGS} (i). 
\end{proof}
\end{theorem}

\vspace{0.5cm}

\centerline{\textbf{Acknowledgment}}

\ 

L.G.F. was partially supported by CNPq and FAPEMIG/Brazil. C.G. was partially supported by Capes/Brazil. The authors also thank Svetlana Roudenko for fruitful conversations concerning this work.\\

\bibliographystyle{abbrv}
\bibliography{bibguzman}

\end{document}